\documentclass[reqno]{amsart}

\usepackage{graphicx}
\usepackage{amssymb}
\usepackage{amsmath}
\usepackage{enumerate}
\usepackage{amsfonts}
\usepackage[mathcal]{euscript}
\usepackage{amsthm}
\usepackage[all,2cell,import]{xy} \UseAllTwocells
\usepackage{xcolor}

\usepackage{graphicx}

\usepackage{mathrsfs}
\usepackage{epsfig}
\usepackage{amscd}

\def\E{\ifmmode{\mathbb E}\else{$\mathbb E$}\fi} 
\def\N{\ifmmode{\mathbb N}\else{$\mathbb N$}\fi} 
\def\R{\ifmmode{\mathbb R}\else{$\mathbb R$}\fi} 
\def\Q{\ifmmode{\mathbb Q}\else{$\mathbb Q$}\fi} 
\def\C{\ifmmode{\mathbb C}\else{$\mathbb C$}\fi} 
\def\H{\ifmmode{\mathbb H}\else{$\mathbb H$}\fi} 
\def\Z{\ifmmode{\mathbb Z}\else{$\mathbb Z$}\fi} 
\def\P{\ifmmode{\mathbb P}\else{$\mathbb P$}\fi} 
\def\T{\ifmmode{\mathbb T}\else{$\mathbb T$}\fi} 
\def\SS{\ifmmode{\mathbb S}\else{$\mathbb S$}\fi} 
\def\DD{\ifmmode{\mathbb D}\else{$\mathbb D$}\fi} 

\renewcommand{\d}{\delta}

\newcommand{\del}{\partial}

\newcommand{\Hom}{{\operatorname{Hom}}}

\newcommand{\ben}{\begin{enumerate}}
\newcommand{\een}{\end{enumerate}}
\newcommand{\be}{\begin{equation}}
\newcommand{\ee}{\end{equation}}
\newcommand{\bea}{\begin{eqnarray}}
\newcommand{\eea}{\end{eqnarray}}
\newcommand{\beastar}{\begin{eqnarray*}}
\newcommand{\eeastar}{\end{eqnarray*}}
\newcommand{\bc}{\begin{center}}
\newcommand{\ec}{\end{center}}

\theoremstyle{theorem}
\newtheorem{thm}{Theorem}[section]
\newtheorem{cor}[thm]{Corollary}
\newtheorem{lem}[thm]{Lemma}
\newtheorem{prop}[thm]{Proposition}

\theoremstyle{definition}
\newtheorem{defn}[thm]{Definition}
\newtheorem{rem}[thm]{Remark}
\newtheorem{ques}[thm]{Question}

\newtheorem{hypo}[thm]{Hypothesis}
\newtheorem{cond}[thm]{Condition}
\newtheorem{choice}[thm]{Choice}

\newtheorem{notation}[thm]{\rm\bfseries{Notation}}

\newtheorem*{thm*}{Theorem}

\numberwithin{equation}{section}

\def\R{{\mathbb R}}

\def\E{{\mathbb E}}
\def\Z{{\mathbb Z}}
\def\C{{\mathbb C}}
\def\R{{\mathbb R}}
\def\P{{\mathbb P}}

\def\N{{\mathbb N}}

\def\11{{\mathbb I}}

\def\delbar{{\overline \partial}}

\def\d_Hudtau{{\frac{\del w}{\del \tau}}}
\def\d_Hudt{{\frac{\del w}{\del t}}}

\def\C{\mathbb{C}}
\def\Z{\mathbb{Z}}

\def\T{\mathbb{T}}

\def\Q{\mathbb{Q}}

\def\E{\ifmmode{\mathbb E}\else{$\mathbb E$}\fi} 
\def\N{\ifmmode{\mathbb N}\else{$\mathbb N$}\fi} 
\def\R{\ifmmode{\mathbb R}\else{$\mathbb R$}\fi} 
\def\Q{\ifmmode{\mathbb Q}\else{$\mathbb Q$}\fi} 
\def\C{\ifmmode{\mathbb C}\else{$\mathbb C$}\fi} 
\def\H{\ifmmode{\mathbb H}\else{$\mathbb H$}\fi} 
\def\Z{\ifmmode{\mathbb Z}\else{$\mathbb Z$}\fi} 
\def\P{\ifmmode{\mathbb P}\else{$\mathbb P$}\fi} 
\def\SS{\ifmmode{\mathbb S}\else{$\mathbb S$}\fi} 
\def\DD{\ifmmode{\mathbb D}\else{$\mathbb D$}\fi} 

\def\R{{\mathbb R}}

\def\E{{\mathbb E}}
\def\Z{{\mathbb Z}}
\def\C{{\mathbb C}}
\def\R{{\mathbb R}}

\def\N{{\mathbb N}}

\def\delbar{{\overline \partial}}

\def\d{\delta}

\def\m{\mu}

\def\p{\psi}
\def\CA{{\mathcal A}}
\def\CB{{\mathcal B}}

\def\CF{{\mathcal F}}

\def\CH{{\mathcal H}}

\def\CJ{{\mathcal J}}

\def\CL{{\mathcal L}}

\def\CP{{\mathcal P}}

\def\CP{{\mathcal P}}

\def\darr#1{\raise1.5ex\hbox{$\leftrightarrow$}
\mkern-16.5mu #1}

\def\roughly#1{\raise.3ex\hbox{$#1$\kern-.75em
\lower1ex\hbox{$\sim$}}}

\def\opname#1{\mathop{\kern0pt{\rm #1}}\nolimits}

\def\Dev{\operatorname{Dev}}

\def\Image{\operatorname{Image}}

\def\Int{\operatorname{Int}}

\def\ric{\operatorname{Ric}}

\begin{document}
\quad \vskip1.375truein

\title[Perturbed Legendrian contact instantons]{Geometric analysis of
perturbed contact instantons with Legendrian boundary conditions}

\author{Yong-Geun Oh}
\address{Center for Geometry and Physics, Institute for Basic Science (IBS),
77 Cheongam-ro, Nam-gu, Pohang-si, Gyeongsangbuk-do, Korea 790-784
\& POSTECH, Gyeongsangbuk-do, Korea}
\email{yongoh1@postech.ac.kr}

\begin{abstract} In the present article, we provide analytic foundation of the
following nonlinear elliptic system, called the \emph{Hamiltonian-perturbed contact
instanton equation}, 
$$
(du - X_H \otimes \gamma)^{\pi(0,1)} = 0, \quad
 d(e^{g_{H, u}}u^*(\lambda + H \otimes \gamma)\circ j) = 0
$$
associated to a contact triad $(M,\lambda,J)$ and
contact Hamiltonian $H$ and its boundary value problem under the Legendrian boundary condition. 
(1) We identify the correct choice of the action functional
for perturbed contact Hamiltonian trajectories which provides a gradient structure
for the system and derive its first variation formula.
(2) We identify the correct choice of the energy for the bubbling analysis for the finite energy solutions 
for the equation.
(3) We  develop elliptic regularity theory for the solution, called \emph{perturbed contact
instantons}: We first establish a global $W^{2,2}$ bound by the Hamiltonian calculus
and the harmonic theory of the vector-valued one form  $d_Hu : = du - X_H(u)\otimes \gamma$
and its relevant Weitzenb\"ock formulae utilizing the contact triad connection
of the contact triad $(M,\lambda, J)$. Then we establish $C^{k,\alpha}$-estimates
by an alternating boot-strap argument between the $\pi$-component of $d_Hu$ and the
Reeb-component of $d_Hu$. Along the way, we also establish the boundary regularity theorem of
$W^{1,4}$-weak solutions of perturbed contact
instanton equation under the weak Legendrian boundary condition.
(4) Based on this regularity theory, we prove an
asymptotic $C^\infty$ convergence result at a puncture under the hypothesis of finite energy.
 \end{abstract}


\thanks{This work is supported by IBS project \#IBS-R003-D1.}


\date{}

\maketitle

\tableofcontents

\section{Introduction}

Let $(M, \xi)$ be a contact manifold of dimension $2n+1$ with $\xi$ its contact distribution.
Assuming $\xi$ is cooriented, let $\lambda$ be a contact form that is compatible with the coorientation of $\xi$.
Denote by $R_\lambda$ its Reeb vector field.
All contact forms compatible with the coorientation are of the form $f \lambda$ with $f\in C^\infty(M, \R^+)$.

The following system of PDEs, which we call the \emph{contact instanton} equation,
 \be\label{eq:contacton}
\delbar_J^\pi u=0, \quad d(u^*\lambda\circ j)=0
\ee
was introduced by Hofer \cite{hofer:gafa}.  Abbas-Cieliebak-Hofer
\cite{abbas-cieliebak-hofer} and by Abbas \cite{abbas} respectively utilized the equation in
their attempts to attack Weinstein's conjecture in three dimensions. They
lifted the equation to a perturbed pseudoholomorphic curve equation on
the symplectization in their applications.

In \cite{oh-wang:CR-map1, oh-wang:CR-map2}, Wang and the present author systematically
studied the system \eqref{eq:contacton} as it is on the given contact manifold, and established
its regularity theory employing well-known strategy of
utilizing Bochner-Weitzenb\"ock type formulae in the harmonic theory. They also establish  an asymptotic
convergence result on punctured Riemann surfaces in the closed string theory context
on a contact triad $(M, \lambda, J)$ whose domain is a punctured Riemann surface $(\dot\Sigma, j)$.
In a recent work \cite{oh:contacton-Legendrian-bdy}, the present author introduced a relevant
boundary value problem for the equation \eqref{eq:contacton} with Legendrian boundaries on
boundary punctured Riemann surfaces in the open string context. In particular, we show that the
boundary value problem is (nonlinear) elliptic by establishing the boundary counterparts of
the local coercive a priori estimates derived in \cite{oh-wang:CR-map1}. Moreover in the strip-like
coordinates around each puncture of the Riemann surface, it is proved that any finite $\pi$-energy
solution with $C^1$-bound has Reeb chords as its asymptotic limits at the punctures. Most importantly,
the Legendrian \emph{barrier} eliminates the phenomenon of the `appearance of spiraling instantons around the Reeb core'
and set the stage for a compactification and a Fredholm theory of the moduli space of
finite energy contact instantons constructed in \cite{oh-wang:CR-map1,oh:contacton}. Based on these, we
expected that such a compactified moduli space can be applied to the problem in contact topology.
Indeed in \cite{oh:entanglement1}, we have employed this analytical machinery
to prove a conjecture of Sandon \cite{sandon-translated} and Shelukhin \cite{shelukhin:contactomorphism} on an Arnold
conjecture-type question for the existence of translated points of contactomorphisms on
compact contact manifolds. In another front,  we also introduced the perturbed
contact  instanton equation deformed
by contact vector fields \cite{oh:contacton-Legendrian-bdy}.

In the present paper, we continue to take the spirit of similar harmonic theory
and establish all the counterparts of those
established in \cite{oh-wang:CR-map1,oh:contacton-Legendrian-bdy} for the Hamiltonian-perturbed
contact instanton equations using the covariant tensor calculus.
This time the task requires a correct choice of the density function and some ingenuity in
carrying out tensor calculations.
Furthermore the relevant tensor calculations and estimates become much harder than
the unperturbed case of \cite{oh-wang:CR-map1,oh:contacton-Legendrian-bdy}
due to the presence of the perturbation term of contact Hamiltonian vector field.
We would like to compare the present study of perturbed
contact instanton equation relative to the unperturbed ones in \cite{oh:contacton},
\cite{oh:contacton-Legendrian-bdy} in contact geometry,
 with Floer's breaking of conformal symmetry
of Gromov's pseudoholomorphic curves by Hamiltonian vector fields in
symplectic geometry.  In this regard, we anticipate that
the Hamiltonian perturbed contact instantons will play similar role in the study of
contact topology and contact Hamiltonian dynamics as Floer's Hamiltonian Floer homology did to
symplectic topology and Hamiltonian dynamics in symplectic geometry.
(See \cite{oh:entanglement1,oh:entanglement2} and \cite{oh-yso:1jet} for such applications.)
The geometry and analysis developed in the present paper, combined with those already given in \cite{oh:entanglement1,oh:contacton-gluing},
will be important in this regard.

\subsection{Hamiltonian perturbed contact instanton equation}

Fix a domain $\dot\Sigma$ which is a boundary-punctured surface as described above and
a  one-form $\gamma$ on $\dot \Sigma$ in general. (See Condition \ref{cond:gamma} for
the condition that we require $\gamma$ to satisfy.)
For a given  domain-dependent Hamiltonian $H: \dot \Sigma \times M \to \R$ with $H = H(z,y)$,
we consider the map
$$
u \mapsto du - X_H \otimes \gamma = d_H u
$$
as a section of the (infinite dimensional) vector bundle
$$
\CH \to \CF
$$
where $\CF$ is the set of smooth maps $u: \cdot \Sigma \to M$ and the fiber of the bundle at
$u$ is given by the set of $u^*TM$-valued one-form on $\dot \Sigma$
$$
\CH_u = \Omega^1(u^*TM) = \{\eta : \dot \Sigma \to u^*TM \mid \pi \circ \xi = id \}.
$$
We decompose
\be\label{eq:dHu-decompose}
d_H u = d_H^\pi u + \lambda(d_Hw)\, R_\lambda
\ee
where we write
\be\label{eq:lambdadHw}
\lambda(d_H u) = u^*\lambda + H \, \gamma.
\ee
More explicitly, we have
$$
du - X_H \otimes \gamma  =  \Pi(du - X_H \otimes \gamma) + \lambda(du - X_H \otimes \gamma))\, R_\lambda.
$$
For the convenience of exposition and for the simplicity of notations,
we  develop some notations here by systematically introducing the following notations.

\begin{notation}\label{nota:dHpiu-lambdaH}
Let $\gamma$ be a one-form on $\dot \Sigma$, and consider the
$u^*TM$-valued one-form given by
$$
du - X_H \otimes \gamma.
$$
By slight abuse of notations, we use the following notations:
\begin{enumerate}
\item
We denote
\be\label{eq:dHpiu}
d_H^\pi u = \Pi(du - X_H \otimes \gamma) =:(du - X_H \otimes \gamma)^\pi.
\ee
\item We also write
$\lambda_H: = \lambda + H\, \gamma$ and
\be\label{eq:u*lambdaH}
u^*\lambda_H: = u^*(\lambda + H\, \gamma) = u^*\lambda + u^* H\, \gamma.
\ee
\end{enumerate}
\end{notation}

We then decompose $d_H^\pi u: T\dot \Sigma \to \xi$ into its complex linear and
anti-complex linear parts
$$
d_H^\pi u = \del_H^\pi u + \delbar_H^\pi u
$$
where we have
\beastar
\del_H^\pi u & : = & (du - X_H \otimes \gamma)^{\pi(1,0)} =  \frac{d_H^\pi u- J \circ d_H^\pi u \circ j}{2}\\
\delbar_H^\pi u & : = &  (du - X_H \otimes \gamma)^{\pi(0,1)} = \frac{d_H^\pi u+ J \circ d_H^\pi u \circ j}{2}.
\eeastar

Next, for each given coorientation preserving contact diffeomorphism $\psi$ of $(M,\xi)$, we have
$\psi^*\lambda = e^g \lambda$  for some function $g$ which we denote by $g = g_\psi$ and call the \emph{conformal
exponent} of $\psi$.
We will need to consider a contact isotopy $\psi_H =\{\psi_H^t\}$ generated by a contact Hamiltonian
$H = H(t,x)$.

We will also use the following notations consistently throughout the present paper.

\begin{notation}
Let a time-dependent Hamiltonian $H = H(t,x)$ be given.
\begin{enumerate}
\item We denote by $\psi_H^t$ the flow of the contact Hamiltonian vector field $X_H$, and
denote by $\phi_H^t$ the isotopy
\be\label{eq:phiHt}
\phi_H^t: = \psi_H^t (\psi_H^1)^{-1}
\ee
for the simplicity of notation.
\item Let $u = u(\tau,t)$ be a function $u: \R \times [0,1] \to M$ be given. Then we define the function
$g_{H,u}$ on $\R \times [0,1]$ by
\be\label{eq:gHu}
g_{H,u}(\tau,t) :=  g_{(\phi_H^t)^{-1}}(t,u(\tau,t))= - g_{\phi_H^t}\left((\phi_H^t)^{-1}(u(\tau,t))\right).
\ee
\item By a slight abuse of notation, we write
$$
u^*(e^{g_{(\phi_H^t)^{-1}}}\lambda): = e^{g_{H,u}} u^*\lambda, \quad u^*(e^{g_{(\phi_H^t)^{-1}}}\lambda_H)
: = e^{g_{H,u}}(u^*\lambda + u^*H\, dt).
$$
\end{enumerate}
\end{notation}

Now we are ready to write down the \emph{perturbed contact instanton equation}, which is
\be\label{eq:perturbed-contacton-bdy-intro}
\begin{cases}
(du - X_H \otimes dt)^{\pi(0,1)} = 0, \\
 d(e^{g_{H, u}}u^*(\lambda + H \, dt)\circ j) = 0
 \end{cases}
\ee
when the domain surface  $\dot \Sigma$ is $ \R \times [0,1]$
together with the boundary condition
\be\label{eq:Legendrian-bdy-condition}
u(\tau,0) \in R_0, \quad u(\tau,1) \in R_1.
\ee
More generally when we are given a Legendrian tuple $(R_1,\cdots, R_k)$ is given, we consider
the equation of the type
\be\label{eq:contacton-Legendrian-bdy}
\begin{cases}
(du - X_H \otimes \gamma)^{\pi(0,1)} = 0, \quad
 d(e^{g_{H, u}}u^*(\lambda + H \, \gamma)\circ j) = 0, \\
u(\overline{z_iz_{i+1}}) \subset R_i, \quad i = 1,\cdots, k
\end{cases}
\ee
on a general punctured Riemann surfaces $\dot \Sigma$ with $k$ boundary
punctures equipped with a one-form $\gamma$.

We mention that when $H = 0$, the equation is reduced to the contact instanton equation
\eqref{eq:contacton} the analysis of which has been
studied in \cite{oh-wang:connection,oh-wang:CR-map1,oh:contacton,oh:contacton-Legendrian-bdy}.

\subsection{Perturbed action functional and its first variation}

Next we introduce the crucial notions of perturbed action integrals of a path
and the energy relevant to the global study of perturbed contact instantons
\eqref{eq:contacton-Legendrian-bdy} introduced in \cite[Introduction]{oh:contacton-Legendrian-bdy}.
It is not apparent at all to make the correct choice of the functional that provides
the equation \eqref{eq:perturbed-contacton-bdy-intro} with the gradient structure.
Our choice of the current definition partially arises as a hindsight of our earlier
work \cite{oh:entanglement1} on the proof of Sandon-Shelukhin's conjecture.

We start with the definition of perturbed action integrals associated to
contact Hamiltonian $H = H(t,x)$. We recall the standard contact action functional
$$
\CA(\gamma) = \int \gamma^*\lambda
$$
associated to the contact form $\lambda$ (associated to $H = 0$) in contact geometry.

\begin{defn}[Perturbed action functional] Let $H = H(t,x)$ be a contact Hamiltonian and recall
$\phi_H^t = \psi_H^t (\psi_H^1)^{-1}$.
We define a functional $\CA_H: \CL(R_0,R_1)  \to \R$ by
\be\label{eq:action}
\CA_H(\gamma) := \int_\gamma e^{g_{(\phi_H^t)^{-1}}(\gamma(t))} \gamma^*(\lambda + H\, dt)
 \left(=  \int_\gamma e^{g_{(\phi_H^t)^{-1}}(\gamma(t))} \gamma^*\lambda_H \right)
\ee
for any smooth path $\gamma:[0,1] \to M$. When $H = 0$, we write $\CA_0 = \CA$.
\end{defn}
 Then we have the following first variation formula.

\begin{prop}[Proposition \ref{prop:1st-variation}]\label{prop:1st-variation-intro}
For any vector field $\eta$ along $\gamma$, we have
\bea\label{eq:1st-variation}
\delta \CA_H(\gamma)(\eta) & = & \int_0^1 e^{g_{(\phi_H^t)^{-1}}(\gamma(t))}
\left(d\lambda(\dot \gamma - X_H(t, \gamma(t)), \eta) \right)\, dt \nonumber\\
 & {}& + \lambda( \eta(1))
-  e^{-g_{\psi_H^1}((\psi_H^1)^{-1}(\gamma(0))} \lambda(\eta(0))).
\eea
\end{prop}

When we are given a pair $(R_0,R_1)$ of Legendrian submanifolds,
we can consider the path space
$$
\CL(R_0,R_1) = \CL(M;R_0,R_1): = \{ \gamma:[0,1] \to M
\mid \gamma(0) \in R_0, \, \gamma(1) \in R_1\}
$$
and the restriction of $\CA$ thereto.
An immediate corollary of this proposition is that the critical point equation of the action functional
\emph{under the Legendrian boundary condition} is precisely
$$
(\dot \gamma(t) - X_H(t,\gamma(t)))^\pi = 0,
$$
i.e., $\dot \gamma(t) - X_H(t,\gamma(t)) = a(t) R_\lambda(\gamma(t))$
for some function $a = a(t)$. (See Proposition \ref{prop:lifting}
for some relevant discussion on the function $a$.)

It turns out that the correct definition of the
$\pi$-energy is the following. We will give the correct definition of \emph{horizontal energy}
(or the $\lambda$-energy) elsewhere. (See \cite{oh:contacton} for the closed string case
and \cite{oh:entanglement1} for the open string case.)

\begin{defn} Let $\dot\Sigma$ be a boundary-punctured Riemann surface of genus zero with punctures
equipped with a K\"ahler metric $h$ with \emph{strip-like ends} outside a compact subset
$K \subset \dot\Sigma$. Let $u: \dot \Sigma \to M$ be any smooth map with Legendrian boundary
condition. We define the total $\pi$-harmonic energy $E^\pi(u)$
by
\be\label{eq:energy-intro}
E_H^\pi(u) =  \frac{1}{2} \int_{\dot \Sigma} e^{g_{H,u}}  |d^\pi u - X_H^\pi \otimes \gamma|_J^2
\ee
where the norm is taken in terms of the given metric $h$ on $\dot \Sigma$
and the triad metric associated to the contact triad $(M,\lambda,J)$.
\end{defn}
With this definition, we show the following identity
\be\label{eq:energy-action}
E_{H,J} ^\pi(u) = \CA_H(u(+\infty) - \CA_H(u(-\infty))
\ee
in Theorem \ref{thm:energy-action} for any finite energy perturbed contact instanton $u$.
Here $u$ satisfies
a suitable asymptotic conditions on the strip-like region of $\dot \Sigma$ spelled out
as follows: the associated energy density function thereon
is nothing but the $\pi$-energy $E^\pi(\overline u)$ for the map
\be\label{eq:ubar-intro}
\overline u(\tau,t) = (\psi_H^t(\psi_H^1)^{-1}))^{-1}(u(\tau,t)).
\ee
(See \cite[Section 6]{oh:entanglement1} for the explanation on the perspective of such a choice.)
\begin{rem}
The transformation \eqref{eq:ubar-intro} is associated to the boundary condition
$$
\overline u(\tau,0) \in \psi(R_0), \quad \overline u(\tau,1) \in R_1
$$
when $u$ satisfies $u(\tau,0) \in R_0$, $u(\tau,1) \in R_1$.
Depending on the circumstances, one may also use the transformation
$$
\widetilde u(\tau,t): = \psi_H^t(u(\tau,t))
$$
which satisfies
$
\widetilde u(\tau,0) \in R_0, \, \widetilde u(\tau,1) \in \psi_H^1(R_1).
$
Under this gauge transformation, we should replace the definition of the function $g_{H,u}$
by the simpler one
$$
g'_{H,u}(t,x) = g_{\psi_{H^t}}(x).
$$
To be consistent with the exposition of \cite{oh:entanglement1}, we do not use this one but
the above \eqref{eq:ubar-intro} in the present paper.
(See  \cite{oh:cag} for the explanation on the general perspective on these transformations
in the symplectic geometry.)
\end{rem}

The upshot of the definition \eqref{eq:energy-intro} of the $\pi$-energy is the presence of the
conformal factor $e^{g_{H,u}}$. This is not totally surprising in that
the conformal exponent is known to play an important role
in the study of contact Hamiltonian dynamics because general contactomorphism
 does not preserve the contact form $\lambda$, unless it
is a \emph{strict contactomorphism}. (See \cite{mueller-spaeth-I}, \cite{usher:conformal-factor} for such
a study.)

\subsection{Harmonic theory of perturbed contact instantons}

Our main strategy of proving a priori estimates is to use the harmonic theory
 based on the strategy of utilizing the Weitzenb\"ock formula. Such strategy
is well established prevalently in the geometric analysis. To maximize the advantage of
using the tensor calculus in the study of contact instantons, we will always use the
\emph{contact triad connection} that Wang and the present author had introduced in \cite{oh-wang:connection}
associated to each contact triad $(M,\lambda,J)$. We also freely use the covariant tensor
calculus of vector-valued differential forms which is the standard practice in geometric
analysis but not so for the expected readers of the present article.
 For readers' convenience, we summarize the defining
properties of the connection in Appendix \ref{sec:connection}, and the covariant calculus of
vector-valued forms in Appendix \ref{sec:vectorvalued-forms} respectively.

\begin{defn}[Contact Hermitian connection \cite{oh-wang:connection}]
Let $\nabla$ be the contact triad connection on $TM$. We denote by $\nabla^\pi$ the associated Hermitian
connection on the Hermitian vector bundle $(\xi, d\lambda|_\xi,J)$ with $\xi \subset TM$, which is defined by
$$
\nabla^\pi_X : = \Pi \nabla_X|_{\xi}: TM \otimes \xi \to \xi,
$$
 i.e.,
$$
\nabla_X^\pi Y = \Pi(\nabla_X Y)
$$
for any pair of $(X,Y) \in \Gamma(TM) \times \Gamma(\xi)$.
\end{defn}

The following is the perturbed analog to the contact Cauchy-Riemann map introduced in \cite[Definition 1.1]{oh-wang:CR-map1}.

\begin{defn}[Perturbed contact Cauchy-Riemann map] We call a map $u: \dot \Sigma \to M$
a \emph{$H$-perturbed contact Cauchy-Riemann map}, abbreviated as \emph{$H$-perturbed CR map}, if
$$
du -X_H(u)\otimes \gamma \in \Hom(T\dot \Sigma, u^*TM)
$$
is a $(j,J)$-complex linear map, i.e., if $u$ satisfies
$$
(du-X_H(u)\otimes \gamma)^{\pi (0,1)}=0.
$$
\end{defn}

We define the \emph{off-shell}
perturbed energy $e_H(u)$ and the $\pi$-energy density functions $e_H^\pi(u)$ 
of any smooth function by
\beastar
e_H(u) &: = & |d_H u|^2 = |du - X_H \otimes \gamma|^2\\
e_H^\pi(u)& : = & |d_H^\pi u|^2 = |(du - X_H \otimes \gamma)^\pi|^2.
\eeastar
By definition, we have the identity $e_H(u) = e_H^\pi(u) + |u^*\lambda|^2$.
The following a priori \emph{on-shell} identity for the perturbed  $\pi$-harmonic energy density
function is the basis for our a priori estimates. This is perturbed counterpart of
\cite[Theorem 1.2]{oh-wang:CR-map1}.

\begin{thm}[Fundamental equation; Theorem \ref{thm:fundamental}]\label{thm:fundamental-intro}
Suppose that $\gamma$ is a one-form on $\dot \Sigma$ and let $u:(\dot \Sigma, j) \to (M,J)$
be a $H$-perturbed CR map. Then we have
\bea\label{eq:dd_Hu2}
d^{\nabla^\pi}(d_H^\pi u) & = & u^*\lambda\wedge(\frac{1}{2}(\CL_{R_\lambda} J)J d_H^\pi u) -
2 T^\pi(X_H^\pi(u), \gamma \wedge d_H^\pi u) \nonumber \\
&{}& +  u^*\lambda\wedge(\frac{1}{2}(\CL_{R_\lambda} J)JX_H^\pi(u)\, \gamma)
- d^{\nabla^\pi} (X_H^\pi \, \gamma).
\eea
\end{thm}

The upshot of  this identity is that the derivative of the one-from  $d_H^\pi u$ is expressed
in terms of the linear expression thereof with some inhomogeneous terms so that the
equation provides the stage of boot-strap arguments for the regularity study.

Using the definition \eqref{eq:dHpiu}, we  have the decomposition
$$
d_H u = d_H^\pi u + u^*\lambda_H \otimes R_\lambda
$$
and  hence the decomposition of the full covariant differential
$$
\nabla (d_H u) = \nabla (d_H^\pi u) + \nabla (u^*\lambda_H) \, R_\lambda
+u^*\lambda_H \, \nabla R_\lambda.
$$
We will later show the equality $\nabla (d_H^\pi u) = \nabla^\pi (d_H^\pi u)$ which
reduces  the task of estimating the full covariant derivative $\nabla (d_Hu)$ to a simpler
problem of estimating the $\xi$-component $\nabla^\pi (d_H^\pi u)$ and the Reeb
component $\nabla(u^*\lambda_H)$ separately.
(See  the proof of Proposition \ref{prop:nabladHu}.)
A priori, the task of estimating the norm square $|\nabla(d_Hu)|^2$
of the full derivative would be much more challenging because  the two  components have
rather different behavior owing to the fact that they satisfy  different-type of  equations as
presented in \eqref{eq:contacton-Legendrian-bdy}.

Utilizing  the general Weitzenb\"ock formula (e.g.,
in \cite[Appendix C]{freed-uhlenbeck}, \cite[Appendix A]{oh-wang:CR-map1}), we derive the following
fundamental formula for the covariant harmonic Laplacian $\Delta e_H^\pi(u)$: This is the perturbed counterpart of
the formula \cite[Equation (4.11)]{oh-wang:CR-map1}. In the present paper, we
will freely use standard convention for the covariant calculus of vector-bundle valued
forms presented such as in \cite[Section 5 of Chapter 5]{wells-book} or in
\cite[Appendix B]{oh-wang:CR-map1}.

\begin{thm}[Laplacian of the energy density]\label{thm:e-pi-weitzenbeck}
Let $u$ be any $H$-perturbed contact CR map.
Then we have
\bea\label{eq:e-pi-weitzenbeck}
-\frac{1}{2}\Delta e_H^\pi(u)&=&|\nabla^\pi (\del_H^\pi u)|^2+K|\del_H^\pi u|^2+\langle \ric^{\nabla^\pi} (\del_H^\pi u), \del_H^\pi u\rangle\nonumber\\
&{}&- \langle \delta^{\nabla^{\pi}}[(u^*\lambda \wedge (\CL_{X_\lambda}J)J \del_H^\pi u],
\del_H^\pi u\rangle \nonumber\\
&{}& + \langle \delta^{\nabla^\pi}[ u^*\lambda\wedge((\CL_{R_\lambda} J)JX_H^\pi(u)],
\del_H^\pi u\rangle \nonumber\\
&{}&- 4 \langle \delta^{\nabla^\pi} [T^\pi(X_H^\pi(u), \gamma \wedge d_H^\pi u)], d_H^\pi u\rangle
\nonumber\\
&{}& - 2 \langle d^{\nabla^\pi}X_H^\pi (u)\wedge \gamma, \del_H^\pi u \rangle
- 2 \langle X_H^\pi(u) \wedge d\gamma, \del_H^\pi u\rangle.
\eea
\end{thm}

\subsection{$W^{2,2}$-estimates and higher $C^{k,\alpha}$-regularity estimates}

By applying tensorial calculations and utilizing the Weitzenb\"ock formula in a crucial way, we
derive the following crucial differential inequality.

\begin{thm}[Fundamental differential inequality]\label{thm:FDI-intro}
Let $u$ be any $H$-perturbed contact CR map satisfying Legendrian boundary
condition. Then
\be
|\nabla (d_Hu)|^2 \leq |\nabla^\pi (d_H^\pi u)|^2 + |\nabla (u^*\lambda_H)|^2 + \|\nabla R_\lambda\|_{C^0} |u^*\lambda||du|
\label{eq:|nabladu|}
\ee
and
\bea
\frac18 |\nabla (d_H u)|^2
&\leq & -\frac{1}{2}\Delta e_H(u) + C_1'e_H(u)^2 +  C_2' e_H(u) \nonumber \\
& {} &  \quad + C_3' |du|^2 + C_4' |du|^4. \label{eq:higher-derivative}
\eea
\end{thm}
With these differential inequalities in our disposal, we can establish the following
a priori local $W^{2,2}$-estimates in the same fashion as the one
given in \cite[Appendix C]{oh-wang:CR-map1} but with much more complex tensorial calculations.

\begin{thm}[Theorem \ref{thm:coercive-L2}]\label{thm:coercive-L2-intro}
For any pair of domains $D_1$ and $D_2$ in $\dot\Sigma$ such that $\overline{D_1}\subset D_2$,
$$
\|\nabla(d_H u)\|^2_{L^2(D_1)}  \leq
 C_1 \|d_H u\|_{L^4(D_2)}^4 + C_2 \|d_H u\|_{L^2(D_2)}^2 + C_3 \|du\|_{L^4(D_2)}^4 +
 C_4 \|du\|_{L^2(D_2)}^2 + 2 \CB
$$
with the boundary contribution
$$
\CB =
\int_{\del D_2} ( C_5 |du|^3 |d_Hu| + C_6 |du|^2|d_Hu|  + C_7 |d_H u| |du|)
$$
for any perturbed contact instanton $u$,
where $C_i = C_i(D_1, D_2)$ are some constants which depend on $D_1$, $D_2$ and
 $(M, \lambda, J)$ and $H$, but are independent of $u$.
\end{thm}

Then we proceed with the higher regularity estimates. For this one, we follow the scheme
used in \cite{oh:contacton-Legendrian-bdy} of directly proving the following Schauder-type
estimates.

\begin{thm}\label{thm:local-regularity-intro}
Let $u$ be a perturbed contact instanton satisfying
and \eqref{eq:contacton-Legendrian-bdy}.
Then for any pair of domains $D_1 \subset D_2 \subset \dot \Sigma$ such that $\overline{D_1}\subset D_2$, we have
$$
\|d_H u\|_{C^{k,\alpha}(D_1)} \leq C \|d_H u\|_{W^{1,2}(D_2)}.
$$
for some constant $C$ depending on $J$, $\lambda$ and $D_1, \, D_2$
 and $(M, \lambda, J)$ and $H$, but are independent of $u$, but independent of $u$.
\end{thm}

The proof of this higher regularity estimates will be given by an alternating bootstrap argument
similarly as in \cite{oh:contacton-Legendrian-bdy}, which is based on the following structure of the above fundamental
equation in \emph{isothermal coordinates}. (See also \cite[Subsection 11.5]{oh-wang:CR-map2}.)

\begin{prop}[Fundamental equation in isothermal coordinates]\label{prop:FE-in-isothermal-intro}
Let $u$ be a solution to \eqref{eq:contacton-Legendrian-bdy}.
We consider $\zeta = d_H^\pi u(\del_x)$ and a complex valued function
$$
\alpha(x,y) = \lambda_H\left(\frac{\del u}{\del y}\right)
+ \sqrt{-1}\left(\lambda_H\left(\frac{\del u}{\del x}\right)\right)
$$
in an isothermal coordinate $(x,y)$ of $(\dot \Sigma,j)$. Then  the pair $(\zeta,\alpha)$ satisfies
\be\label{eq:main-eq-isothermal}
\begin{cases}
\nabla_x^\pi \zeta + J \nabla_y^\pi \zeta + B(u^*\lambda, \zeta) = -* JP(u) \\
\zeta(z) \in TR_i \quad \text{for } \, z \in \del D_2
\end{cases}
\ee
where $B(u^*\lambda, \zeta)$, $P(u)$ are defined in \eqref{eq:P(u)},  and
\be\label{eq:equation-for-alpha}
\begin{cases}
\delbar \alpha
=\frac12 |\zeta|^2 + G(du,H)\\
\alpha(z) \in \R \quad \text{for } \, z \in \del D_2
\end{cases}
\ee
where
$$
G(du,H) = (u^*(R_\lambda[H] \wedge \gamma + u^*H d\gamma)(\del_x,\del_y).
$$
Furthermore we have $JP(u), \, G(du,H) \in W^{1,2}(\dot \Sigma)$, when $u \in W^{2,2}(\dot \Sigma)$.
\end{prop}

The above higher regularity estimate reduces the convergence study of perturbed contact instantons
 to the study of $C^1$-estimates which is based on the bubbling analysis
developed in \cite{oh:contacton}.

\subsection{Asymptotic subsequence convergence}

We put the following hypotheses in our asymptotic study of the finite
energy contact instanton maps $u$ similarly as in \cite{oh-wang:CR-map1},
\cite{oh:contacton-Legendrian-bdy}.

\begin{hypo}\label{hypo:basic}
Let $h$ be the metric on $\dot \Sigma$ given above.
Assume $u:\dot\Sigma\to M$ satisfies the contact instanton equation \eqref{eq:contacton-Legendrian-bdy}
and
\begin{enumerate}
\item $E_H^\pi(u)<\infty$ (finite $\pi$-energy);
\item $\|d u\|_{C^0(\dot\Sigma)}, \quad \|d_H\|_{C^0(\dot \Sigma)} <\infty$.
\item $\Image u \subset K \subset M$ for some compact set $K$.
\end{enumerate}
\end{hypo}

Let $u$ satisfy Hypothesis \ref{hypo:basic}. We associate two natural asymptotic invariants
to each given boundary puncture  of $\dot \Sigma$.

\begin{defn}\label{eq:action-charge}
 Consider the family of paths $u_\tau$ defined by
$$
u_\tau(t) := u(\tau,t)
$$
for $\tau \in [0,\infty)$ in the given strip-like coordinates $(\tau,t)$ at the
positive puncture of our concern. We define
\bea
T_H & := & \frac{1}{2}\int_{[0,\infty) \times [0,1]} e^{g_{H,u}} |d_H ^\pi u|^2
+ \int_{[0,1]} e^{g_{H,u}} (u^*\lambda + H\, dt)|_{\{\tau = 0\}} \label{eq:TQ-T}\\
Q_H & : = & \int_{[0,1]} e^{g_{H,u}}  u^*(\lambda + H\, dt)\circ j)|_{\{\tau =0\}}
\label{eq:TQ-Q}
\eea
We call $T_H$ and $Q_H$ the \emph{asymptotic (perturbed) action} and
\emph{asymptotic (perturbed) charge} of $u$ respectively.
(The case of negative punctures is similar and omitted.)
\end{defn}

With this definition, we prove the following asymptotic convergence result
which generalizes the ones from \cite{oh-wang:CR-map1} and \cite{oh:contacton-Legendrian-bdy}
to the case of perturbed contact instantons.

\begin{thm}[Subsequence Convergence, Theorem \ref{thm:subsequence}]
\label{thm:subsequence-intro}
Let $u:[0, \infty)\times [0,1]\to M$ satisfy the
perturbed contact instanton equations \eqref{eq:contacton-Legendrian-bdy}
and Hypothesis \ref{hypo:basic}.
Then for any sequence $s_k\to \infty$, there exists a subsequence, still denoted by $s_k$, and a
massless instanton $u_\infty(\tau,t)$ (i.e., $E^\pi(u_\infty) = 0$)
on the cylinder $\R \times [0,1]$  that satisfies the following:
\begin{enumerate}
\item $\delbar_H^\pi u_\infty = 0$ and
$$
\lim_{k\to \infty}u(s_k + \tau, t) = u_\infty(\tau,t)
$$
in the $C^l(K \times [0,1], M)$ sense for any $l$, where $K\subset [0,\infty)$ is an arbitrary compact set.
\item The limit $u_\infty$ has vanishing asymptotic charge $Q_H = 0$.
\item At each puncture, there  exists a Reeb chord $\gamma$
 joining $R_0$ and $R_1$ with its action $T_H \neq 0$ such that satisfies
\be\label{eq:Reeb-translated-chord}
\overline u_\infty(\tau,t)= \gamma(T_H\, t).
\ee
\end{enumerate}
\end{thm}

Motivated by this theorem and also by the notion of \emph{translated points} introduced by
Sandon \cite{sandon-translated}, we introduce the following intersection analog thereto.

\begin{defn}[Reeb translated Hamiltonian chords] Let $H = H(t,x)$ be a contact Hamiltonian,
and $(R_0,R_1)$ be a pair of Legendrian submanifolds.
We call the path $\widetilde \gamma$ of the type
$$
\widetilde \gamma(t) = \psi_H^t(\psi_H^1)^{-1}(\gamma(Tt)), \quad \widetilde \gamma(0) \in R_0, \, \widetilde \gamma(1) \in R_1
$$
a \emph{Reeb-translated Hamiltonian chord} from $R_0$ to $R_1$.
\end{defn}
These Reeb translated Hamiltonian chords will be the generators of a Floer-type
complex associated to the perturbed contact instantons.
(See \cite{oh:entanglement1}.)

\subsection{Comments, convention and notations}

We suspect that the style of the present paper, which is full of tensor
calculations, is not familiar to most of anticipated  readers thereof. Because of this,
we would like to make some tips  and warnings in reading the present paper.

First of all, we presume readers are already familiar with the defining 
properties of contact triad connection.
Many boot-strap arguments establishing the coercive regularity estimates 
in the present paper extensively relies on precise Bochner-Weitzenbeck type identities and more
whose derivations crucially rely on our systematic usage of \emph{contact triad
connection} on $M$ and its associated \emph{contact Hermitian connection} 
introduced by Wang and the present author in \cite{oh-wang:connection} and utilized
by them in \cite{oh-wang:CR-map1,oh-wang:CR-map2} and by the present author
in \cite{oh:contacton-Legendrian-bdy}.  In the general harmonic theory e.g., in the 
transcedental method of studying geometry of holomorphic vector bundles in
complex geometry,  usage of Chern connection \cite{chern-connection} has been useful to 
optimize relevant tensor calculations. Likewise the way how Wang and the present
author had come up with the \emph{contact triad connection} in \cite{oh-wang:connection}
was through their effort in \cite{oh-wang:CR-map1,oh-wang:CR-map2} 
to simplify tensor calculations in their study of harmonic
theory of contact instantons associated to the given contact triad $(M,\lambda, J)$. 
It has also been proven
by the tensor calculations performed in \cite{oh:contacton-Legendrian-bdy} and in 
the present paper that the triad connection is also well-adapted to the calculations
involving the Legendrian boundary condition: It turns out to be very useful in organizing
various terms into something that carries natural geometric meaning (See the proofs of 
\cite[Theorem 4.5]{oh:contacton-Legendrian-bdy} and of Subsections 
\ref{subsec:W22-estimate} and \ref{subsec:boundary-integral}.)

Therefore it will be important for readers to
digest the defining properties of the contact triad connection  and the way how they are exercised in this 
calculation in \cite{oh-wang:CR-map1}. For readers' convenience, we
summarize basic properties of the contact triad connection in Appendix \ref{sec:connection},
but we strongly suggest readers to have these two papers in handy and consult them 
whenever they feel necessary in following tensor calculations.

We will also use the following sign convention and notations consistently throughout
the paper.

\begin{itemize}
\item We denote by $\nabla$ both the triad connection on $M$ and its pull-back 
connection $u^*\nabla$ on $\dot \Sigma$
for notational simplicity, whose usage should be clear from the context.
\item {(Contact Hamiltonian)} We define the contact Hamiltonian of a contact vector field $X$ to be
$$
- \lambda(X).
$$
\item
For given time-dependent function $H = H(t,x)$, we denote by $X_H$ the associated contact Hamiltonian vector field
whose associated Hamiltonian $- \lambda(X_t)$ is given by $H = H(t,x)$, and its flow by $\psi_H^t$.
\item When $\psi = \psi_H^1$, we say $H$ generates $\psi$ and write $H \mapsto \psi$.
\item {(Developing map)} $\Dev(t \mapsto \psi_t)$: denotes the time-dependent contact Hamiltonian generating
the contact Hamiltonian path $t \mapsto \psi_t$.
\item For given one-form $\gamma$ on $\dot \Sigma$, we write $d_Hu: = du - X_H \otimes \gamma$,
and $\lambda_H = \lambda + H\, \gamma$.
\item We write $\phi_H^t: = \psi_H^t \circ (\psi_H^1)^{-1}$.
\item $g_{\psi}$: Conformal exponent defined by the equation $\psi^*\lambda = e^{g_\psi}\lambda$
for a contactomorphism $\psi$ of contact manifold $(M,\xi)$.
\item $g_{H,u}: = g_{(\phi_H^t)^{-1}}\circ u$ for given Hamiltonian $H$ and a map $u:\dot \Sigma \to M$.
\item The constants $C, \, C'$ vary place by place which are independent of the function $u$
depending only on the triad $(M,\lambda, J)$ and the domain metric $h$ of $\dot \Sigma$.
\end{itemize}

\section{Perturbed action functional and the first variation}
\label{sec:perturbed-action}

Our earlier study of perturbed contact instantons and its application to Sandon-Shelukhin's conjecture
in \cite{oh:contacton-Legendrian-bdy}, \cite{oh:entanglement1} in hindsight lead us to the following
off-shell definition of perturbed action functionals for the contact Hamilton's equation
$\dot x = X_H(t,x)$.

\begin{defn}[Perturbed action] Let $H = H(t,x)$ be a contact Hamiltonian
and a pair $(R_0,R_1)$ of Legendrian submanifolds. We consider the path space
$$
\CL(R_0,R_1) = \CL(M;R_0,R_1): = \{ \gamma:[0,1] \to M \mid \gamma(0) \in R_0, \, \gamma(1) \in R_1\}.
$$
We define a functional $\CA_H: \CL(R_0,R_1)  \to \R$ given by
$$
\CA_H(\gamma) :=  \int_\gamma e^{g_{(\phi_H^t)^{-1}}(\gamma(t))} \gamma^*\lambda_H
= \int_\gamma e^{g_{(\phi_H^t)^{-1}}(\gamma(t))} \gamma^*(\lambda + H\, dt).
$$
When $H = 0$, we write $\CA_0 = \CA$.
\end{defn}

The following lemma connects this definition of perturbed action functional with the unperturbed one.

\begin{lem}\label{lem:action-identity} For given path $\gamma \in \CL(R_0,R_1)$, consider the path
$\overline \gamma \in \CL(\psi_H^1(R_0), R_1)$ defined by
$$
\overline \gamma(t): = (\psi_H^t(\psi_H^1)^{-1})^{-1}(\gamma(t)).
$$
Then we have
$$
\CA_H(\gamma) = \CA(\overline \gamma).
$$
\end{lem}
\begin{proof} We compute
$$
(\overline \gamma)^*\lambda(\del_t) = \lambda\left(\frac{\del \overline \gamma}{\del t} \right).
$$
By definition, we have
$$
\gamma(t) = \psi_H^t(\psi_H^1)^{-1}(\overline \gamma(t)).
$$
Then
$$
\frac{d \gamma}{d t}(t)
= d\left(\psi_H^t(\psi_H^1)^{-1}\right)\left(\frac{d \overline \gamma}{d t}\right) + X_H(\gamma(t))
$$
and hence we have
$$
\frac{d \overline \gamma}{dt} = d(\psi_H^t(\psi_H^1)^{-1}))^{-1}
\left(\frac{d\gamma }{d t} - X_H(\gamma(t))\right).
$$
Now we evaluate
\beastar
\lambda\left(\frac{d \overline \gamma}{dt}\right) & = &
\lambda\left(d(\psi_H^t(\psi_H^1)^{-1})\right)^{-1}
\left(\frac{d\gamma }{d t} - X_H(\gamma(t)\right)\\
& = & e^{g_{(\phi_H^t)^{-1}}(\gamma(t))}
\lambda\left(\frac{d\gamma }{d t} - X_H(\gamma(t))\right)\\
& = & e^{g_{(\phi_H^t)^{-1}}(\gamma(t))}
(\gamma^*\lambda +  H\, dt)(\del_t) =
e^{g_{(\phi_H^t)^{-1}}(\gamma(t))}  \gamma^*\lambda_H(\del_t).
\eeastar
By taking the integral over $[0,1]$, we have finished the proof.
\end{proof}

Consider the free path space
$$
\CP: = C^\infty([0,1],M) = \{\gamma: [0,1] \to M \}
$$
and consider the action functional
\be\label{eq:AAH}\CA_H(\gamma)
= \int_0^1 e^{g_{(\phi_H^t)^{-1}}(\gamma(t))} \left(\lambda\left(\dot \gamma(t)\right)+ H\right)\, dt.
\ee
We now derive the following  first variation formula
of the action functional \eqref{eq:AAH} on a \emph{free} path space $\CP$.

\begin{prop}\label{prop:1st-variation} For any vector field $\eta$ along $\gamma$, we have
\bea\label{eq:1st-variation}
\delta \CA_H(\gamma)(\eta) & = & \int_0^1 e^{g_{(\phi_H^t)^{-1}}(\gamma(t))}
\left(d\lambda(\dot \gamma - X_H(t, \gamma(t)), \eta) \right)\, dt \nonumber\\
 & {}& +  \lambda( \eta(1))
-  e^{g_{(\psi_H^1)^{-1}}(\gamma(0))} \lambda(\eta(0)).
\eea
\end{prop}
\begin{proof} If we consider the path $\overline \gamma$
defined by $\overline \gamma(t) = (\phi_H^t)^{-1}(\gamma(t))$, we have
$$
\overline \gamma^*\lambda(\del_t) = e^{g_{\phi_H^t(\gamma(t)}} \gamma^*\lambda_H.
$$
Let $\gamma_s$ with $ -\varepsilon < s < \varepsilon$ such that
$$
\gamma_0 = \gamma, \quad \frac{\del \gamma_s}{\del s}\Big|_{s = 0} = \eta.
$$
We compute
$$
\frac{d}{ds}\Big|_{s=0} \CA(\overline \gamma_s) = \frac{d}{ds}\Big|_{s=0}\int \overline
\gamma_s^*\lambda.
$$
We compute
\be\label{eq:dds-gamma*lambda}
\frac{\del}{\del s}\Big|_{s=0} \overline \gamma_s^*\lambda
= d(\overline \eta \rfloor \lambda) + \overline \eta \rfloor d\lambda
\ee
where we put
$$
\overline \eta: = \frac{\del \overline \gamma_s}{\del s}\Big|_{s = 0} .
$$
We compute
$$
\overline \eta = \frac{\del \overline \gamma_s}{\del s}\Big|_{s = 0} (t)
= d(\phi_H^t)^{-1}(\eta(t))
$$
Substituting this into \eqref{eq:dds-gamma*lambda}, we derive
\bea\label{eq:dds-int}
\frac{d}{ds}\Big|_{s=0} \int \overline \gamma_s^*\lambda
& = & \int_{\overline \gamma} d(\overline \eta \rfloor \lambda) + \overline \eta \rfloor d\lambda \nonumber \\
& = & \int_{\overline \gamma} \overline \eta \rfloor d\lambda
+ \lambda(\overline \eta(1)) -  \lambda(\overline \eta(0))
\nonumber \\
& = & \int_{\overline \gamma}  (d(\phi_H^t)^{-1}(\gamma(t))(\eta(t))  \rfloor d\lambda
 + \lambda(\overline \eta(1)) -  \lambda(\overline \eta(0)) \nonumber \\
& = & \int_0^1 d\lambda\left((d(\phi_H^t)^{-1}(\eta(t)), \dot{\overline \gamma}(t)\right)
\, dt + \lambda(\overline \eta(1)) -  \lambda(\overline \eta(0)).
\eea
Finally we note
$$
\dot{\overline \gamma} = d(\phi_H^t)^{-1}(\dot \gamma - X_{H_t}(\gamma(t)).
$$
By substituting this into \eqref{eq:dds-int}, we have derived
\beastar
\frac{d}{ds}\Big|_{s=0} \int \overline \gamma_s^*\lambda &= &
\int_0^1 e^{g_{(\phi_H^t)^{-1}}(\gamma(t))}
d\lambda \left(\eta(t), \dot \gamma(t) - X_{H_t}(\gamma(t)\right)\, dt
\nonumber \\
& {}& + \lambda( \eta(1))
-  e^{g_{(\psi_H^1)^{-1}}(\gamma(0))} \lambda(\eta(0))\\
& = & \int_0^1 e^{-g_{\phi_H^t}((\phi_H^t)^{-1}\gamma(t))}
d\lambda \left(\eta(t), \dot \gamma(t) - X_{H_t}(\gamma(t)\right)\, dt
\nonumber \\
& {}& + \lambda( \eta(1))
-  e^{-g_{\psi_H^1}((\psi_H^1)^{-1}(\gamma(0)))} \lambda(\eta(0))\\
\eeastar
where for the second equality we used the identity
$$
g_{\psi^{-1}} = - g_\psi \circ \psi^{-1}
$$
for the conformal exponent of general contactomorphism $\psi$.
(See \cite[Lemma 2.4 (2)]{oh:contacton-Legendrian-bdy}.)
Combining the above, we have finished the proof.
\end{proof}

An immediate corollary of this first variation formula shows that
the Legendrian boundary condition is a natural boundary condition for the
action functional $\CA_H$ in that it kills the boundary contribution in
the first variation.

\begin{cor} Let $(R_0,R_1)$ be a Legendrian pair. Then we have
\be\label{eq:1st-variation-R0R1}
\delta \CA_H(\gamma)(\eta) = \int_0^1 e^{g_{(\phi_H^t)^{-1}}(\gamma(t))}
\left(d\lambda(\dot \gamma - X_H(t, \gamma(t)), \eta) \right)\, dt
\ee
on $\CL(R_0,R_1)$.
In particular, $\gamma$ is a critical point of $\CA_H|_{\CL(R_0,R_1)}$ if and only if
$$
(\dot \gamma - X_H(t, \gamma(t)))^\pi = 0.
$$
\end{cor}
\begin{proof} The first statement follows immediately by the vanishing of $\lambda$
on the Legendrian submanifold. Therefore we have derived
$\delta \CA_H(\gamma)(\eta) = 0$ for all $\eta \in T_{\gamma}\CL(R_0,R_1)$ if
and only if
$$
\left(d(\phi_H^t)^{-1}(\dot \gamma - X_{H_t}(\gamma(t))\right)^\pi = 0.
$$
On the other hand, the differential $d(\phi_H^t)^{-1}$ of the contactomorphism
$(\phi_H^t)^{-1}$ preserves the contact distribution $\xi$, which shows that
the above vanishing is equivalent to
$$
(\dot \gamma - X_H(t,\gamma(t))^\pi = 0.
$$
This finishes the proof.
\end{proof}

We now examine the relationship between the critical points of
the aforementioned constrained action functional and the contact Hamiltonian trajectories.

\begin{prop}\label{prop:lifting} Suppose a path $\gamma:[0,1] \to M$ satisfies
\be\label{eq:Crit-AAH}
(\dot \gamma - X_H(t, \gamma(t)))^\pi = 0.
\ee
Define the function $\rho: [0,1] \to \R$ by $\rho(t): = \CA_H(\gamma|_{[0,t]})$ and
consider the Reeb-translated Hamiltonian
$$
\widetilde H(t,x): = H(t, \phi_{R_\lambda}^{\rho(t)}(x)).
$$
Then the Reeb-translated curve
$$
\widetilde \gamma(t) = \phi_{R_\lambda}^{-\rho(t)}(\gamma(t))
$$
satisfies the Hamilton's equation $\dot x = X_{\widetilde H}(t,x)$ for $\widetilde H$.
\end{prop}
\begin{proof} Since $\gamma$ satisfies \eqref{eq:Crit-AAH}, we can write
$\dot \gamma(t) - X_H(t,\gamma(t)) = b(t) R_\lambda(\gamma(t))$, i.e.,
$$
\dot \gamma(t) = b(t) R_\lambda(\gamma(t)) + X_H(t,\gamma(t))
$$
for some function $b = b(t)$. In fact, we have
$$
b(t) = \lambda(\dot \gamma(t) - X_H(t,\gamma(t))) = \lambda(\dot \gamma(t)) + H(t,\gamma(t)).
$$

Now consider the flow of the time-dependent
vector field $b(t) R_\lambda$ which is just a reparameterization of the Reeb flow
$$
t \mapsto \phi_{R_\lambda}^{\rho(t)}, \quad \rho(t): = \int_0^t b(u)\, du = \CA_H(\gamma|_{[0,t]}).
$$
By definition, we have $b(t) = \rho'(t)$. We define
$$
\widetilde \gamma(t): = \phi_{R_\lambda}^{-\rho(t)}(\gamma(t))
$$
and compute its derivative
\beastar
\frac{d}{dt} \widetilde \gamma(t) & = & - \rho'(t) R_\lambda(\widetilde \gamma(t))
+ d\psi_{R_{\lambda}}^{-\rho(t)} \dot \gamma(t)\\
& = & - \rho'(t) R_\lambda(\widetilde \gamma(t))
+ d\psi_{R_{\lambda}}^{\rho(t)}(\rho'(t)R_\lambda(\gamma(t)) + X_H(t,\gamma(t)))\\
& = & d\psi_{R_{\lambda}}^{-\rho(t)}(X_H(t,\gamma(t)))
= (\psi_{R_{\lambda}}^{\rho(t)})^*X_H(t,\widetilde \gamma(t))) \\
& = & X_{H_\circ \phi_{R_\lambda}^{\rho(t)}}(t,\widetilde \gamma(t))).
\eeastar
where we use the identity
$$
R_\lambda(\widetilde \gamma(t)) = d\psi_{R_{\lambda}}^{\rho(t)} R_\lambda(\gamma(t))
$$
for the penultimate equality, and the pull-back formula for the contact Hamiltonian
$$
\Dev_\lambda(\psi^*X) = e^{g_\psi} \Dev_\lambda(X)
$$
 for the last equality, respectively.
This finishes the proof.
\end{proof}

\begin{rem} In particular, the asymptotic limit $\gamma$ of any finite $\pi$-energy
solution of \eqref{eq:perturbed-contacton-bdy} satisfies
$$
(\dot \gamma - X_H(t,\gamma(t)))^{\pi} = 0.
$$
It is shown in \cite{oh:entanglement1} that the asymptotic limit indeed of the form
$$
t \mapsto \psi_H^t \circ (\psi_H^1)^{-1}\circ \psi_{R_\lambda}^t(\psi_H^1(p))
$$
for some point $p \in R_0$.
\end{rem}

\section{Perturbed contact instantons, $\pi$-energy and gauge transformation}
\label{sec:conversion}

Let $(M,\lambda,J)$ be a contact triad and
let $R_0, \, R_1$ be a pair of Legendrian submanifolds in $(M,\xi)$.
Let  a contact Hamiltonian $H= H(t,x)$ be given.

By the similar spirit considered in \cite[Section 2.5 \& 2.6]{abouzaid-seidel}, we make the following
choice of one-form $\gamma$ on $\dot \Sigma$. In the present paper, we do not need to be
as specific as therein at least for the main purpose of the present paper.

\begin{cond}[One-form $\gamma$]\label{cond:gamma}
Let $\epsilon^j: \pm [0,\infty, \infty) \to \dot \Sigma$ be
the map defining the given strip-like coordinate $(\tau,t)$  at each puncture $z_j \in \Sigma$ of $\dot \Sigma$.
We require $\gamma$ to satisfy the following requirement:
\begin{enumerate}
\item $\gamma|_{\del \dot \Sigma} = 0$,
\item  $(\epsilon^j)^*\gamma = dt$ for all $j$.
\end{enumerate}
\end{cond}

The following is the definitions of the main objects of the study of the present paper.

\begin{defn}
\begin{enumerate}
\item A \emph{Hamiltonian perturbed contact Cauchy--Riemann map}
from $\dot\Sigma$ to a contact triad $(M, \lambda, J)$ is a smooth map $u: \dot\Sigma\to M$ so that
$$
\delbar_H^\pi u: = (du - X_H(t,u))^{\pi(0,1)} =0.
$$
\item A \emph{Hamiltonian perturbed contact instanton}
is a $H$-perturbed contact CR map $u: \dot\Sigma\to M$ satisfying
the equation $ d(e^{g_{H,u}}(u^*\lambda_H \circ j)) = 0$ in addition, i.e., a map
satisfying the following system of equations
$$
\delbar_H^\pi u = 0, \quad d(e^{g_{H,u}}(u^*\lambda_H \circ j)) = 0.
$$
\end{enumerate}
\end{defn}

The choice of correct definition of the $\pi$-energy for the equations \eqref{eq:perturbed-contacton-bdy}
is rather subtle.  To motivate our definition, we consider the contact counterparts
for the relevant materials from the Lagrangian Floer theory in symplectic geometry.

For a given contact Hamiltonian $H \mapsto \psi$, i.e., $\psi = \psi_H^1$,
we denote by  $\Omega(\psi(R_0),R_1 )$ the set of smooth path
space between the pair $(\psi(R_0), R_1)$.

We recall that the assignment
$$
\ell(t) \mapsto(\psi_H^t(\psi_H^1)^{-1})^{-1}(\ell'(t))
$$
defines a bijective map
\be\label{eq:PhiH}
\Phi_H:\CL(\psi_H^1(R_0) ,R_1) \to \CL(R_0,R_1).
\ee
In particular the translated intersection point of the contact instanton complex for the pair
$(R_0, R_1)$ is generated by the set of Reeb chords
$$
{\mathfrak R}eeb(\psi_H^1(R_0) ,R_1)
$$
between the pair $((\psi_H^1(R_0) ,R_1)$. These curves defines the Floer-type complex
whose boundary map is constructed by the moduli space of unperturbed contact instanton
equation \eqref{eq:perturbed-contacton-bdy}.

On the other hand, the dynamical version of the complex is generated by
some set of solutions of Hamilton's equation $\dot x = X_H(t,x)$
and its boundary map is constructed by the moduli space of \eqref{eq:perturbed-contacton-bdy}.
(See Proposition \ref{prop:Ham=Reeb}  below for the
details of this correspondence.)
These two frameworks are related by the bijective map $\Phi_H$
via the correspondence
\be\label{eq:gauge1}
\ell(t) = \psi_H^t ((\psi_H^1)^{-1}(\ell'(t))), \quad
u(s,t)=\psi_H^t((\psi_H^1)^{-1}(\overline u(s,t))).
\ee
\begin{rem}
Our motivation for this choice also comes from the well-established coordinate change
between the \emph{dynamical version} and the \emph{intersection theoretic
version} of the Floer homology in symplectic geometry.
This has been systematically utilized by the present author
in the symplectic Floer theory. (See \cite{oh:jdg,oh:cag,oh:dmj}
and the book \cite[Section 12.7]{oh:book2}.)
\end{rem}

To make the aforementioned correspondence precise in the current contact case,
we need to take the CR almost complex structure into consideration
arising from a contact triad $(M,\lambda,J)$ as in \cite{oh-wang:connection}.

\begin{choice}\label{choice:LJt}
For an given contact Hamiltonian $H = H(t,x)$ and a one-parameter family
$J = \{J_t\}$ of CR-almost complex structures adapted to $\lambda$, we
consider another family $J_t'$ defined by the relation
\be\label{eq:Jt}
J = \{J_t\}_{0 \leq t \leq 1}, \quad J_t:= ((\psi_H^t(\psi_H^1)^{-1})^{-1})^*J_t'
\ee
of $\lambda$-admissible almost complex structures.
\end{choice}

Now we recall the aforementioned gauge transformation $\Phi_H^{-1}$ to $u$ and define $\overline u : = \Phi_H^{-1}(u)$
which has the expression
\be\label{eq:ubar}
\overline u(\tau,t) = (\psi_H^t(\psi_H^1)^{-1})^{-1}(u(\tau,t))
= \psi_H^1(\psi_H^t)^{-1}(u(\tau,t)).
\ee
Then the following equivalence of the two equations,
which is the reason why we adopt the equation \ref{eq:perturbed-contacton-bdy} as
the correct form of \emph{Hamiltonian perturbed contact instanton equation}.

\begin{prop}\label{prop:Ham=Reeb} Let $J_0 \in \CJ(\lambda)$ and $J_t$ defined as
in \eqref{eq:Jt}. Let $\dot \Sigma \cong \R \times [0,1]$ and let $g_{H,u}$ be the
conformal exponent function defined as above.
Then $u$ satisfies
\be\label{eq:perturbed-contacton-bdy}
\begin{cases}
(du - X_H \otimes dt)^{\pi(0,1)} = 0, \quad d((e^{g_{H, u}}(u^*\lambda + H\, dt)\circ j) = 0\\
u(\tau,0) \in R_0, \quad u(\tau,1) \in R_1
\end{cases}
\ee
with respect to $J_t$ if and only if $\overline u$ satisfies
\be\label{eq:perturbed-intersection}
\begin{cases}
\delbar^\pi \overline u = 0, \quad d(\overline u^*\lambda \circ j) = 0 \\
\overline u(\tau,0) \in  \psi_H^1(R_0), \, \overline u(\tau,1) \in R_1
\end{cases}
\ee
where $\delbar^\pi = \delbar_{J'}^\pi$ associated to $J' = \{J_t'\}$ defined in Choice \ref{choice:LJt}.
\end{prop}

Postponing its full proof  till Appendix \ref{sec:conversion}, we just mention the following:
Consider the transformation from $u$ to $\overline u$ which is given by
\be\label{eq:u-to-w}
\overline u(\tau,t) = (\psi_H^t(\psi_H^1)^{-1})^{-1}(u(\tau,t)).
\ee
Then if $u$ is satisfies \eqref{eq:perturbed-contacton-bdy}
with respect to $J_t$, then $\overline u$ is a contact instanton with respect to $J_0$ with the boundary condition
$$
\overline u(\tau,0) \in \psi_H^1(R_0), \, \overline u(\tau,1) \in R_1.
$$

\begin{ques}
What would be the correct \emph{off-shell} choice of energy for perturbed contact instantons
equation \eqref{eq:contacton-Legendrian-bdy} for a \emph{general pair} $(H,J)$
with $J = \{J_t\}$ an arbitrary family of $\lambda$-adapted CR-almost complex structures?
\end{ques}

It turns out that the answer is  the following.

\begin{defn}[The $\pi$-energy of perturbed contact instanton]\label{defn:energy-dynamical}
Let $u: \R \times [0,1] \to M$ be any smooth map. We define
$$
E^\pi_{J,H}(u): = \frac12 \int e^{g_{H, u}}|(d^\pi u - X_H^\pi (u) \otimes \gamma)^\pi|^2_J.
$$
\end{defn}

Then we have the following energy identity between the maps satisfying
\eqref{eq:perturbed-contacton-bdy} and those satisfying \eqref{eq:perturbed-intersection},
when $J=\{J_t\}$ is the one given by \eqref{eq:Jt}.

\begin{prop}[Proposition 7.2 \cite{oh:entanglement1}]
\label{prop:EpiJH=Epi} Let $J' = \{J_t'\}$ be as in \eqref{eq:Jt}.
For any smooth map $u: \R \times [0,1] \to M$. let $\overline u$ be as above. Then
\be\label{eq:EpiJH=Epi}
E^\pi_{J,H}(u) = E^\pi_{J'}( \overline u).
\ee
\end{prop}

The following proposition is one of the key
energy estimates for the solutions $u$ of \eqref{eq:perturbed-contacton-bdy}
in terms of the geometry of Legendrian boundary conditions and its asymptotic chords.
This is an immediate generalization of \cite[Proposition 7.2]{oh:entanglement1} which considers only the case
with the case of constant family $J_t' \equiv J_0$.
For readers' convenience, we provide its proof in Appendix \ref{sec:conversion}.

\begin{prop}[Proposition 7.4 \cite{oh:entanglement1}]\label{prop:pienergy-bound}
Let $\phi_H^t $ be as in \eqref{eq:phiHt}.
Let $u$ be any finite energy solution of \eqref{eq:perturbed-contacton-bdy-intro}
associated to the pair $(H,J)$ as in \eqref{eq:Jt} with the asymptotic limits
$$
\gamma_\pm(t) := \lim_{\tau \to \pm \infty} u(\tau,t).
$$
Let  $\overline u$ be the map defined  as above and consider the paths given by
$$
\overline \gamma_\pm(t) = (\phi_H^t)^{-1}(\gamma_\pm(t)).
$$
Then $\overline \gamma_\pm$ are Reeb chords from $\psi_H^1(R_0)$ to $R_1$ and satisfy
\be\label{eq:pi-energy}
E^\pi_{J,H}(u) = \int_0^1 (\overline \gamma_+)^*\lambda - \int_0^1 (\overline \gamma_-)^*\lambda.
\ee
\end{prop}

Now we give the crucial action identity for the energy which provides the gradient structure of
the perturbed contact instanton equation.

\begin{thm}\label{thm:energy-action}
Let $\phi_H^t $ be as in \eqref{eq:phiHt}. Consider a general pair $(H,J)$.
Let $u$ be any finite energy solution of \eqref{eq:perturbed-contacton-bdy-intro}
associated to the $(H,J)$  with the asymptotic limits
$$
\gamma_\pm(t) := \lim_{\tau \to \pm \infty} u(\tau,t).
$$
Let  $\overline u$ be the map defined  as above and consider the paths given by
$$
\overline \gamma_\pm(t) = (\phi_H^t)^{-1}(\gamma_\pm(t)).
$$
Then $\overline \gamma_\pm$ are Reeb chords from $\psi_H^1(R_0)$ to $R_1$ and satisfy
\be\label{eq:energy-action}
E^\pi_{J,H}(u) = \CA_H(\gamma_+) - \CA_H(\gamma_-).
\ee
\end{thm}
\begin{proof} We first generalize the identity \eqref{eq:EpiJH=Epi} to any general pair
$J= \{J_t\}$ and $J' = \{J'_t\}$ satisfying
\be\label{eq:Jt-J't}
J = \{J_t\}_{0 \leq t \leq 1}, \quad J_t:= ((\psi_H^t(\psi_H^1)^{-1})^{-1})^*J_t'.
\ee
In fact an examination of the proof of \cite[Proposition 7.2]{oh:entanglement1} actually shows
the same inequality for any pair $J, \m J'$ that satisfy this relation.

Therefore, for the study of energy identity \eqref{eq:energy-action}, it is enough to compute $E^\pi_J(\overline u)$.
By the same proof of \cite[Proposition 7.4]{oh:entanglement1}, we compute, using the equation
 $Jd\overline u(\del_\tau) = d\overline u(\del t)$,
\beastar
&{}&
\int_{-\infty}^\infty \int_0^1 |d^\pi \overline u(\del_\tau)|^2\, dt\, d\tau \\
&=& \int_{-\infty}^\infty \int_0^1 d\lambda(d\overline u(\del_\tau), J d\overline u(\del_\tau))\, dt\, d\tau
= \int_{-\infty}^\infty \int_0^1 d\lambda(d\overline u(\del_\tau), d\overline u(\del_t))\, dt\, d\tau\\
& = & \int (\overline u)^*d\lambda \\
& = & \left(\int_0^1 ({\overline \gamma_+})^*\lambda - \int_0^1 ({\overline \gamma_-})^*\lambda \right)
+ \int_{-\infty}^\infty \lambda\left(\frac{\del \overline u}{\del \tau}(\tau,0) \right)\, d\tau -
\int_{-\infty}^\infty \lambda\left(\frac{\del \overline u}{\del \tau}(\tau,1) \right)\, d\tau\\
& = & \int_0^1 ({\overline \gamma_+})^*\lambda - \int_0^1 ({\overline \gamma_-})^*\lambda.
\eeastar
Here the last equality follows from the Legendrian boundary condition
$$
\overline u(\tau,0) \in \psi_H^1(R), \quad \overline u(\tau,1) \in R
$$
for $i = 0, \, 1$. On the other hand, Lemma \ref{lem:action-identity} shows that the last line is the same as
$$
\CA_H(\gamma_+) - \CA_H(\gamma_-).
$$
This finishes the proof.
\end{proof}

\section{Weitzenb\"ock formula for the $\pi$-energy density function}

In this section, we start with the study of moduli spaces of perturbed contact instantons
by establishing local regularity theory thereof. For this purpose, we utilize the
Weitzenb\"ock formula by extending the calculations of the second covariant differential and
the Laplacian of $\pi$-harmonic energy
carried out in \cite{oh-wang:connection} to the  Hamiltonian perturbed contact Cauchy--Riemann maps.
This formula will be a fundamental ingredient for the regularity estimate
and other analytic study of perturbed contact instantons.

Let $\lambda$ be a contact form associated to the contact manifold $(M,\xi)$.
We decompose the total (perturbed) energy density function similarly as in
\cite[p.661]{oh-wang:CR-map1}. For this purpose, we consider a general
$u^*TM$-valued one-form $\eta$ on $\dot \Sigma$. Using the decomposition
$$
Y = Y^\pi + \lambda(Y) R_\lambda
$$
we can decompose $\eta = \eta \circ \pi + \langle \eta, R_\lambda\rangle \, \lambda$
which is an orthogonal direct sum. Therefore we have
$$
|\eta|^2 = |\eta^\pi|^2 + |\langle \eta,R_\lambda \rangle|^2.
$$
We apply this to $\eta = d_Hu$. This leads us to the following definition which is a direct generalization of
\cite[p.661]{oh-wang:CR-map1}.

\begin{defn}[Total energy density function]
We define
$$
e_H(u) = e_H^\pi(u) + |u^*\lambda_H|^2 = |(du - X_H \otimes \gamma)^\pi|^2 + |u^* \lambda
+ H\, \gamma|^2.
$$
\end{defn}

Next It is crucial to compute the covariant derivative $\nabla(d_Hu)$ and the
Hodge-Laplacian $\Delta d_Hu$ to apply celebrated Bochner-Weitzenb\"ock formula.
Before actually carrying out this computation, we make some preliminary simplification of
the upcoming rather complex tensorial computations, utilizing the defining properties of
contact triad connection which is summarized in Appendix \ref{sec:connection}.

\subsection{Preliminary simplifications for tensorial calculation}

We start with the following, which crucially relies on the property of
contact triad connection introduced in \cite{oh-wang:connection}.
(See Appendix \ref{sec:connection} for a summary of its defining properties.)

\begin{lem}\label{lem:nablaRlambda}
Let $u:\dot \Sigma \to M$ be any $H$-perturbed CR map and
$\alpha \in \Omega^{\pi(0,1)}(u^*TM)$. Then
\bea
\langle d_Hu, R_\lambda \rangle & = & u^*(\lambda + H \gamma) \label{eq:dHuRlambda}\\
\langle \alpha, \nabla R_\lambda \rangle & = & 0 \label{eq:alpha,nablaRlambda}
\eea
\end{lem}
\begin{proof} We compute
\beastar
\langle d_H u, R_\lambda \rangle & = &
\langle du - u^*X_H \otimes \gamma \lambda), R_\lambda \rangle \\
& = & \langle du, R_\lambda \rangle -  \langle X_H,R_\lambda \rangle \gamma =
u^*\lambda + u^*(H) \otimes \gamma.
\eeastar
This finishes the proof of \eqref{eq:dHuRlambda}, which actually holds for any smooth map, not just for
contact Cauchy-Riemann maps.

For the proof of \eqref{eq:alpha,nablaRlambda}, we first note that Property (3)
of contact triad connection in Appendix \ref{sec:connection} implies that the linear map
$$
\nabla R_\lambda: TM \to TM
$$
satisfies $\nabla_{R_\lambda} R_\lambda = 0$ and that it naturally restricts to the linear map $\nabla R_\lambda: \xi \to \xi$.
Then the property \eqref{eq:dnablaYR}  of contact triad connection therein implies that
the linear map
$$
\nabla_{\Pi(\cdot)} R_\lambda: \xi \to TM
$$
is of $(1,0)$-type. On the other hand, since $d_H^\pi u$
is $(1,0)$-type for any $H$-perturbed contact CR map, we conclude
$$
\langle d_H^\pi u, \nabla_{d^\pi u} R_\lambda \rangle = 0
$$
and so is $\langle \alpha, \nabla_{du} R_\lambda \rangle$. This finishes the proof.
\end{proof}

Next we derive the following identity.

\begin{prop}\label{prop:nabladHu} For any perturbed Cauchy-Riemann map, we have
\be\label{eq:nabladHu}
\nabla(d_Hu) = \nabla^\pi (d^\pi_H u) + \nabla(u^*\lambda_H) R_\lambda + u^*\lambda_H \nabla R_\lambda.
\ee
\end{prop}
\begin{proof} We start from the decomposition
$$
d_Hu = d_H^\pi u + u^*\lambda_H \otimes R_\lambda
$$
and $du = d^\pi u + u^*\lambda R_\lambda$.
Therefore we have
\be\label{eq:nabladHu-temp}
\nabla(d_H u) = \nabla(d_H^\pi u) + \nabla (u^*\lambda_H R_\lambda) = \nabla (d_H^\pi u) + \nabla (u^*\lambda_H) R_\lambda
+ u^*\lambda_H \nabla R_\lambda
\ee
Recalling $\nabla$ in $\nabla(d_H u)$ is the
pull-back connection of the contact triad connection $\nabla$ on $(M,\lambda, J)$,
we actually have
$$
\nabla_e (d_H^\pi u) = \nabla_{du(e)}(d_H^\pi  u)
$$
for any $e \in T\dot \Sigma$. We have
$$
\nabla_{du(e)}(d_H^\pi u) = \nabla_{d^\pi u(e)} (d_H^\pi u )
+ \nabla_{u^*\lambda(e) R_\lambda} (d_H^\pi  u).
$$
We will now prove
\be\label{eq:nabladHu,Rlambda=0}
\langle \nabla_{u^*\lambda(e) R_\lambda} (d_H^\pi  u), R_\lambda \rangle = 0.
\ee
Since $\langle d_H^\pi u, R_\lambda \rangle = 0$ and $\nabla$ preserves the triad metric, we have
$$
\langle \nabla_{d^\pi u(e)} (d_H^\pi u ), R_\lambda \rangle
= - \langle d_H^\pi u, \nabla_{d^\pi u(e)} R_\lambda \rangle.
$$
On the other hand,  $d_H^\pi u$ is of type $(1,0)$ since $u$ is assumed to be a $H$-perturbed CR map.
Therefore we derive
$$
\langle d_H^\pi u,\nabla_{du} R_\lambda \rangle = 0
$$
by \eqref{eq:alpha,nablaRlambda}. This proves \eqref{eq:nabladHu,Rlambda=0} which in turn implies
$$
\nabla_{du} (d_H^\pi u) = \nabla_{d^\pi u} (d_H^\pi u) = \nabla^\pi_{du} (d_H^\pi u).
$$
Substituting this into \eqref{eq:nabladHu-temp}, we have finished the proof.
\end{proof}

Taking the norm square of \eqref{eq:nabladHu}, we have obtained the following upper bound for
the $|\nabla (d_H u)|^2$.

\begin{cor} Let $u$ be any $H$-perturbed CR map. Then we have
\be\label{eq:|nabladHu|2}
|\nabla (d_H u)|^2 \leq |\nabla^\pi (d^\pi_H u)|^2 + |\nabla(u^*\lambda_H)|^2 + \|\nabla R_\lambda\|^2_{C^0}|u^*\lambda||du|
\ee
where $\nabla R_\lambda$ is regarded as a linear map $\nabla R_\lambda : TM \otimes \xi \to \xi$.
\end{cor}

\begin{rem} This corollary reduces the problem of estimating the full derivative $\nabla (d_Hu)$
to a much simpler problem of estimating the $\xi$-component $\nabla^\pi (d_H^\pi u)$ and
the Reeb component $\nabla (u^*\lambda_H)$ separately, and the rather trivial problem
of estimating the  term $\|\nabla R_\lambda\|_{C^0}^2 |u^*\lambda||du|$.
\end{rem}

Keeping this proposition and the remark above in our mind,
we apply the standard Weitzenb\"ock formula to the horizontal component
$$
d_H^\pi u: = (du - X_H \otimes \gamma)^\pi = \Pi(du - X_H \otimes \gamma)
$$
of the vector bundle $u^*\xi \to \dot \Sigma$, and get the following basic
Bochner-Weitzenb\"ock identity as the first step similarly as in \cite[Equation (4.1)]{oh-wang:connection}

We  start with looking at the (Hodge) Laplacian
of the $\pi$-harmonic energy density $e_H^\pi (u)$ of an arbitrary smooth map
$u: \dot \Sigma \to M$ in the \emph{off-shell level}.

\begin{lem} Let $u$ be any smooth map. Then
\bea\label{eq:bochner-weitzenbeck-e}
-\frac{1}{2}\Delta e_H^\pi(u)&=&|\nabla^\pi(d_H^\pi u) |^2-\langle \Delta^{\nabla^\pi} d_H^\pi u, d_H^\pi u\rangle
\nonumber \\
&{}&  +K\cdot |d_H^\pi u|^2+\langle \ric^{\nabla^\pi}(d_H^\pi u), d_H^\pi u\rangle.
\eea
\end{lem}

Here $e_H^\pi:=e_H^\pi(u)$, $K$ is the Gaussian curvature of $(\dot\Sigma,h)$,
and $\ric^{\nabla^\pi}$ is the Ricci tensor of the connection $\nabla^\pi$
on the vector bundle $u^*\xi$ and
$$
\Delta^{\nabla^\pi} = \delta^{\nabla^\pi} d^{\nabla^\pi} + d^{\nabla^\pi} \delta^{\nabla^\pi}
$$
is the covariant Hodge Laplacian.
(See the proof of \cite[Appendix A]{oh-wang:CR-map1} for the proof.
For the basic differential notations, such as $d^{\nabla^\pi}$, $\delta^{\nabla^\pi}$ etc.,
we also refer readers to  \cite[Appendix B]{oh-wang:CR-map1}. For readers' convenience,
we collect basic lemmata on the operators in Appendix \ref{sec:vectorvalued-forms}.)

\subsection{Calculations for the Hodge Laplacian of $d_H^\pi u$}

According to the Bochner-Weitzenb\"ock identity \eqref{eq:bochner-weitzenbeck-e}, it is crucial to
derive an explicit formula for the term
$$
\langle \Delta^{\nabla^\pi} d_H^\pi u, d_H^\pi u \rangle
$$
which involves the third order derivatives, and  is of the same order as that of
the left hand side $\Delta e^\pi_H$. For the simplification of notations, we may
sometimes drop the superindex $\nabla^\pi$ from $\Delta^{\nabla^\pi}$ which should not
confuse readers at all.

For this purpose, we first
transfer the following important formula for $d^{\nabla^\pi}(d^\pi u)$ in the off-shell level
from \cite{oh-wang:connection}.

\begin{lem}[Lemma 4.1 \cite{oh-wang:connection}]\label{lem:fundamental}
Let $w: \dot\Sigma \to M$ be \emph{any smooth} map. Denote by $T^\pi$ the torsion tensor of $\nabla^\pi$.
Then as a two form with values in $u^*\xi$,
$d^{\nabla^\pi} (d^\pi w)$ has the expression
\be\label{eq:dnabladpiw}
d^{\nabla^\pi} (d^\pi w)= T^\pi(\Pi d_Hu, \Pi d_Hw )
+ w^*\lambda \wedge \left(\frac{1}{2} (\CL_{X_\lambda}J)\, Jd^\pi w\right).
\ee
\end{lem}
We warn that readers should not get confused with the wedge product we have used here, which is
the wedge product for forms in the usual sense, i.e., $(\alpha_1\otimes\zeta)\wedge\alpha_2=(\alpha_1\wedge\alpha_2)\otimes\zeta$
for $\alpha_1, \alpha_2\in \Omega^*(P)$ and $\zeta$ a section of $E$.
(See \cite[Appendix 2]{oh-wang:connection}, \cite{wells-book}, e.g.,
for some systematic discussion on this  convention.)

We first derive the following generalization of \eqref{eq:dnabladpiw} which
is an immediate corollary of the above lemma by definition of $d_H^\pi u
= d^\pi u - X_H^\pi \otimes \gamma$.

\begin{cor}\label{cor:FE-autono}
Let $u: \dot\Sigma \to M$ be any smooth map. Denote by $T^\pi$ the torsion tensor of $\nabla^\pi$.
Then as a two form with values in $w^*\xi$,
$d^{\nabla^\pi} (d_H^\pi w)$ has the expression
\be\label{eq:dnabladpiw}
d^{\nabla^\pi} (d_H^\pi u)= T^\pi(\Pi du, \Pi du)+ u^*\lambda \wedge \left(\frac{1}{2} (\CL_{X_\lambda}J)\, Jd^\pi u \right) - d^{\nabla^\pi}(X_H^\pi(u) \otimes \gamma).
\ee
\end{cor}

Using Corollary \ref{cor:FE-autono}, we derive the following fundamental equation
for the perturbed Cauchy-Riemann maps which is the perturbed analog to
\cite[Theorem 4.2]{oh-wang:CR-map1}.

For the simplicity of notation, we often write $X_H \, \gamma$ in place of $X_H \otimes \gamma$
in the calculations we do in the rest of the paper,
as long as there is no danger of confusion.

\begin{thm}[Fundamental equation]\label{thm:fundamental}
 Suppose that $\gamma$ is a one-form on $\dot \Sigma$ and let $u$ be a perturbed contact
 Cauchy-Riemann map. Then we have
\bea\label{eq:dd_Hu2}
d^{\nabla^\pi}(d_H^\pi u) & = & u^*\lambda\wedge(\frac{1}{2}\left(\CL_{R_\lambda} J)J d_H^\pi u\right) -
2 T^\pi(X_H^\pi(u), \gamma \wedge d_H^\pi u) \nonumber \\
&{}& +  u^*\lambda\wedge(\frac{1}{2}\left(\CL_{R_\lambda} J)JX_H^\pi(u)\otimes \gamma\right)
- d^{\nabla^\pi} (X_H^\pi(u) \otimes \gamma).
\eea
\end{thm}
\begin{proof}
Since  $d_H^\pi u = d^\pi u - X_H^\pi  \otimes \gamma$
is assumed to be of $(1,0)$-type by hypothesis and $T^\pi$ is of $(0,2)$-type, we obtain
\beastar
0 & = & T^\pi(d^\pi u - X_H^\pi(u)\, \gamma, d^\pi u- X_H^\pi(u)\, \gamma) \\
& = & T^\pi(d^\pi u, d^\pi u) - T^\pi(X_H^\pi(u)\, \gamma, d^\pi u)- T^\pi(d^\pi u, X_H^\pi(u)\, \gamma)\\
&{}& \quad  +  T^\pi(X_H^\pi(u)\, \gamma, X_H^\pi(u)\, \gamma).
\eeastar
in \eqref{eq:dd_Hu2}.
Obviously, we have $ T^\pi(X_H^\pi(u)\, \gamma, X_H^\pi(u)\, \gamma) = 0$ by the skew-symmetry of $T^\pi$ applied to the $u^*\xi$-valued one-form proportional to $\gamma$.

Next we evaluate
$$
\left(T^\pi(X_H^\pi(u)\, \gamma, d^\pi u) + T^\pi(d^\pi u, X_H^\pi(u)\, \gamma)\right)(\del_x,\del_y)
$$
for an isothermal coordinates $(x,y)$. This becomes
\beastar
&{}& 2 T^\pi\left(\Pi \gamma_x X_H, \Pi\frac{\del u}{\del y}\right)
- 2 T^\pi\left(\Pi \gamma_y X_H, \Pi \frac{\del u}{\del x}\right)\\
& = & 2 T^\p\left(\Pi X_H, \gamma_x \Pi \frac{\del u}{\del y}\right) + 2 T^\pi \left(\Pi X_H,\gamma_y
 \Pi \frac{\del u}{\del x}\right)\\
& = & 2\left(T^\pi(\Pi X_H, \gamma_x \Pi \frac{\del u}{\del y}- \gamma_y  \Pi \frac{\del u}{\del x}\right)\\
& = & 2T^\pi(\Pi X_H, \gamma \wedge d^\pi u)(\del_x,\del_y).
\eeastar
Therefore we have obtained
\beastar
T^\pi(d^\pi u,d^\pi  u) & = & T^\pi(X_H^\pi(u)\, \gamma, d^\pi u) + T^\pi(d^\pi u, X_H^\pi(u)\, \gamma)\\
& = & 2T^\pi(\Pi X_H, \gamma \wedge d^\pi u) \\
& = & 2T^\pi(\Pi X_H, \gamma \wedge (d_H^\pi u + X_H \otimes \gamma)) \\
& = & 2T^\pi(\Pi X_H, \gamma \wedge (d_H^\pi u)
\eeastar
where the last equality holds by the vanishing
 $T^\pi(\Pi X_H, \gamma \wedge X_H \otimes \gamma) = 0$.

Finally using $du = d_H u + X_H(u) \otimes \gamma$, we rewrite
$$
u^*\lambda \wedge \left(\frac{1}{2} (\CL_{X_\lambda}J)\, Jd^\pi u \right)
=  u^*\lambda\wedge(\frac{1}{2}\left(\CL_{R_\lambda} J)J d_H^\pi u\right)  +
u^*\lambda\wedge(\frac{1}{2}\left(\CL_{R_\lambda} J)JX_H^\pi(u)\otimes \gamma\right).
$$
Substituting the last two into \eqref{eq:dnabladpiw}, we have finished the proof.
\end{proof}

For the study of \eqref{eq:bochner-weitzenbeck-e}, it is fundamental to estimate the inner product
$$
\langle \Delta^{\nabla^\pi} d_H^\pi u, d_H^\pi u\rangle
$$
for perturbed contact Cauchy--Riemann maps $u$ for which $d_Hu$ is of $(1,0)$-type, i.e.,
$$
d_H^\pi u = \del_H^\pi u.
$$
For this purpose, we need the following crucial lemma.

\begin{lem} \label{lem:2delta}
Let $\alpha  \in\Omega^{(1,0)}(u^*\xi)$ be any $u^*\xi$-valued $(1,0)$-form.
Then we have
\be\label{ddelta=deltad}
\langle d^{\nabla^\pi}\delta^{\nabla^\pi} \alpha,\alpha \rangle =
\langle \delta^{\nabla^\pi}d^{\nabla^\pi}\alpha,\alpha \rangle.
\ee
\end{lem}
\begin{proof}
  Using the formula $\delta^{\nabla^\pi} = -* d^{\nabla^\pi}*$ when acting on 2-forms and the fact
that $*$ is an isometry,  we have
$$
\langle \delta^{\nabla^\pi}d^{\nabla^\pi}  \alpha, \alpha\rangle= -\langle *d^{\nabla^\pi}  * d^{\nabla^\pi}  \alpha, \alpha\rangle
$$
Then we use $*\alpha=-\alpha\circ j$ for any $1$-form $\alpha$ to rewrite the last term into
$$
-\langle d^{\nabla^\pi} * d^{\nabla^\pi}  \alpha, *\alpha\rangle.
$$
Then since $\alpha$ is complex-linear and the connection $\nabla^\pi$ is $J$-linear, this becomes
\bea
-\langle d^{\nabla^\pi} * d^{\nabla^\pi}  \alpha, -\alpha\circ j\rangle
&=&\langle d^{\nabla^\pi} * d^{\nabla^\pi}  \alpha, J \alpha\rangle\nonumber\\
&=&-\langle J d^{\nabla^\pi} * d^{\nabla^\pi}  \alpha,  \alpha\rangle\nonumber\\
&=&-\langle d^{\nabla^\pi} * d^{\nabla^\pi}  J\alpha,  \alpha\rangle\nonumber\\
&=&-\langle d^{\nabla^\pi}  * d^{\nabla^\pi}  \alpha\circ j,  \alpha\rangle\nonumber\\
&=&\langle d^{\nabla^\pi} * d^{\nabla^\pi} * \alpha,  \alpha\rangle\label{eq:dstard5}\\
&=&\langle d^{\nabla^\pi}  \delta^{\nabla^\pi}\alpha, \alpha\rangle.\nonumber
\eea
Here for \eqref{eq:dstard5}, we again use $* \alpha = -\alpha \circ j$ for any one-form $\alpha$.
This finishes the proof.
\end{proof}

Applying this lemma to $\alpha = \del_H^\pi u$, we immediately obtain
the following which is the analog \cite[Lemma 4.5]{oh-wang:connection} with $\del_H^\pi u$
replaced by $\del^\pi u$.

\begin{cor}[Compare with Lemma 4.5 \cite{oh-wang:connection}]
\label{cor:2delta}
For any smooth map $u$, we have
\beastar
\langle d^{\nabla^\pi}  \delta^{\nabla^\pi}\del_H^\pi u, \del_H^\pi u\rangle
=\langle \delta^{\nabla^\pi}d^{\nabla^\pi} \del_H^\pi u, \del_H^\pi u\rangle.
\eeastar
As a consequence,
\bea
\langle \Delta^{\nabla^\pi} \del_H^\pi u, \del_H^\pi u\rangle
=2\langle \delta^{\nabla^\pi}d^{\nabla^\pi}  \del_H^\pi u, \del_H^\pi u\rangle.\label{eq:2Delta}
\eea
\end{cor}

The following lemmata express $\langle\Delta^{\nabla^\pi}d_H^\pi u, d_H^\pi u\rangle$,
which involves the third derivative of $u$,
in terms of the expressions involving derivatives of order at most two
for any $H$-perturbed contact CR map $u$.
This is a generalization of
\cite[Lemma 4.6]{oh-wang:connection}.

\begin{lem}[Compare with Lemma 4.6 \cite{oh-wang:connection}] \label{lem:laplaceproduct}
Let $u$ be a $H$-perturbed contact CR map. Then
we have
\bea\label{eq:Laplaciandpiw}
-\langle \Delta^{\nabla^\pi}d_H^\pi u, d_H^\pi u\rangle
& = & - \delta^{\nabla^\pi}[(u^*\lambda \wedge (\CL_{X_\lambda}J)J \del_H^\pi u], d_H^\pi u\rangle\nonumber\\
&{}& + \langle \delta^{\nabla^\pi}[ u^*\lambda\wedge (\CL_{R_\lambda} J)JX_H^\pi(u) \otimes \gamma],
d_H^\pi u \rangle \nonumber \\
&{}& - 4 \langle \delta^{\nabla^\pi} [T^\pi(X_H^\pi(u), \gamma \wedge d_H^\pi u)], d_H^\pi u\rangle \nonumber\\
&{}& - 2 \langle d^{\nabla^\pi}X_H^\pi \wedge \gamma, \del_H^\pi u\rangle -
2 \langle X_H^\pi \wedge d\gamma, \del_H^\pi u\rangle.
\eea
\end{lem}
\begin{proof} Since $(du - X_H \otimes \gamma)^{\pi(0,1)} = 0$,
we have $d_H^\pi u = \del_H^\pi u$. Then
by applying $\delta^{\nabla^\pi}$ to \eqref{eq:dd_Hu2} and then combining
\eqref{eq:2Delta} and the formula
$$
d^{\nabla^\pi}(X_H^\pi(u) \otimes \gamma) = d^{\nabla^\pi}(X_H(u)) \otimes \gamma + X_H^\pi(u) \otimes
d\gamma,
$$
we finish the proof of the lemma.
\end{proof}

Here in the above lemmata, $\langle\cdot, \cdot\rangle$ denotes the inner product induced from $h$, i.e.,
$\langle\alpha_1\otimes\zeta, \alpha_2\rangle:=h(\alpha_1, \alpha_2)\zeta$,
for any $\alpha_1, \alpha_2\in \Omega^k(P)$ and $\zeta$ a section of $E$.
(See \cite[Appendix]{oh-wang:connection}, \cite{wells-book} for some detailed discussion on such convention.)

\begin{thm}[Laplacian of the $\pi$-energy density]\label{thm:e-pi-weitzenbeck}
Let $u$ be any $H$-perturbed contact CR map.
Then we have
\bea\label{eq:e-pi-weitzenbeck}
-\frac{1}{2}\Delta e_H^\pi(u)&=&|\nabla^\pi (\del_H^\pi u)|^2+K|\del_H^\pi u|^2+\langle \ric^{\nabla^\pi} (\del_H^\pi u), \del_H^\pi u\rangle\nonumber\\
&{}&- \underbrace{\langle \delta^{\nabla^{\pi}}[u^*\lambda \wedge (\CL_{X_\lambda}J)J \del_H^\pi u],
\del_H^\pi u\rangle}_{(A)} \nonumber\\
&{}& + \underbrace{\langle \delta^{\nabla^\pi}[ u^*\lambda\wedge(\CL_{R_\lambda} J)JX_H^\pi(u)
\otimes \gamma], \del_H^\pi u\rangle}_{(B)} \nonumber\\
&{}&- 4 \underbrace{\langle \delta^{\nabla^\pi} [T^\pi(X_H^\pi(u), \gamma \wedge d_H^\pi u)],
d_H^\pi u\rangle}_{(C)}
\nonumber\\
&{}& - 2 \langle d^{\nabla^\pi}X_H^\pi (u)\wedge \gamma, \del_H^\pi u \rangle
- 2 \langle X_H^\pi(u) \wedge d\gamma, \del_H^\pi u\rangle.
\eea
\end{thm}
\begin{proof}  This immediately follows by substituting \eqref{eq:Laplaciandpiw} into
\eqref{eq:bochner-weitzenbeck-e}.
\end{proof}

\subsection{Calculations for the Reeb-component $u^*\lambda_H$}

Recall the notations
$$
\lambda_H = \lambda + H \, \gamma, \quad d_H^\pi u = d^\pi u - X_H \otimes \gamma
$$
by definition.
Again using the Bochner--Weitzenb\"ock formula (applied to differential forms on a Riemann surface),
we get the following identity
\bea\label{eq:e-lambda-weitzenbeck}
-\frac{1}{2}\Delta|u^*\lambda_H|^2 & = & |\nabla u^*\lambda_H|^2
+K|u^*\lambda_H|^2 -\langle \Delta u^*\lambda_H,  u^*\lambda_H\rangle
\eea
where
$$
\Delta u^*\lambda_H=d\delta (u^*\lambda_H)+\delta d(u^*\lambda_H).
$$
We now derive
\begin{prop}\label{prop:Deltau*lambdaH}
\bea\label{eq:Delta-u*lambda+Hgamma}
\Delta (u^*\lambda_H) &  = &  d* (d(g_{H, u}) \wedge u^*\lambda_H  \circ j )\nonumber \\
&{}& + \frac12 *d|(du - X_H \otimes \gamma)^\pi|^2 \nonumber \\
&{}&-* d*(u^*(R_\lambda[H]\, \lambda)\wedge \gamma  + u^*H\, d\gamma)
\eea
\end{prop}
\begin{proof}
By definition, a perturbed contact instanton is a perturbed contact CR map
that also satisfy
\be\label{eq:perturbed-deltaw=0}
d\left(e^{g_{H, u}}(u^*\lambda_H) \circ j\right) = 0
\ee
in addition.
\begin{lem}
\be\label{eq:2du*lambdaH}
 d((u^*\lambda_H) \circ j) = - d(g_{H, u}) \wedge (u^*\lambda_H \circ j).
\ee
\end{lem}
\begin{proof}
By multiplying $e^{-g_{H, u}}$ followed by expanding the left hand side of \eqref{eq:perturbed-deltaw=0},
we derive
\beastar
0 & = &
e^{-g_{H, u}}d\left(e^{g_{H, u}}\left(u^*\lambda_H \circ j\right)\right)  \\
& = &
e^{-g_{H, u}} d(e^{g_{H, u}}) \wedge (u^*\lambda_H \circ j)
+ d(u^*\lambda_H \circ j)
\eeastar
which is equivalent to \eqref{eq:2du*lambdaH}.
\end{proof}

Using the identities $**\alpha = -\alpha$ and
$*\alpha = - \alpha \circ j$ for the one-form $\alpha$, we derive
$$
\delta (u^*\lambda_H) =  - *
d((u^*\lambda_H) \circ j) = *\left(
d(g_{H, u}) \wedge u^*\lambda_H\circ j \right)
$$
from \eqref{eq:2du*lambdaH}. Therefore we obtain
\be\label{eq:ddeltau*lambda}
d \delta (u^*\lambda_H)= d* \left(d(g_{H, u}) \wedge u^*\lambda_H \circ j\right).
\ee

Next we compute $\delta d(u^*\lambda_H)$. For this purpose, we first compute
$d(u^*\lambda_H)$.

\begin{lem}\label{lem:du*lambda}  We have
\be\label{eq:du*lambdaH=}
 d(u^*\lambda_H) = \frac 12 |d_H^\pi u|^2\, dA
 + u^*(R_\lambda[H]\, \lambda)\wedge \gamma  + u^*H\, d\gamma
\ee
\end{lem}
\begin{proof}
Let $(e,je)$ be an orthonormal basis of $T\dot \Sigma$.  We evaluate
\beastar
d(u^*\lambda_H) (e,je) &= & d\lambda(du(e), du(je)) + d(u^*H\, \gamma)(e,je)\\
& = & d\lambda(du(e) - X_H \, \gamma(e),du(je)) - X_H \, \gamma(je))\\
&{}& \quad + d\lambda(du(e), X_H\, \gamma(je)) + d\lambda (X_H\, \gamma(e), du(je)) \\
&{}& \quad + (u^*dH\wedge \gamma)(e,je) + u^*H\, d\gamma(e,je)\\
& = & d\lambda(du(e) - X_H \, \gamma(e),du(je)) - X_H \, \gamma(je))\\
&{}& \quad - d\lambda(du(e), X_H\, \gamma(je)  + d\lambda (X_H, \gamma(e) du(je) ) \\
&{}& \quad + (u^*dH\wedge \gamma)(e,je) + u^*H\, d\gamma(e,je).
\eeastar
Here since $du- X_H \, \gamma$ is of $(1,0)$-type, we obtain
\beastar
&{}&
d\lambda(du(e) - X_H \, \gamma(e),du(je)) - X_H \, \gamma(je))  \\
& = & d\lambda((du - X_H \, \gamma)(e), (du- X_H \, \gamma)(je))\\
& = & d\lambda((du - X_H \, \gamma)(e), J (du- X_H \, \gamma)(e))\\
& = & |du(e) - X_H \, \gamma(e)|^2 = \frac12 |du - X_H \, \gamma|^2.
\eeastar
On the other hand, we can rewrite
\beastar
&{}&
- d\lambda(X_H,\, \gamma(je) du(e)) + d\lambda (X_H, \gamma(e) du(je))\\
& = & - d\lambda(X_H, \, du(e) \gamma(je) -  du(je)  \gamma(e)) \\
& = & - (X_H \rfloor d\lambda)((du \wedge \gamma)(e,je))\\
& = & -u^*(X_H \rfloor  d\lambda) \wedge \gamma(e,je).
\eeastar
Then combining the identity $dH = X_H \rfloor d\lambda + R_\lambda[H]\, \lambda$, we have
\beastar
&{}& - d\lambda(X_H\, \gamma(je) du(e)) + d\lambda (X_H, \gamma(e) du(je)) + (u^*dH\wedge \gamma)(e,je)\\
& = & -\left( u^*(X_H \rfloor d\lambda) \wedge \gamma \right)(e,je) + (u^*dH\wedge \gamma) (e,je)\\
& = & (u^*(R_\lambda[H]\, \lambda) \wedge\gamma) (e,je).
\eeastar
Therefore we derive
\be\label{eq:du*lambdaH}
u^*d\lambda_H = \frac12 |d_H^\pi u|^2\, dA + u^*(R_\lambda[H]\, \lambda)\wedge \gamma  + u^*H\, d\gamma
\ee
which finishes the proof.
\end{proof}

By taking the codifferential $\delta$ to \eqref{eq:du*lambdaH=}, we have derived
\be\label{eq:deltadu*lambda=}
\delta d(u^*\lambda_H) = - \frac 12 * d |d_H^\pi u|^2
 -* d*(u^*(R_\lambda[H]\, \lambda)\wedge \gamma  + u^*H\, d\gamma).
\ee
By adding up \eqref{eq:ddeltau*lambda} and \eqref{eq:deltadu*lambda=}, we have finished the proof of
\eqref{eq:Delta-u*lambda+Hgamma}.
\end{proof}

\section{Fundamental differential inequality for {$\Delta e_H$}}

Our derivation of a priori coercive elliptic estimates relies on Weitzenb\"ock's formula which
provides  a differential inequality for the Laplacian of the perturbed  energy density 
$e_H = e_H^\pi + |u^*\lambda_H|^2$
$$
\Delta e_H(u) = \Delta e_H^\pi(u) + \Delta |u^*\lambda_H|^2.
$$
We will compute  each summand separately below.

\subsection{Differential inequality for {$\Delta e_H^\pi$}}
In this section,
we start from Theorem \ref{thm:e-pi-weitzenbeck} to derive a bound for the $\pi$-energy
$\Delta e_H^\pi = |du - X_H \otimes \gamma|^2$, and prove the following differential inequality.

\begin{prop}\label{prop:Delta-epiu} Let $u$  be a $H$-perturbed contact CR map.
Then for any $c > 0$ we have
\bea
&{}&
-\frac{1}{2}\Delta e_H^\pi(u)\nonumber \\
& \geq & (1 - \frac{5}{2c})|\nabla^\pi (d_H^\pi u)|^2 + (\min K -\|\ric^{\nabla^\pi}\|_{C^0}) | d_H^\pi u|^2 \nonumber\\
&{}& - \frac{1}{c}|\nabla u^*\lambda_H |^2
- \frac1{4c}( \|\nabla H \gamma)\|_{C^0}^2|du|^4 + \|H\|_{C^0}^2 \|\nabla \gamma\|_{C^0}^2|du|^2)  \nonumber \\
&{}& - \frac{c}{2} \|\CL_{X_\lambda}J J\|^2_{C^0(M)} |d_H^\pi u|^4 \nonumber\\
&{}& - \frac{c}{2} \|\CL_{X_\lambda}J J\|_{C^0(M)}^2 \|X_H^\otimes \gamma\|_{C^0}^2
 |d_H^\pi u|^2 \nonumber \\
 &{}& - \frac{c}{4} \|\nabla^\pi(\CL_{X_\lambda}J) J\|_{C^0(M)}(|du|^4 +  |d_H^\pi u|^4) \nonumber\\
&{}&  - \frac12 \left(\|\nabla^\pi(\CL_{X_\lambda}J J)\|_{C^0(M)}\|X_H^\pi\|_{C^0}+ \|\CL_{X_\lambda}J J\|_{C^0(M)}\|\nabla^\pi X_H^\pi \|_{C^0}|\right)\nonumber \\
&{}& \quad \times (|du|^4 + |d_H^\pi u|^2) \nonumber\\
&{}& - 16c (\|T^\pi\|_{C^0}\|X_H^\pi\|_{C^0}\|\gamma \circ j\|_{C^0})^2 \nonumber\\
&{}& - 2 \left(\|\nabla^\pi T^\pi\|_{C^0} \|X_H^\pi \otimes \gamma \|_{C^0}\|\gamma\circ j\|_{C^0}
+\| T^\pi\|_{C^0} \|\nabla^\pi (X_H^\pi \otimes \gamma\|_{C^0}\|\gamma\circ j\|_{C^0}\right)
\nonumber \\
&{}& \quad \times (|du|^2 + |d_H^\pi u|^4)\nonumber\\
&{}& - 4 \|T^\pi\|_{C^0} \|X_H^\pi\|_{C^0} \|\gamma\circ j\|_{C^0}
 |d^{\nabla^\pi}d_H^\pi u| |d_H^\pi u| \nonumber\\
&{}&  - \left(\|\nabla X_H^\pi\|_{C^0} \|\gamma\|_{C^0}
+ \|X_H^\pi\|_{C^0} \|d\gamma\|_{C^0}\right)(|du|^2 + |d_H^\pi u|^4)\nonumber\\
\label{eq:Reeb-contribution}
\eea
\end{prop}
\begin{proof} We need to estimate the terms (A), (B), (C) of \eqref{eq:e-pi-weitzenbeck}
in Theorem \ref{thm:e-pi-weitzenbeck}.

The first line of \eqref{eq:e-pi-weitzenbeck} is bounded below by
$$
|\nabla^\pi(d^\pi_H(u))|^2 + (\min K -\|\ric^{\nabla^\pi}\|_{C^0}) | d_H^\pi u|^2.
$$
We then prove the following two lemmata in  Appendix \ref{sec:A-B}
which rewrite the middle 3 lines of \eqref{eq:e-pi-weitzenbeck} in terms of covariant derivatives.

\begin{lem}\label{lem:A-B}
\bea
(A)& := & \delta^{\nabla^\pi}[(u^*\lambda \wedge (\CL_{X_\lambda}J)J \del_H^\pi u]\nonumber\\
&=&-*\langle (\nabla^\pi(\CL_{X_\lambda}J J))\del_H^\pi u, u^*\lambda\rangle \nonumber\\
&{}&
-*\langle (\CL_{X_\lambda}J J)\nabla^\pi\del_H^\pi u, u^*\lambda\rangle
-*\langle (\CL_{X_\lambda}J J)\del_H^\pi u, \nabla u^*\lambda\rangle.
\label{eq:A} \\
(B) & := & \delta^{\nabla^\pi}[(u^*\lambda \wedge (\CL_{X_\lambda}J)J X_H^\pi(u) \otimes \gamma
]\nonumber\\
&=&-*\langle (\nabla^\pi(\CL_{X_\lambda}J J))J X_H^\pi(u) \otimes \gamma,
u^*\lambda\rangle \nonumber\\
&{}&
-*\langle (\CL_{X_\lambda}J J)\nabla^\pi( J X_H^\pi(u) \otimes \gamma), u^*\lambda\rangle
-*\langle (\CL_{X_\lambda}J J)J X_H^\pi(u) \otimes \gamma, \nabla u^*\lambda\rangle.\nonumber \\
&{}& \label{eq:B}\\
(C)& : =&  \delta^{\nabla^\pi}[T^\pi(X_H^\pi,\gamma \wedge d_H^\pi u)] \nonumber\\
& = &  *(\nabla^\pi T^\pi)(X_H^\pi, \langle d_H^\pi u,\gamma \circ j \rangle)
+ T^\pi(*\nabla^\pi X_H^\pi, \langle d_H^\pi u,\gamma \circ j \rangle) \nonumber\\
&{}& + * T^\pi(X_H^\pi, \langle d^{\nabla^\pi}d_H^\pi u,\gamma \circ j\rangle )
- * T^\pi(X_H^\pi, \langle d_H^\pi u, d(\gamma\circ j) \rangle).
\label{eq:C}
\eea
\end{lem}

\subsubsection{Estimate for (A)}

Using the formula \eqref{eq:A}, we compute
\beastar
\langle (A), \del_H^\pi u \rangle & = & - \langle \delta^{\nabla^\pi}[(u^*\lambda
\wedge (\CL_{X_\lambda}J)J \del_H^\pi u], \del_H^\pi u\rangle \nonumber\\
&=&\langle *\langle (\nabla^\pi(\CL_{X_\lambda}J J))\del_H^\pi u, u^*\lambda\rangle,
\del_H^\pi u\rangle \nonumber\\
& = & + \langle *\langle (\CL_{X_\lambda}J J)\nabla^\pi\del_H^\pi u, u^*\lambda\rangle,
\del_H^\pi u\rangle \nonumber\\
&{}& + \langle *\langle (\CL_{X_\lambda}J J)\del_H^\pi u, \nabla u^*\lambda\rangle,
\del_H^\pi u\rangle.
\eeastar
Hence noting $\nabla^\pi$ above denotes the pull-back connection by the map $u$, we actually mean
$$
(\nabla^\pi(\CL_{X_\lambda}J)) \del_H^\pi u = (\nabla^\pi_{du}(\CL_{X_\lambda}J)) \del_H^\pi u
$$
and so get a bound for the term  $|\langle (A), \del_H^\pi u \rangle|$ by
\bea\label{eq:delta-ident}
&{}&|\langle \delta^{\nabla^\pi}[(u^*\lambda \wedge (\CL_{X_\lambda}J)J \del_H^\pi u], \del_H^\pi u\rangle
\nonumber|\\
&\leq&
\|\nabla^\pi(\CL_{X_\lambda}J J)\|_{C^0(M)}|du||u^*\lambda||d_H^\pi u|^2
+|\langle (\CL_{X_\lambda}JJ )\nabla^\pi(\del_H^\pi u), u^*\lambda\rangle||d_H^\pi u|\nonumber\\
&{}& +|\langle (\CL_{X_\lambda}J J)\del_H^\pi u, \nabla u^*\lambda\rangle||d_H^\pi u| \nonumber\\
&\leq& \|\nabla^\pi(\CL_{X_\lambda}J)\|_{C^0(M)}|du|^2|d_H^\pi u|^2
+ \|\CL_{X_\lambda}J J\|_{C^0} |\nabla^\pi(\del_H^\pi u)|| u^*\lambda| |d_H^\pi u|\nonumber\\
&{}& + \|\CL_{X_\lambda}JJ \|_{C^0} |\del_H^\pi u|  |\nabla u^*\lambda| |d_H^\pi u|
\eea
Using the general inequality
$
ab \leq \frac1{2c} a^2 + \frac{c}{2} b^2,
$
which holds  for any constant $c > 0$, we further bound the last two terms of \eqref{eq:delta-ident} via
\beastar
&{}&
\|\CL_{X_\lambda}J J\|_{C^0(M)}|\nabla^\pi (\del_H^\pi u)||du| |d_H^\pi u|\\
&\leq&\frac{1}{2c}|\nabla^\pi (\del_H^\pi u)|^2+\frac{c}{2}\|\CL_{X_\lambda}J J\|^2_{C^0(M)}
|du|^2|d_H^\pi u|^2
\eeastar
and
\beastar
\|\CL_{X_\lambda}JJ \|_{C^0} |\del_H^\pi u|  |\nabla u^*\lambda| |d_H^\pi u|
&\leq&\frac{1}{2c}|\nabla u^*\lambda|^2+\frac{c}{2}\|\CL_{X_\lambda}J J\|^2_{C^0(M)}|d_H^\pi u|^4.
\eeastar
Here $c$ is any positive constant to be chosen later. Finally, adding them up, we get the upper bound
\bea\label{eq:deltanabla-pi-j}
&{}&|\langle \delta^{\nabla^\pi}[(u^*\lambda \wedge (\CL_{X_\lambda}J)J \del_H^\pi u], \del_H^\pi u\rangle|\nonumber\\
&\leq&
\frac{1}{2c}\left(|\nabla^\pi (\del_H^\pi u)|^2  + |\nabla u^*\lambda|^2\right)\nonumber\\
&{}&
+\frac{c}{2} \|\CL_{X_\lambda}J J\|^2_{C^0(M)} |d_H^\pi u|^4
 + \frac{c}{2} \|\nabla^\pi(\CL_{X_\lambda}J J)\|_{C^0(M)}|du|^2  |d_H^\pi u|^2 \nonumber\\
&\leq &
\frac{1}{2c}\left(|\nabla^\pi (\del_H^\pi u)|^2  + |\nabla u^*\lambda|^2\right)\nonumber\\
&{}&
+\frac{c}{2} \|\CL_{X_\lambda}J J\|^2_{C^0(M)} |d_H^\pi u|^4
 + \frac{c}{4} \|\nabla^\pi(\CL_{X_\lambda}J J)\|_{C^0(M)}(|du|^4 +  |d_H^\pi u|^4).  \nonumber \\
&{}&
\eea
\subsubsection{Estimate for (B)}

Similarly using the fommular for the term (B)
of \eqref{eq:e-pi-weitzenbeck}, we obtain
\bea\label{eq:XHu-pi}
|\langle (B), \del_H^\pi u\rangle |
& = &|\langle \delta^{\nabla^\pi}[(u^*\lambda)\wedge
(\CL_{X_\lambda}J)J X_H^\pi(u)\otimes \gamma], \del_H^\pi u\rangle|\nonumber\\
&\leq&
\|\nabla^\pi(\CL_{X_\lambda}J J)\|_{C^0(M)}|du||u^*\lambda|
|X_H^\pi(u)\otimes \gamma||d_H^\pi u| \nonumber\\
&{}& +|\langle (\CL_{X_\lambda}JJ )\nabla^\pi(X_H(u)\otimes \gamma),
u^*\lambda\rangle||d_H^\pi u|\nonumber\\
&{}& +|\langle (\CL_{X_\lambda}J J)X_H(u)\otimes \gamma,
\nabla u^*\lambda\rangle||d_H^\pi u| \nonumber\\
&\leq& \|\nabla^\pi(\CL_{X_\lambda}J J)\|_{C^0(M)}|du|^2|X_H^\pi(u)\otimes \gamma|
|d_H^\pi u| \nonumber\\
&{}& +|\langle (\CL_{X_\lambda}J J)\nabla^\pi(X_H^\pi(u)\otimes \gamma),
u^*\lambda\rangle||d_H^\pi u|\nonumber\\
&{}& +|\langle (\CL_{X_\lambda}J J)X_H^\pi(u)\otimes \gamma,
\nabla u^*\lambda\rangle||d_H^\pi u|
\eea
We get the bound for the first term
\beastar
&{}&
 \|\nabla^\pi(\CL_{X_\lambda}J J)\|_{C^0(M)}|du|^2|X_H^\pi(u)\otimes \gamma||d_H^\pi u|\\
 &\leq & \frac12  \|\nabla^\pi(\CL_{X_\lambda}J J)\|_{C^0(M)}\|X_H^\pi\otimes \gamma\|_{C^0}
 (|du|^4 + |d_H^\pi u|^2).
\eeastar
We get the bound for the last two terms of \eqref{eq:XHu-pi} via
\beastar
&{}&
|\langle (\CL_{X_\lambda}J J)\nabla^\pi(X_H^\pi(u) \otimes \gamma), u^*\lambda\rangle||d_H^\pi u|\\
&\leq&
\|\CL_{X_\lambda}J J\|_{C^0(M)}\|\nabla^\pi (X_H^\pi \otimes \gamma) \|_{C^0}|du|^2 |d_H^\pi u|\\
&\leq& \frac12 \|\CL_{X_\lambda}J J\|_{C^0(M)}\|\nabla^\pi (X_H^\pi \otimes \gamma) \|_{C^0} \\
&{}& \times |(|du|^4 +  |d_H^\pi u|^2)
\eeastar
and
\beastar
&{}&
|\langle (\CL_{X_\lambda}JJ ) X_H^\pi (u)\otimes \gamma,
 \nabla u^*\lambda\rangle||d_H^\pi u|\\
&\leq&\frac{1}{2c}|\nabla u^*\lambda|^2+\frac{c}{2}\|\CL_{X_\lambda}J J\|^2_{C^0(M)}
\|X_H^\pi\otimes \gamma\|^2_{C^0} |d_H^\pi u|^2.
\eeastar
Finally, adding the three up and inserting the sum into \eqref{eq:XHu-pi}, we get the upper bound
\bea\label{eq:deltaXH(u)-pi-j}
&{}&|\langle \delta^{\nabla^\pi}[(u^*\lambda \wedge (\CL_{X_\lambda}J)J X_H^\pi(u)\otimes \gamma],
\del_H^\pi u\rangle|\nonumber\\
&\leq&
\frac{1}{2c} |\nabla u^*\lambda|^2 +
 \frac{c}{2} \|\CL_{X_\lambda}J J\|_{C^0(M)}^2 \|X_H^\pi\otimes \gamma\|_{C^0}^2
  |d_H^\pi u|^2 \nonumber\\
&{}&  + \left(\frac12 (\|\nabla^\pi(\CL_{X_\lambda}J J)\|_{C^0(M)}\|X_H^\pi\otimes \gamma|_{C^0}
\ + \|\CL_{X_\lambda}J J\|_{C^0(M)}\|\nabla^\pi X_H^\pi \otimes \gamma\|_{C^0}|\right) \nonumber\\
&{}& \quad \times (|du|^4 + |d_H^\pi u|^2)\nonumber\\
 &{}&
\eea

\begin{rem} We remark that the above calculation is a much more
complex variation of the ones presented in \cite[(5.2)--(5.3)]{oh-wang:connection} for the case $H =0$.
\end{rem}

\subsubsection{Estimate for (C)}

Next, using the formula (C), we estimate the term
$$
\langle (C) ,\del_H^\pi u\rangle
= - 4 \langle \delta^{\nabla^\pi} [T^\pi(X_H^\pi(u), \gamma \wedge d_H^\pi u)], d_H^\pi u\rangle
$$
of \eqref{eq:e-pi-weitzenbeck} in a similar way. Utilizing \eqref{eq:C}, we compute
\bea\label{eq:(C)}
&{}&
|\langle (C), \del_H^\pi u\rangle| \nonumber \\
& = &
|4 \langle \delta^{\nabla^\pi} [T^\pi(X_H^\pi(u), \gamma \wedge d_H^\pi u)], d_H^\pi u\rangle|\nonumber\\
& \leq & \frac{2}{c} |d^{\nabla^\pi}d_H^\pi u|^2
+ 16c (\|T^\pi\|_{C^0}\|X_H^\pi\|_{C^0}\|(\gamma \circ j)\|_{C^0})^2 \nonumber\\
&{}& +4 (\|\nabla^\pi T^\pi\|_{C^0} \|X_H^\pi\|_{C^0}\|\gamma\circ j\|_{C^0}
+ \| T^\pi\|_{C^0} \|\nabla^\pi X_H^\pi\|_{C^0}\|\gamma\circ j\|_{C^0})|du| |d_H^\pi u|^2\nonumber\\
& \leq &
\frac{2}{c} |d^{\nabla^\pi}d_H^\pi u|^2
+ 16c (\|T^\pi\|_{C^0}\|X_H^\pi\|_{C^0}\|(\gamma \circ j)\|_{C^0})^2 \nonumber\\
&{}& + 2(\|\nabla^\pi T^\pi\|_{C^0} \|X_H^\pi\|_{C^0}\|\gamma\circ j\|_{C^0}
 + \| T^\pi\|_{C^0} \|\nabla^\pi X_H^\pi\|_{C^0}\|\gamma\circ j\|_{C^0})(|du|^2+ |d_H^\pi u|^4.)\nonumber\\
&{}&+ 4 \|T^\pi\|_{C^0} \|X_H^\pi\|_{C^0} \|\gamma\circ j\|_{C^0}
 |d^{\nabla^\pi}d_H^\pi u| |d_H^\pi u|.
\eea

By substituting \eqref{eq:deltanabla-pi-j}, \eqref{eq:deltaXH(u)-pi-j} and \eqref{eq:(C)}
into \eqref{eq:e-pi-weitzenbeck} and using the identity
$u^*\lambda = u^*\lambda_H - H(u) \gamma$,
we obtain
\beastar
&{}&
-\frac{1}{2}\Delta e_H^\pi(u) \\
& \geq &
|\nabla^\pi (d_H^\pi u)|^2+K|d_H^\pi u|^2+
\langle \ric^{\nabla^\pi} (\del_H^\pi u), \del_H^\pi u \rangle  \\
&{}& - \frac{1}{2c}\left(|\nabla^\pi (d_H^\pi u)|^2+ |\nabla(u^*\lambda)|^2\right)
- \frac{c}{2} \|\CL_{X_\lambda}J J\|^2_{C^0(M)} |d_H^\pi u|^4 \\
&{}& - \frac{c}{4}\|\nabla^\pi(\CL_{X_\lambda}J J)\|_{C^0(M)}(|du|^4 + |d_H^\pi u|^4)\\
&{}& \\
&{}& - \frac{1}{2c} |\nabla u^*\lambda_H |^2 -
 \frac{c}{2} \|\CL_{X_\lambda}J J\|_{C^0(M)}^2 \|X_H^\otimes \gamma\|_{C^0}^2 |d_H^\pi u|^2 \nonumber\\
&{}&  - \frac12 \left(\|\nabla^\pi(\CL_{X_\lambda}J J)\|_{C^0(M)}\|X_H^\pi \otimes \gamma\|_{C^0}
+ \|\CL_{X_\lambda}J J\|_{C^0(M)}\|\nabla^\pi (X_H^\pi\otimes \gamma) \|_{C^0}|\right) \\
&{}& \times  (|du|^4 + |d_H^\pi u|^2)\\
 &{}& \\
&{}& - \frac{2}{c} |\nabla^\pi(d_H^\pi u)|^2 - 16c  (\|T^\pi\|_{C^0}
\|X_H^\pi\|_{C^0}\|\gamma \circ j\|_{C^0})^2 \\
&{}& - 2 \left(\|\nabla^\pi T^\pi\|_{C^0} \|X_H^\pi\|_{C^0}\|\gamma\circ j\|_{C^0}
+\| T^\pi\|_{C^0} \|\nabla^\pi X_H^\pi\|_{C^0}\|\gamma\circ j\|_{C^0}\right)(|du|^2 + |d_H^\pi u|^4)\\
&{}& - 4 \|T^\pi\|_{C^0} \|X_H^\pi\|_{C^0} \|\gamma\circ j\|_{C^0}
 |d^{\nabla^\pi}d_H^\pi u| |d_H^\pi u| \\
&{}&  - (\|\nabla X_H^\pi\|_{C^0} \|\gamma\|_{C^0}
+ \|X_H^\pi\|_{C^0} \|d\gamma\|_{C^0}) |du||\del_H^\pi u|^2.
\eeastar
By rearranging various summands of the above, we derive
\eqref{eq:Reeb-contribution}.
\end{proof}

\subsection{Differential inequality for $\Delta |u^*\lambda_H|^2$}

In this subsection, we examine the contribution by the Reeb component $u^*\lambda$ such as \eqref{eq:Reeb-contribution}.
Here crucially enters the second defining equation of \emph{perturbed contact instantons}
\eqref{eq:perturbed-deltaw=0} below.

Specifically, we derive the following inequality starting from Proposition \ref{prop:Deltau*lambdaH}.

\begin{prop}\label{prop:Delta-u*lambda-Hgamma} We have
\bea\label{eq:laplacian-pi-upper}
\langle \Delta  (u^*\lambda_H),  u^*\lambda_H \rangle
&\leq & \frac{1}{2c} |\nabla^\pi d_H^\pi u|^2  + \frac{1}{c} |\nabla( u^*\lambda_H)|^2
+ \frac{1}{2c} |\nabla d u|^2 \nonumber\\
&{}& + C_1 |du|^2 + C_2 |u^*\lambda_H|^2 + C_3|du| |u^*\lambda_H|^2 \nonumber\\
&{}&  + C_4|du|^2 |u^*\lambda_H|^2 + C_5 |u^*\lambda_H|^4
+ C_6 |du||u^*\lambda_H| \nonumber \\
&{}&
\eea
where the constants $C_i$ with $i=1, \ldots, 6$ depend only on the geometry of $(M,\lambda, J)$,
$(\dot \Sigma, j, h)$, the Hamiltonian $H$ and $c$, but independent of the map $u$.
\end{prop}
\begin{proof}
Recall
$$
g_{H, u}(z) = g_{(\phi_H^t)^{-1}}(u(\tau,t)) = - g_{\phi_H^t}((\phi_H^t)^{-1}((u(\tau,t)))
$$
where
\be\label{eq:phiHt}
\phi_H^t: = \psi_H^t \circ (\psi_H^1)^{-1}.
\ee
Therefore we have
$$
dg_{H,u} = - d\left(g_{\phi_H^t} \circ \phi_H^t\right) \circ du
- R_\lambda[H](t,\phi_H^t(u(\tau,t)))\, dt.
$$
We compute
\beastar
\nabla(dg_{H,u}) & = & \nabla (- dg_{\phi_H^t} \circ \phi_H^t \circ du - R_\lambda[H](t,\phi_H^t(u(\tau,t)))\, dt)\\
& = & \nabla (- dg_{\phi_H^t} \circ \phi_H^t \circ du)
- \nabla \left(R_\lambda[H](t,\phi_H^t(u(\tau,t)))\, dt\right).
\eeastar
Here we have
\beastar
\nabla (- dg_{\phi_H^t} \circ \phi_H^t \circ du)
& = & \nabla (- dg_{\phi_H^t} \circ \phi_H^t) \circ du
- \nabla_{du} \left(g_{\phi_H^t} \circ \phi_H^t\right) \circ \nabla du,
\eeastar
and
\beastar
\nabla \left(R_\lambda[H](t,\phi_H^t(u(\tau,t)))\, dt\right)
=  \nabla_{du} \left(R_\lambda[H](t,\phi_H^t\right)\, dt.
\eeastar

Summarizing the above calculations, we obtain the following inequality.

\begin{lem}\label{lem:dgHw-nabladgHw} We have
$$
|dg_{H,u}|\leq \|dg_{\phi_H^t} \circ \phi_H^t\|_{C^0} \cdot |du|
+ \|R_\lambda[H]\|_{C^0}
$$
and
\beastar
|\nabla dg_{H,u}| & \leq &
\left(\|\nabla (dg_{\phi_H^t} \circ \phi_H^t)\|_{C^0}
 +  \|\nabla \left(R_\lambda[H](t,\phi_H^t)\right)\|_{C^0})\right) \cdot |du| \\
&{}& + \|d(g_{\phi_H^t} \circ \phi_H^t) \|_{C^0} \cdot |\nabla du|.
\eeastar
\end{lem}

We recall \eqref{eq:Delta-u*lambda+Hgamma} here:
\beastar
\Delta (u^*\lambda_H) &  = &  d*( (dg_{H,u}) \wedge u^*\lambda_H \circ j)\nonumber \\
&{}& + \frac12 *d|(du - X_H \otimes \gamma)^\pi|^2 \nonumber \\
&{}&-* d*(u^*(R_\lambda[H]\, \lambda)\wedge \gamma  + u^*H\, d\gamma)
\eeastar
and hence
\beastar
\langle \Delta (u^*\lambda_H), (u^*\lambda_H) \rangle
&  = & \underbrace{\langle  d*((dg_{H,u}) \wedge u^*\lambda_H\circ j)
, (u^*\lambda_H) \rangle}_{(I)} \nonumber \\
&{}& + \underbrace{\frac12 \langle *d|(du - X_H \otimes \gamma)^\pi|^2,
(u^*\lambda_H) \rangle}_{(II)}
 \nonumber \\
&{}&\underbrace{-\langle * d*(u^*(R_\lambda[H]\, \lambda)\wedge \gamma  + u^*H\, d\gamma),
, (u^*\lambda_H) \rangle}_{(III)}
\eeastar

We now compute the three terms separately. We first compute
\beastar
&{}&
d* ((dg_{H,u}) \wedge u^*\lambda_H \circ j)\\
& = & - *d \langle u^*\lambda_H, d(g_{H, u}) \circ j \rangle \\
& = & - *\langle d^\nabla (u^*\lambda_H), dg_{H,u} \circ j \rangle
- *\langle  u^*\lambda_H, d^\nabla( dg_{H,u} \circ j)\rangle.
\eeastar
Therefore
\beastar
|(I)| & = & |\langle   d* (dg_{H,u} \circ j \wedge u^*\lambda_H \circ j),
u^*\lambda_H \rangle|\\
& = & | *\langle d^\nabla (u^*\lambda_H), dg_{H,u} \circ j \rangle
- *\langle  u^*\lambda_H, d^\nabla(dg_{H,u} \circ j) \rangle|\cdot
|u^*\lambda_H|.
\eeastar
Here we have
\beastar
&{}&
| *\langle d^\nabla (u^*\lambda_H), dg_{H,u} \circ j \rangle
- *\langle  u^*\lambda_H, d^\nabla (dg_{H,u} \circ j) \rangle|\\
& \leq & (|\nabla(u^*\lambda_H)|\cdot |dg_{H,u} \circ j |
+| u^*\lambda_H||d^\nabla(dg_{H,u} \circ j)|.
\eeastar
Altogether, we have derived
\beastar
&{}&
|\langle   d* (dg_{H,u}\circ j \wedge u^*\lambda_H, u^*\lambda_H \rangle|\\
& \leq &  (|\nabla(u^*\lambda_H)|\cdot |dg_{H,u}\circ j ||u^*\lambda_H|
 + |\nabla (dg_{H,u} \circ j)| | u^*\lambda_H|^2.
\eeastar

Substituting the inequalities in Lemma \ref{lem:dgHw-nabladgHw} into above, we have derived
\bea\label{eq:d*dghHu}
&{}&
|\langle  * d* (dg_{H,u}) \circ j \wedge u^*\lambda_H, u^*\lambda_H \rangle|\nonumber\\
& \leq & \left( \|dg_{\phi_H^t} \circ \phi_H^t\|_{C^0} \cdot |du \circ j|
+ \|R_\lambda[H]\|_{C^0}\right) |\nabla(u^*\lambda_H)||u^*\lambda_H|\nonumber\\
& + & \left(\|\nabla (dg_{\phi_H^t} \circ \phi_H^t)\|_{C^0}
 +  \|d\left(R_\lambda[H](t,\phi_H^t)\right)\|_{C^0})\right)
 \cdot |du \circ j|  | u^*\lambda_H|^2\nonumber \\
 &{}&  +
\|dg_{\phi_H^t} \circ \phi_H^t) \|_{C^0} \cdot |\nabla du| | u^*\lambda_H|^2.
 \eea
We get the bound
\beastar
&{}&
\|dg_{\phi_H^t} \circ \phi_H^t\|_{C^0} \cdot |du \circ j| |\nabla(u^*\lambda_H)||u^*\lambda_H|\\
& \leq & \frac1{2c} |\nabla(u^*\lambda_H)|^2 + \frac{c}{2}
(\|dg_{\phi_H^t} \circ \phi_H^t\|_{C^0})^2 |du\circ j|^2 |u^*\lambda_H|^2,
\eeastar
and
$$
 \|R_\lambda[H]\|_{C^0} |\nabla(u^*\lambda_H) ||u^*\lambda_H|
\leq \frac1{2c} |\nabla(u^*\lambda_H)|^2
 + \frac{c}{2}\|R_\lambda[H]\|_{C^0}^2 |u^*\lambda_H|^2,
$$
and
$$
\|dg_{\phi_H^t} \circ \phi_H^t \|_{C^0} \cdot |\nabla du| |u^*\lambda_H|^2
\leq  \frac1{2c} |\nabla du|^2 + \frac {c}{2}\|dg_{\phi_H^t} \circ \phi_H^t \|_{C^0}^2
 |u^*\lambda_H|^4.
$$
Substituting these into \eqref{eq:d*dghHu} and rearranging terms, we get
\bea\label{eq:item-1}
&{}&
|\langle  d* (dg_{H,u} \circ j \wedge u^*\lambda_H), u^*\lambda_H \rangle| \nonumber\\
& \leq & \frac1{c} |\nabla(u^*\lambda_H)|^2 + \frac1{2c} |\nabla du|^2\nonumber \\
&{}&  + \frac {c}{2}\|dg_{\phi_H^t} \circ \phi_H^t) \|_{C^0}^2 |u^*\lambda_H|^4 \nonumber\\
&{}& + \frac{c}{2} \|dg_{\phi_H^t} \circ \phi_H^t\|_{C^0}^2 |du\circ j|^2 |u^*\lambda_H|^2
+ \frac{c}{2}\|R_\lambda[H]\|_{C^0}^2 |u^*\lambda_H|^2 \nonumber\\
&{}& + \left(\|\nabla (dg_{\phi_H^t} \circ \phi_H^t)\|_{C^0}
+ \|d\left(R_\lambda[H](t,\phi_H^t\right)\|_{C^0}\right)
\cdot |du\circ j| | u^*\lambda_H|^2.\nonumber\\
{}
\eea

Next we compute
\bea\label{eq:item-2}
|(II)| & = & \frac12 \left| \langle *d|(du - X_H \otimes \gamma)^\pi|^2, u^*\lambda_H \rangle
\right| \nonumber\\
& \leq &  |\nabla^\pi( d_H^\pi u)| |d_H^\pi u||u^*\lambda_H|.
\eea
Finally we have
\beastar
|* d*(u^*(R_\lambda[H]\, \lambda)\wedge \gamma)|
& \leq & 2\|R_\lambda[H]\|_{C^0}|\gamma||du|\\
|*d*u^*H \, d\gamma| & \leq & \|dH\|_{C^0} |d\gamma|_{C^0} |du|
\eeastar
and hence
\be\label{eq:item-3}
|(III)| \leq (2\|R_\lambda[H]\|_{C^0}|\gamma| + \|dH\|_{C^0} |d\gamma|_{C^0})
 |du| |u^*\lambda_H|.
\ee
By summing up \eqref{eq:item-1}, \eqref{eq:item-2} and \eqref{eq:item-3} and
$|du| = |du\circ j|$, we have derived
\beastar
&{}&
\langle \Delta  (u^*\lambda_H),  (u^*\lambda_H)\rangle  \nonumber\\
&\leq & \frac{1}{2c} |\nabla^\pi d_H^\pi u|^2  + \frac{1}{c} |\nabla( u^*\lambda_H)|^2
+ \frac{1}{2c} |\nabla du|^2 \nonumber\\
&{}& + C_1 |du|^2 + C_2 |u^*\lambda_H|^2 + C_3|du| |u^*\lambda_H|^2 \nonumber\\
&{}&  + C_4|du|^2 |u^*\lambda_H|^2 + C_5 |u^*\lambda_H|^4
+ C_6 |du||u^*\lambda_H|.
\eeastar
This finishes the proof of \eqref{eq:laplacian-pi-upper}.
\end{proof}

An immediate corollary of this proposition and \eqref{eq:e-lambda-weitzenbeck} is the following.
\begin{cor} \label{cor:e-lambda-weitezenbeck} We have
\beastar
-\frac12 \Delta |u^*\lambda_H|^2 & \geq & \left(1- \frac1{2c}\right)  |\nabla u^*\lambda_H|^2 + \min K |u^*\lambda_H|^2  \\
&{}& \left(1-  \frac{1}{c} \right) |\nabla( u^*\lambda_H)|^2 - \frac{1}{2c} |\nabla du|^2 \\
&{}& - C_1 |du|^2 - C_2 |u^*\lambda_H|^2 - C_3|du| |u^*\lambda_H|^2 \\
&{}&  - C_4|du|^2 |u^*\lambda_H|^2 - C_5 |u^*\lambda_H|^4
- C_6 |du||u^*\lambda_H|.
\eeastar
\end{cor}

\section{Local coercive $W^{2,2}$-estimates}

In this section, combining the calculations given in the previous section,
we first calculate the total energy density which is defined as
$$
e_H(u):=|d_H u|^2=e_H^\pi(u)+|u^*\lambda_H|^2.
$$

\subsection{Bounding $\Delta e_H(u)$ from above}

The following differential inequality is the key inequality to establish
whose proof will occupy the entirety of this subsection.

\begin{prop}\label{prop:second:derivative}
Let $u$ be any  perturbed contact instanton. Then there exist constants $C_i'$, $i = 1, \ldots, 4$
independent of $u$' such that
\be\label{eq:second-derivative}
\frac18|\nabla(d_H^\pi u)|^2 + \frac38|\nabla(u^*\lambda_H)|^2 \leq 
-\frac{1}{2}\Delta e_H(u) + C_1'e_H(u)^2 +  C_2' e_H(u) + C_3' |du|^2 + C_4' |du|^4.
\ee
\end{prop}

Inserting \eqref{eq:laplacian-pi-upper} into \eqref{eq:e-lambda-weitzenbeck} and
summing it up with \eqref{eq:e-pi-weitzenbeck},
we obtain the following inequality for any contact instanton $u$
\bea
&{}& -\frac{1}{2}\Delta e_H(u) = -\frac12 (\Delta e_H^\pi(u) - \frac12 \Delta |u^*\lambda_H|^2
 \nonumber\\
& \geq & (1 - \frac{5}{2c})|\nabla^\pi (d_H^\pi u)|^2
 + \left(1- \frac{3}{2c}\right) |\nabla u^*\lambda_H |^2 - \frac1{2c} |\nabla du|^2 \nonumber \\
&{}& - (\min K -\|\ric^{\nabla^\pi}\|_{C^0}) | d_H^\pi u|^2 \nonumber\\
&{}& - \frac1{4c}( \|\nabla H \gamma)\|_{C^0}^2|du|^4 + \|H\|_{C^0}^2 \|\nabla \gamma\|_{C^0}^2|du|^2)
\nonumber \\
&{}& - \frac{c}{2} \|\CL_{X_\lambda}J J\|^2_{C^0(M)} |d_H^\pi u|^4 \nonumber\\
&{}& \nonumber\\
&{}& - \frac{c}{2} \|\CL_{X_\lambda}J J\|_{C^0(M)}^2 \|X_H^\otimes \gamma\|_{C^0}^2
 |d_H^\pi u|^2 \nonumber \\
 &{}& - \frac{c}{4} \|\nabla^\pi(\CL_{X_\lambda}J) J\|_{C^0(M)}(|du|^4 +  |d_H^\pi u|^4) \nonumber\\
&{}&  - \frac12 \left(\|\nabla^\pi(\CL_{X_\lambda}J J)\|_{C^0(M)}\|X_H^\pi \otimes \gamma \|_{C^0}
+ \|\CL_{X_\lambda}J J\|_{C^0(M)}\|\nabla^\pi(X_H^\pi \otimes \gamma) \|_{C^0}|\right)\nonumber \\
&{}& \quad \times (|du|^4 + |d_H^\pi u|^2) \nonumber\\
&{}& \nonumber \\
&{}& - 16c (\|T^\pi\|_{C^0}\|X_H^\pi\|_{C^0}\|\gamma \circ j\|_{C^0})^2 \nonumber\\
&{}& - 2 \left(\|\nabla^\pi T^\pi\|_{C^0} \|X_H^\pi\|_{C^0}\|\gamma\circ j\|_{C^0}
+\| T^\pi\|_{C^0} \|\nabla^\pi X_H^\pi\|_{C^0}\|\gamma\circ j\|_{C^0}\right) \nonumber \\
&{}& \times (|du|^2 + |d_H^\pi u|^4)\nonumber\\
&{}& - 4 \|T^\pi\|_{C^0} \|X_H^\pi\|_{C^0} \|\gamma\circ j\|_{C^0}
 |d^{\nabla^\pi}d_H^\pi u| |d_H^\pi u| \nonumber\\
&{}&  - \left(\|\nabla X_H^\pi\|_{C^0} \|\gamma\|_{C^0}
+ \|X_H^\pi\|_{C^0} \|d\gamma\|_{C^0}\right)(|du|^2 + |d_H^\pi u|^4)\nonumber\\
&{}& \nonumber \\
&{}& - \frac1{2c}\|\nabla(H\gamma)\|_{C^0} |du|^2  +\min K| u^*\lambda_H|^2  \nonumber\\
&{}& + C_1 |du|^2 - C_2 |u^*\lambda_H|^2 - C_3|du| |u^*\lambda_H|^2 \nonumber\\
&{}&  - C_4|du|^2 |u^*\lambda_H|^2 - C_5 |u^*\lambda_H|^4
- C_6 |du||u^*\lambda_H|
\eea
By the inequality
$$
\frac12 |\nabla du|^2 \leq |\nabla(d_H u)|^2 + |\nabla (X_H \otimes \gamma)|^2)
$$
combined with the inequality from \eqref{eq:|nabladHu|2}
$$
|\nabla (d_H u)|^2 \leq |\nabla^\pi (d^\pi_H u)|^2 + |\nabla(u^*\lambda_H)|^2 + \|\nabla R_\lambda\|^2_{C^0}|u^*\lambda||du|,
$$
 the first line of the above inequality is bounded below by
\beastar
&{}&
+  (1 - \frac{5}{2c})|\nabla^\pi (d_H^\pi u)|^2 + (1 - \frac{3}{2c}) |\nabla u^*\lambda_H|^2\\
&{}& - \frac{1}{c} \left( |\nabla^\pi (d^\pi_H u)|^2 + |\nabla(u^*\lambda_H)|^2 + \|\nabla R_\lambda\|^2_{C^0}|u^*\lambda||du|\right)  - \frac1c |\nabla(X_H \otimes \gamma)|^2 \\
& \geq &  (1 - \frac{7}{2c})|\nabla^\pi (d_H^\pi u)|^2 + (1 - \frac{5}{2c}) |\nabla u^*\lambda_H|^2 \\
&{}& - \frac1{c} \|\nabla R_\lambda\|^2_{C^0}(|u^*\lambda|^2 + |du|^2) 
- \frac1c  |\nabla(X_H \otimes \gamma)|^2.
\eeastar

By rearranging the summands above,  we arrive at the following
\bea
&{}& -\frac{1}{2}\Delta e_H(u) = -\frac12 (\Delta e_H^\pi(u) - \frac12 \Delta |u^*\lambda_H|^2
 \nonumber\\
& \geq &  (1 - \frac{7}{2c})|\nabla^\pi (d_H^\pi u)|^2 + (1 - \frac{5}{2c}) |\nabla u^*\lambda_H|^2
\nonumber\\
&{}& - \frac1{c} \|\nabla R_\lambda\|^2_{C^0}(|u^*\lambda|^2 + |du|^2) 
 - \frac1c|\nabla (X_H \otimes \gamma)|^2 \nonumber \\
 &{}& \nonumber\\
&{}& - (\min K -\|\ric^{\nabla^\pi}\|_{C^0}) | d_H^\pi u|^2 \nonumber\\
&{}& - \frac1{4c}( \|\nabla H \gamma)\|_{C^0}^2|du|^4 + \|H\|_{C^0}^2 \|\nabla \gamma\|_{C^0}^2|du|^2)
\nonumber \\
&{}& - \frac{c}{2} \|\CL_{X_\lambda}J J\|^2_{C^0(M)} |d_H^\pi u|^4 \nonumber\\
&{}& \nonumber\\
&{}& - \frac{c}{2} \|\CL_{X_\lambda}J J\|_{C^0(M)}^2 \|X_H^\otimes \gamma\|_{C^0}^2
 |d_H^\pi u|^2 \nonumber \\
 &{}& - \frac{c}{4} \|\nabla^\pi(\CL_{X_\lambda}J) J\|_{C^0(M)}(|du|^4 +  |d_H^\pi u|^4) \nonumber\\
&{}&  - \frac12 \left(\|\nabla^\pi(\CL_{X_\lambda}J J)\|_{C^0(M)}\|X_H^\pi\|_{C^0}+ \|\CL_{X_\lambda}J J\|_{C^0(M)}\|\nabla^\pi X_H^\pi \|_{C^0}|\right)
\nonumber\\
&{}& \quad \times (|du|^4 + |d_H^\pi u|^2) \nonumber\\
&{}& \nonumber \\
&{}& - 16c (\|T^\pi\|_{C^0}\|X_H^\pi\|_{C^0}\|\gamma \circ j\|_{C^0})^2 \nonumber\\
&{}& - 2 \left(\|\nabla^\pi T^\pi\|_{C^0} \|X_H^\pi \otimes \gamma \|_{C^0}\|\gamma\circ j\|_{C^0}
+\| T^\pi\|_{C^0} \|\nabla^\pi (X_H^\pi \otimes \gamma\|_{C^0}\|\gamma\circ j\|_{C^0}\right) \nonumber\\
&{}& \times (|du|^2 + |d_H^\pi u|^4)\nonumber\\
&{}& - 4 \|T^\pi\|_{C^0} \|X_H^\pi\|_{C^0} \|\gamma\circ j\|_{C^0}
 |d^{\nabla^\pi}d_H^\pi u| |d_H^\pi u| \nonumber\\
&{}&  - \left(\|\nabla X_H^\pi\|_{C^0} \|\gamma\|_{C^0}
+ \|X_H^\pi\|_{C^0} \|d\gamma\|_{C^0}\right)(|du|^2 + |d_H^\pi u|^4)\nonumber\\
&{}& \nonumber \\
&{}& - \frac1{2c}\|\nabla(H\gamma)\|_{C^0} |du|^2  +\min K| u^*\lambda_H|^2 \nonumber\\
&{}& + C_1 |du|^2 - C_2 |u^*\lambda_H|^2 - C_3|du| |u^*\lambda_H|^2 \nonumber\\
&{}&  - C_4|du|^2 |u^*\lambda_H|^2 - C_5 |u^*\lambda_H|^4
- C_6 |du||u^*\lambda_H|
\eea

We fix a sufficiently large constant $c$ so that $c \geq 4 $  and obtain
\bea
\label{eq:laplace-e-derivative}
-\frac{1}{2}\Delta e_H(u)
&\geq & (1  - \frac{7}{2c})|\nabla^\pi (d_H^\pi u)|^2 + (1 - \frac{5}{2c}) |\nabla ( u^*\lambda_H)|^2 \nonumber\\
&{}& - C_1'e_H(u)^2 - C_2' e_H(u) - C_3' |du|^2 - C_4' |du|^4
\eea
where $C_i'$ does not depend on $u$ but depends only of $(M,\lambda,J)$ and $(\dot\Sigma, j,h)$ and
on the choice of $c$ but fixed.

Now we rewrite $\eqref{eq:laplace-e-derivative}$ into
\bea\label{eq:laplace-higherderivative}
&{}&
(1  - \frac{7}{2c})|\nabla^\pi (d_H^\pi u)|^2 + (1 - \frac{5}{2c}) |\nabla ( u^*\lambda_H)|^2  \nonumber\\
&\leq & -\frac{1}{2}\Delta e_H(u) + C_1'e_H(u)^2 + C_2' e_H(u) + C_3' |du|^2 + C_4' |du|^4.
 \eea
We still take $c=4$ and get the following coercive estimate for contact instantons
\beastar
\frac18 |\nabla^\pi (d_H^\pi u)|^2 + \frac38 |\nabla ( u^*\lambda_H)|^2
& \leq & -\frac{1}{2}\Delta e_H(u) + C_1'e_H(u)^2 +  C_2' e_H(u)\nonumber \\
&{}& \quad + C_3' |du|^2 + C_4' |du|^4.
\eeastar
This finishes the proof.

\subsection{Proof of a priori local $W^{2,2}$ estimate}
\label{subsec:W22-estimate}

The proof of the following local inequality
is  similar to that of \cite[Proof of Theorem 4.5]{oh:contacton-Legendrian-bdy}
except $dw$ there replaced by $d_H u$ in the current circumstance.

\begin{thm}\label{thm:coercive-L2}
For any pair of domains $D_1$ and $D_2$ in $\dot\Sigma$ such that $\overline{D_1}\subset D_2$,
\beastar
&{}&
\|\nabla(d_H u)\|^2_{L^2(D_1)}\\
 &\leq&
 C_1 \|d_H u\|_{L^4(D_2)}^4 + C_2 \|d_H u\|_{L^2(D_2)}^2 + C_3 \|du\|_{L^4(D_2)}^4
 + C_4 \|du\|_{L^2(D_2}^2  + 2\CB
\eeastar
with the boundary contribution
$$
\CB = \int_{\del D_2} ( C_7 |du|^3 |d_Hu| + C_8 |du|^2|d_Hu|  + C_9 |d_H u| |du|).
$$
for any perturbed contact instanton $u$,
where $C_i = C_i(D_1, D_2)$ are some constants which depend on $D_1$, $D_2$ and $(M, \lambda, J)$,
but are independent of $u$.
\end{thm}

The rest of the remaining section will be occupied by the proof of Theorem \ref{thm:coercive-L2}.
We have only to consider the case of a pair of semi-discs $D_1,\, D_2 \subset \dot \Sigma$
with $\overline D_1 \subset D_2$ such that $\del D_2 \subset \del D_1 \subset \del \dot \Sigma$.
(The open disc cases of $D_1 \subset D_2$ are already treated in \cite[Appendix C]{oh-wang:CR-map1}.)

For the pair of given domains $D_1$ and $D_2$, we choose another domain $D$ such that
$\overline D_1 \subset D \subset \overline D \subset D_2$ and a smooth cut-off function $\chi:D_2\to \R$ such that
$\chi\geq 0$ and
$\chi\equiv 1$ on $\overline{D_1}$, $\chi\equiv 0$ on $D_2-D$.
By multiplication of  $\chi^2$ to eqref{eq:second-derivative} followed by integration, we get
\be\label{eq:int-D1}
\int_{D_1}|\nabla(d_Hu)|^2 \leq \int_{D}\chi^2|\nabla(d_Hu)|^2
\leq \int_{D_2}|\nabla(d_Hu)|^2.
\ee
Then we start with the inequality \eqref{eq:second-derivative} for the proof. Utilizing
\eqref{eq:nabladHu} and adjusting coefficients, we obtain
$$
\frac18 |\nabla (d_H u)|^2
\leq  -\frac{1}{2}\Delta e_H(u) + C_1'e_H(u)^2 +  C_2' e_H(u)
+ C_3' |du|^2 + C_4' |du|^4.
$$
Integrating this over $D_2$ and applying the second inequality of the above
inequality, we derive
\beastar
\frac18 \int_{D}\chi^2|\nabla(d_Hu)|^2
&\leq& C_1' \int_{D_2}|d_Hu|^4+ C_2' \int_{D_2}|d_Hu|^2 + C_3' \int_{D^2} |du|^4
+ C_4' \int_{D_2} |du|^2 \\
&{}&  -\frac12 \int_{D}\chi^2\Delta e_H.
\eeastar

We now deal with the last term $\int_{D_2}\chi^2 \Delta e_H$.
Its integrand becomes
\beastar
\chi^2\Delta e_H\, dA&=&*(\chi^2 \Delta e_H)=\chi^2 *\Delta e_H
=-\chi^2 d*de_H\\
&=&-d(\chi^2 *de_H)+2\chi d\chi\wedge (*de_H).
\eeastar
Therefore we have
$$
\int_D \chi^2\Delta e_H\, dA = \int_D -d(\chi^2 *de_H)+ \int_D 2\chi d\chi\wedge (*de_H)
$$
and hence
\bea\label{eq:int-D}
\frac18 \int_{D}\chi^2|\nabla(d_Hu)|^2 &\leq &
 C_1' \int_{D_2}|d_Hu|^4+ C_2' \int_{D_2}|d_Hu|^2 + C_3' \int_{D^2} |du|^4
+ C_4' \int_{D_2} |du|^2 \nonumber\\
&{}& - \int_{D}\chi d\chi \wedge (*de_H) + \frac12 \int_D d(\chi^2 * de_H).
\eea

It remains to estimate the integrals of the last line. For the first one,
by the same argument as in \cite[p.677]{oh-wang:CR-map1} applied to the second integral
and with $dw$ there replaced by $de_H$, we derive the following estimate.
\begin{lem}\label{lem:intD-chidchi}
For any constant $\epsilon > 0$, we have
\be\label{eq:intchidchi}
\left|\int_{D}\chi d\chi\wedge(*de_H)\right| \leq \frac{1}{\epsilon}\int_{D}\chi^2|\nabla(d_Hu)|^2\,dA
+ \epsilon \|d\chi\|_{C^0(D)}^2\int_{D}|d_Hu|^2\,dA.
\ee
\end{lem}
\begin{proof} We first derive
\beastar
|\int_{D}\chi d\chi\wedge(*de_H)|
& =& |\int_{D}\chi\langle d\chi, de_H\rangle \,dA| \\
& \leq & \int_{D}|\chi||\langle d\chi, de_H \rangle \,dA|\leq\int_{D}|\chi||d\chi||de_H|\,dA.
\eeastar
And recall
$$
|de_H| = |d\langle d_H u, d_H u\rangle|=2|\langle \nabla(d_H u), d_H u\rangle|\leq 2|\nabla (d_H u)||d_H u|.
$$
Hence
\beastar
|\int_{D}\chi d\chi\wedge(*de_H)|&\leq& \int_{D}2|\chi||d\chi||\nabla (d_Hu)||d_H u|\,dA\\
&\leq&  \frac{1}{\epsilon}\int_{D}\chi^2|\nabla(d_H u)|^2\,dA+ \epsilon \int_{D}|d\chi|^2|d_H u|^2\,dA\\
&\leq&  \frac{1}{\epsilon}\int_{D}\chi^2|\nabla(d_H u)|^2\,dA+ \epsilon \|d\chi\|_{C^0(D)}^2\int_{D}|d_H u|^2\,dA
\eeastar
This finishes the proof.
\end{proof}

Next we examine the second integral and prove the following inequality
whose derivation is rather complex and so postponed till the end of this section.

 \begin{lem}[Compare with Lemma 4.9 \cite{oh:contacton-Legendrian-bdy}]
 \label{lem:boundary-integral}
\be\label{eq:CB}
\left| \int_{\del D} -\chi^2 * de_H|_{\del D} \right|  \leq
\int_{\del D_2} ( C_7|du|^3 |d_Hu| + C_8 |du|^2|d_Hu|  + C_9 |d_H u| |du|) =: \CB
\ee
where we have
\beastar
C_7 & =  &2\|B\|_{C^0} + \|d\lambda\|_{C^0} \|\lambda_H\|_{C^0} \\
C_8 & = & \|dH\|_{C^0} \|\lambda_H\|_{C^0} \\
C_9 & = & \|H\|_{C^0} \|\nabla^2 \gamma\|_{C^0} \|\lambda_H\|_{C^0}.
\eeastar
where $B = B_i$ is the second fundamental form of $R_i$.
\end{lem}

Assuming this lemma for the moment, we proceed with the proof of Theorem \ref{thm:coercive-L2}.

We insert \eqref{eq:intchidchi} and \eqref{eq:CB} into \eqref{eq:int-D}, we obtain
\bea\label{eq:int-D2}
\frac18 \int_{D}\chi^2|\nabla(d_Hu)|^2 &\leq &
 C_1' \int_{D_2}|d_Hu|^4+ C_2' \int_{D_2}|d_Hu|^2 + C_3' \int_{D^2} |du|^4
+ C_4' \int_{D_2} |du|^2 \nonumber\\
&{}& \int_{D}\frac{2 \chi^2}{\epsilon}|\nabla(d_Hu)|^2\,dA
+ 2\epsilon \|d\chi\|_{C^0(D)}^2\int_{D}|d_Hu|^2\,dA\nonumber \\
&{}& + 2\CB.
\eea
which we rearrange into
\beastar
\frac18 \int_{D}\chi^2|\nabla(d_Hu)|^2
&\leq& \int_D\frac{2\chi^2}{\epsilon}|\nabla(d_Hu)|^2 \\
&{} &  + C_1' \int_{D_2}|d_Hu|^4+(C_2' + 2\epsilon \|d\chi\|_{C^0(D)}^2) \int_{D_2}|d_Hu|^2 \\
&{}& + C_3' \int_{D^2} |du|^4 + C_4' \int_{D_2} |du|^2 + 2\CB.
\eeastar
Then by taking $\epsilon = 32 $ and then setting
$$
C_1 : = 16C_1', \, C_2 = 16(C_2' + 24) \|d\chi\|_{C^0(D)}^2, \, C_3 = 16C_3', \, C_4 = 16C_4',
$$
 we obtain
\beastar
\int_{D}\chi^2|\nabla(d_Hu)|^2
&\leq& C_1\int_{D_2}|d_Hu|^4+ C_2 \int_{D_2}|d_Hu|^2\\
&{}& + C_3 \int_{D^2} |du|^4 + C_4 \int_{D_2} |du|^2 + 2\CB.
\eeastar
By this inequality with $\int_{D_1}|\nabla(d_Hu)|^2\leq\int_{D}\chi^2|\nabla(d_Hu)|^2$,
we have finished the proof of Theorem \ref{thm:coercive-L2}.

\subsection{Proof of Lemma \ref{lem:boundary-integral}}
\label{subsec:boundary-integral}

The proof of Lemma \ref{lem:boundary-integral}
 is rather complex and tedious because
of the complication of the differential of the perturbed energy density function
$$
de_H = d \langle d_Hu, d_H u\rangle = d\langle du - X_H \otimes \gamma,  du - X_H \otimes \gamma \rangle
$$
where the expression $du - X_H \otimes \gamma$ is a $u^*TM$ valued one-form and $\langle \cdot, \cdot \rangle$ is
the inner product thereof.
In this calculation, we exercises the full power of
the defining properties of the contact triad connection exercises.

We first rewrite the integral
$$
\int_D -d(\chi^2 * de_H) = \int_{\del D} - \chi^2 *de_H = \int_{\del D} \chi^2 de_H \circ j
$$
by Stokes' formula and $-* de_H = de_H \circ j$.

We then take  an isothermal coordinate $(x,y)$ adapted to $\del \dot \Sigma$
such that $\del_x$ is tangent to $\del \dot \Sigma$ and $\del_y$ is normal thereto. With this coordinate, we can express
\bea\label{eq:int-delD}
 \int_{\del D} \chi^2 d e_H \circ j
& = & \int_{-\delta}^\delta  \chi^2 \frac{\del}{\del y} \langle d_Hu, d_H u \rangle\, dx
\nonumber\\
& = &  \int_{-\delta}^\delta  2 \chi^2   \langle \nabla_y (d_H u), d_H u\rangle\, dx
\eea
where we set the radius of the semi-disc to be $\delta$.
To unravel the integrand of this integral, we recall the identity
$$
\nabla (d_H u) = \nabla^\pi d_H^\pi u + \nabla(u^*\lambda_H) R_\lambda + u^*\lambda_H \nabla R_\lambda
$$
from \eqref{eq:nabladHu}. Then  we compute
\beastar
 \langle \nabla_y (d_H u), d_H  u\rangle & = &
 \langle  \nabla^\pi d_H^\pi u + \nabla_y (u^*\lambda_H) R_\lambda + u^*\lambda \nabla_y R_\lambda,
 d_H u \rangle \\
 & = & \underbrace{\langle \nabla_y^\pi d_H^\pi u, d_H u \rangle }_{(A)}
 + \underbrace{\langle \nabla_y (u^*\lambda_H), u^*\lambda_H \rangle}_{(B)}
 + \underbrace{\langle u^*\lambda \nabla_y R_\lambda, d_H u\rangle}_{(C)}
 \eeastar
 on $\del \dot \Sigma$.
 We estimate each term of the three separately.

 For the term (A), we note that $du(\del_y)$ is perpendicular to the Legendrian boundary.
 Therefore we have
 \be\label{eq:(A)}
 \langle \nabla_y d_H^\pi u, d_H^\pi u \rangle = - \langle B(d_H^\pi u,d_H^\pi u), du(\del_y)\rangle
 \ee
 on $\del \dot \Sigma$.
 For the term (C), we obtain
 \be\label{eq:(C)}
 |(C)| \leq  \|\nabla R_\lambda\|_{C^0} |du||d_H u| \|\lambda_H\|_{C^0} |du||d_Hu|.
 \ee
For the term (B), we have
\be\label{eq:(B)}
\langle \nabla_y (u^*\lambda_H), u^*\lambda_H \rangle =
\nabla_ y \left(\lambda_H \left(\frac{\del u}{\del x}\right)\right)\, \lambda_H\left(\frac{\del u}{\del x} \right)
+\nabla_y \left(\lambda_H\left(\frac{\del u}{\del y }\right)\right)\,
\lambda_H\left(\frac{\del u}{\del y} \right).
\ee
The estimation of the two terms here is similar. Since
the second summand is easier, we will focus on the study of the first summand and mention
briefly about the second term at the end.

We first write $\gamma = \gamma_x dx + \gamma_y dy$ and then
\beastar
\nabla_y\left(\lambda_H\left(\frac{\del u}{\del x}\right) \right)
& = & \nabla_y\left(\lambda\left(\frac{\del u}{\del x}\right)  + H \gamma_x \right)  \\
& = & \nabla_y\left(\lambda\left(\frac{\del u}{\del x} \right) \right)+ \nabla_y (H\gamma_x).
\eeastar
Obviously we estimate
\be\label{eq:nablaHgammax}
|\nabla_y(H\gamma_x)| \leq \|dH\|_{C^0} \|\gamma_x\|_{C^0} |du|.
\ee
We then rewrite
\be\label{eq:nablaylambda-x}
\nabla_y\left(\lambda\left(\frac{\del u}{\del x}\right) \right)
= (\nabla_y \lambda) \left(\frac{\del u}{\del x} \right)  + \lambda\left(\nabla_y \frac{\del u}{\del x}\right) .
\ee
The proof of the following extensively uses the defining properties of the contact triad
connection. (See the calculations around \cite[Lemma 4.8]{oh:contacton-Legendrian-bdy}
for some relevant calculations for the unperturbed contact instantons.)

\begin{lem} We have
\bea\label{eq:lambda-nabla-nabla}
\lambda\left(\nabla_y\frac{\del u}{\del x}\right)
& = & - \left \langle B\left(\left(\frac{\del u}{\del x} \right)^\pi,
\left(\frac{\del w}{\del y}\right)^\pi \right),R_\lambda \right \rangle
+ d\lambda\left(\left(\frac{\del u}{\del x}\right)^\pi,
\left(\frac{\del w}{\del y}\right)^\pi \right) \nonumber \\
\lambda\left(\nabla_y\frac{\del u}{\del y}\right)
& = & - \left\langle B\left(\left(\frac{\del u}{\del y} \right)^\pi,
\left(\frac{\del w}{\del y}\right)^\pi \right),R_\lambda \right \rangle
\eea
\end{lem}
\begin{proof} We recall $\frac{\del u}{\del y} = \left(\frac{\del u}{\del y}\right)^\pi +
\lambda\left(\frac{\del u}{\del y}\right) R_\lambda$. Using the definition of pull-back connection in general,
we re-write
$$
\nabla_y \frac{\del u}{\del x}\Big|_z = \nabla_{\left(\frac{\del u}{\del y}\right)^\pi(z)} Y
 +
\lambda\left(\frac{\del u}{\del y}\right) \nabla_{R_\lambda(u(z)} Y
$$
where $Y$ is a locally defined vector field on $M$ near $u(z)$ such that
$$
Y(u(z)) = \frac{\del u}{\del x}(z).
$$
Using the torsion property and the property $\nabla_{R_\lambda}X \in \xi$ for all $X$
(see Theorem \ref{thm:connection} (2) and (3) respectively), we compute
$$
\lambda\left(\nabla_{R_\lambda}Y\right) = \lambda(\nabla_{Y(z)} R_\lambda)=
\lambda\left(\nabla_{\frac{\del u}{\del x}} R_\lambda\right) = 0.
$$
Therefore using this vanishing,  we derive
$$
\lambda\left(\nabla_y \frac{\del u}{\del x}\right)
= \lambda\left(\nabla_{\left(\frac{\del u}{\del y}\right)^\pi} \left(\frac{\del u}{\del x} \right)\right).
$$
Then
\beastar
\lambda\left(\nabla_{\left(\frac{\del u}{\del y}\right)^\pi} \left(\frac{\del u}{\del x} \right)\right)
& = & \lambda\left(\nabla_{\left(\frac{\del u}{\del x}\right)} \left(\frac{\del u}{\del y} \right)^\pi \right)
+ \lambda \left(T\left (\left(\frac{\del u}{\del y}\right)^\pi, \left(\frac{\del u}{\del x}\right)\right)\right).
\eeastar
Since $du(\del_x) \in TR \subset \xi$ on $\del \dot \Sigma$, we have
$\frac{\del u}{\del x} = (\frac{\del u}{\del x})^\pi$ on $\del \dot \Sigma$.
Since $du(\del_y) = J du(\del_x)$, we also have $du(\del_y) \in \xi$ on $\del \dot \Sigma$.
Therefore we have derived
\beastar
\lambda\left(\nabla_y \frac{\del u}{\del x}\right)
& = &
 \lambda\left(\nabla_{\left(\frac{\del u}{\del x}\right)^\pi } \left(\frac{\del u}{\del y} \right)^\pi \right)
+ \lambda \left(T\left (\left(\frac{\del u}{\del y}\right)^\pi, \left(\frac{\del u}{\del x}\right)^\pi\right)\right)\\
& = & - \left\langle B\left(\left(\frac{\del u}{\del x} \right)^\pi,
\left(\frac{\del u}{\del y}\right)^\pi \right),R_\lambda \right \rangle
+ d\lambda\left(\left(\frac{\del u}{\del x}\right)^\pi,
\left(\frac{\del u}{\del y}\right)^\pi \right).
\eeastar
The other one can be proved more easily. The torsion term appearing for $\nabla_y \frac{\del u}{\del x}$
does not appear by the skew-symmetry for $\nabla_y \frac{\del u}{\del y}$.
This finishes the proof.
\end{proof}

By definition of $\lambda_H$,
$$
\left|\left(\nabla_y \left(\lambda_H\left(\frac{\del u}{\del x}\right) \right) \right)
\left(\lambda_H\left(\frac{\del u}{\del x}\right) \right)\right|
= \left|\nabla_y\left(\lambda\left(\frac{\del u}{\del x}\right) + H \gamma_x \right)  \right|\,
\left|\lambda_H\left(\frac{\del u}{\del x}\right)\right|.
$$
Using the above lemma and \eqref{eq:nablaylambda-x}, we estimate
$$
\left|\nabla_y\left(\lambda\left(\frac{\del u}{\del x}\right)\right) \right| \leq
 (\|B\|_{C^0}  + \|d\lambda\|_{C^0}) + \|\nabla\lambda\|_{C^0}) |du|^2.
$$
Furthermore we have
$$
\left|\lambda_H\left(\frac{\del u}{\del x}\right)\right|
 \leq \|\lambda_H\|_{C^0} |du|.
 $$
 and
 $$
 \left|\lambda_H\left(\frac{\del u}{\del y}\right)\right| \leq
 \|\lambda_H\|_{C^0} |du|.
 $$
Substituting all these into \eqref{eq:(B)} and summing up \eqref{eq:(A)}, \eqref{eq:(B)} and \eqref{eq:(C)},
we have derived the following:
\beastar
|(A)|&  \leq & (\|B\|_{C^0} + \|d\lambda\|_{C^0})|du|^2 + \|dH\|_{C^0} |du| +
\|_{C^0} \|\gamma\|_{C^2}) \|\lambda_H\|_{C^0} |du| |d_H u|\\
|(B)| &\leq &(\|B\|_{C^0} |du|^2 + \|dH\|_{C^0} |du| +
\|H\|_{C^0} \|\gamma\|_{C^2}) \|\lambda_H\|_{C^0} |du |d_Hu| |
\eeastar
By adding these up and integrating them we have obtained the inequality
\be\label{eq:boundary-integral}
\left| \int_{\del D} -\chi^2 * de_H|_{\del D} \right|  \leq
\int_{\del D_2} ( C_7|du|^3 |d_Hu| + C_8 |du|^2|d_Hu|  + C_9 |d_H u| |du|) =: \CB
\ee
if we set
\beastar
C_7 & =  &2\|B\|_{C^0} + \|d\lambda\|_{C^0} \|\lambda_H\|_{C^0} \\
C_8 & = & \|dH\|_{C^0} \|\lambda_H\|_{C^0} \\
C_9 & = & \|H\|_{C^0} \|\nabla^2 \gamma\|_{C^0} \|\lambda_H\|_{C^0}.
\eeastar
where $B = B_i$ is the second fundamental form of $R_i$. This finishes the proof of
Lemma \ref{lem:boundary-integral}.

\section{$C^{k,\alpha}$ coercive estimates and boundary conditions}
\label{sec:Wk+22-estimates}

Once we have established $W^{2,2}$ estimate, we follow the strategy taken in \cite{oh:contacton-Legendrian-bdy}
for the higher derivative estimates by expressing the following fundamental equation in
the isothermal coordinates of $(\dot \Sigma,j)$.  At this point, we will also
consider the case with Legendrian boundary conditions as in \cite{oh:contacton-Legendrian-bdy}.

We start from the equation \eqref{eq:dd_Hu2}
\bea\label{eq:ddu2}
d^{\nabla^\pi}(d_H^\pi u) & = & u^*\lambda\wedge(\frac{1}{2}\left(\CL_{R_\lambda} J)J d_H^\pi u\right) -
2 T^\pi(X_H^\pi(u), \gamma \wedge d_H^\pi u) \nonumber \\
&{}& + u^*\lambda_H \wedge(\frac{1}{2}\left(\CL_{R_\lambda} J)JX_H^\pi(u)\, \gamma\right)\nonumber\\
&{}& - d^{\nabla^\pi} (X_H^\pi \, \gamma).
\eea
\begin{rem} We observe that the three lines of the right hand side have different characteristics.
We recall the decomposition
$$
d_H u = d_H^\pi u + u^*\lambda_H \otimes R_\lambda.
$$
Then the first line is linear over the horizontal component $d_H^\pi u$, the second line if linear over
the vertical component $u^*\lambda_H = u^*\lambda + H \gamma$
and the third line may be regarded as an \emph{inhomogeneous term} thereof.
\end{rem}

Similarly as in \cite{oh-wang:connection}, \cite{oh:contacton-Legendrian-bdy},
in terms of any isothermal coordinates $z = x + i y$, we write
$$
\zeta = d_H^\pi u(\del_x),\quad \eta = d_H^\pi u(\del_y) .
$$
(In \cite{oh-wang:connection}, we had put $\zeta = d^\pi w$ for a \emph{unperturbed} contact instanton $w$.)

Then we have $\eta = J\zeta$ by the $J$-complex linearity of $d_H^\pi$. Therefore we have
$$
d^{\nabla^\pi}(d_H^\pi u)(\del_x,\del_y) = \nabla_x^\pi \eta - \nabla_y \zeta
= \nabla_x^\pi (J\zeta) - \nabla_y \zeta = J \nabla_x^\pi \zeta - \nabla_y \zeta.
$$
On the other hand, we derive
\beastar
u^*\lambda_H \wedge(\frac{1}{2}(\CL_{R_\lambda} J)J d_H^\pi u)(\del_x,\del_y)
& = & \frac12 \lambda(\frac{\del u}{\del x}) \CL_{R_\lambda}J J\eta -  \frac12 \lambda(\frac{\del u}{\del y}) \CL_{R_\lambda}J J\zeta\\
& = & - \frac12 \lambda(\frac{\del u}{\del x}) \CL_{R_\lambda}J \zeta -  \frac12 \lambda(\frac{\del u}{\del y}) \CL_{R_\lambda}J J\zeta
\eeastar
and
\beastar
 2 T^\pi(X_H^\pi(u), \gamma \wedge d_H^\pi u)(\del_x,\del_y)
= 2 T^\pi(X_H^\pi(u), \gamma_x \eta - \gamma_y \zeta).
\eeastar
This shows
\beastar
&{}&
\left(d^{\nabla^\pi}(d_H^\pi u) - u^*\lambda\wedge(\frac{1}{2}(\CL_{R_\lambda} J)J d_H^\pi u)
+ 2 T^\pi(X_H^\pi(u), \gamma \wedge d_H^\pi u)\right)(\del_x,\del_y) \nonumber \\\\
& = & J \left(\nabla_x^\pi \zeta + J \nabla_y^\pi \zeta -
\frac12 \lambda\left(\frac{\del u}{\del x} \right) \CL_{R_\lambda}\zeta
 + \frac12 \lambda\left(\frac{\del u}{\del y} \right) \CL_{R_\lambda}J J\zeta\right)\\
&{}& + 2 T^\pi(X_H^\pi(u), \gamma_x J\zeta - \gamma_y \zeta).
\eeastar
\begin{defn}[Linearization operator for $\delbar_H^\pi$] We call the linear operator
$D_{J,H}^\pi(u)$ defined by the expression
\beastar
D_{J,H}^\pi(u)(\del_x)(\zeta) & = & \nabla_x^\pi \zeta + J \nabla_y^\pi \zeta -
\frac12 \lambda\left(\frac{\del u}{\del y} \right) \CL_{R_\lambda}\zeta
 + \frac12 \lambda\left(\frac{\del u}{\del x} \right) \CL_{R_\lambda}J J\zeta \\
&{}& -2J T^\pi(X_H^\pi(u), \gamma_x J\zeta - \gamma_y \zeta)
\eeastar
for arbitrary section of the vector bundle $u^*\xi \to \dot \Sigma$ the
\emph{linearization} of the map $u \mapsto \del_{J,H}u$.
\end{defn}
We will see later that $D_{J,H}^\pi(u)$ is indeed the coordinate expression of
the linearization of the nonlinear perturbed Cauchy-Riemann map
$$
D\del_{J,H}^\pi(u): u \mapsto \del_{J,H}^\pi u
$$
in a precise sense.

Now in the isothermal coordinates $(x,y)$, \eqref{eq:ddu2} has the form
\be\label{eq:bootstrap}
\nabla_x^\pi \zeta + J \nabla_y^\pi \zeta + B(\lambda(u^*\lambda, \zeta) = -JPu^*\lambda_H
\ee
where $B$ is a operator of $\zeta$ given by
$$
B(u^*\lambda,\zeta) = - \frac12 \lambda\left(\frac{\del u}{\del y} \right) \CL_{R_\lambda}\zeta
+ \frac12 \lambda\left(\frac{\del u}{\del x} \right) (\CL_{R_\lambda}J) J\zeta
-2J T^\pi(X_H^\pi(u), \gamma_x J\zeta - \gamma_y \zeta)
$$
and
\be\label{eq:P(u)}
Pu^*\lambda_H = u^*\lambda_H \wedge \left(\frac12(\CL_{R_\lambda}J)JX_H^\pi(u)\otimes \gamma\right)(\del_x,\del_y).
\ee
We note that by the Sobolev embedding, $ W^{2,2} \hookrightarrow C^{0,\alpha}$ for $0 \leq \alpha < 1/2$.
Therefore we start from $C^{0,\alpha}$ bound with $0 < \alpha <1/2$ and will inductively bootstrap it to
$C^{k, \alpha}$ bounds for $k \geq 1$ by a Schauder-type estimates.

\begin{thm}\label{thm:local-regularity} Assume $k\geq 1$.
Let $u$ be a perturbed contact instanton satisfying \eqref{eq:contacton-Legendrian-bdy}.
Then for any pair of domains $D_1 \subset D_2 \subset \dot \Sigma$ such that $\overline{D_1}\subset D_2$, we have
$$
\|du\|_{C^{k,\alpha}(D_1)} \leq C \|d_H u\|_{W^{1,2}(D_2)} + C'
$$
for some constants $C, \, C'> 0$ depending on $k$ and $J$, $\lambda$ and $D_1, \, D_2$ but independent of $u$.
\end{thm}

WLOG, we assume that $D_2 \subset \dot \Sigma$ is a semi-disc with $\del D \subset \del \dot \Sigma$
and equipped with an isothermal coordinates $(x,y)$ such that
$$
D_2 = \{ (x,y) \mid |x|^2 + |y|^2 < \delta, \, y \geq 0\}
$$
for some $\delta > 0$
and so $\del D_2 \subset \{(x,y) \in D \mid y = 0\}$. Assume $D_1 \subset D_2$
is the semi-disc with radius $\delta /2$.
We  consider the complex-valued function
\be\label{eq:alpha}
\alpha(x,y) = \lambda_H\left(\frac{\del u}{\del y}\right)
+ \sqrt{-1}\left(\lambda_H\left(\frac{\del u}{\del x}\right)\right)
\ee
as in \cite[Subsection 11.5]{oh-wang:CR-map2}. We note that since $u$ satisfies the
Legendrian boundary condition and $\gamma(\del_x) = 0$ along $\del \dot \Sigma$, we have
\be\label{eq:lambda(delw)=0}
\lambda_H\left(\frac{\del u}{\del x}\right) = 0
\ee
on $\del D_2$.

\begin{lem} Let $\zeta = d_H^\pi u(\del x)$ and $\alpha$ be as above.
Then $\alpha$ satisfies the equations
\be\label{eq:atatau-equation}
\begin{cases}
\delbar \alpha
= \frac{1}{2}|\zeta|^2 + G(du,H) \\
\alpha(z) \in \R \quad z \in \del D_2
\end{cases}
\ee
where $\delbar=\frac{1}{2}\left(\frac{\del}{\del x}+\sqrt{-1}\frac{\del}{\del y}\right)$
is the standard Cauchy-Riemann operator for the standard complex structure $J_0=\sqrt{-1}$,
and
$$
G(du,H) = (u^*(R_\lambda[H]\, \lambda)\wedge \gamma + u^*H\, d\gamma)(\del_x,\del_y).
$$
\end{lem}
\begin{proof} Recall \eqref{eq:du*lambdaH}
$$
d(u^*\lambda_H) = \frac12 |d_H u|^2\, dA + u^*(R_\lambda[H]\, \lambda)\wedge \gamma + u^*H\, d\gamma
$$
here. Since $d\gamma = 0$ by hypothesis and $d(u^*\lambda) = d(u^*\lambda) + d(u^*H\, \gamma)$, we have proved the lemma
by evaluating this against the pair $(\del_x,\del_y) = (\del_x,j\del_y)$.
\end{proof}

Summarizing the above discussion, we have proved the following:

\begin{prop}[Fundamental equation in isothermal coordinates]\label{prop:FE-in-isothermal}
Let $u$ be a solution to \eqref{eq:contacton-Legendrian-bdy}, and let $\zeta = d_H^\pi u(\del_x)$
and $\alpha$ as above in an isothermal coordinate $(x,y)$ of $(\dot \Sigma,j)$.
Then  the pair $(\zeta,\alpha)$ satisfies
\be\label{eq:main-eq-isothermal}
\begin{cases}
\nabla_x^\pi \zeta + J \nabla_y^\pi \zeta + B(u^*\lambda, \zeta) = -* JP(u,\alpha) \\
\zeta(z) \in TR_i \quad \text{for } \, z \in \del D_2
\end{cases}
\ee
 and
\be\label{eq:equation-for-alpha}
\begin{cases}
\delbar \alpha
=\frac12 |\zeta|^2 + G(du,H)\\
\alpha(z) \in \R \quad \text{for } \, z \in \del D_2
\end{cases}
\ee
where
$$
G(du,H) = (u^*[R_\lambda[H] \lambda)\wedge \gamma + u^*H d\gamma)(\del_x,\del_y).
$$
\end{prop}

Now we are ready to establish the higher regularity estimate.

\begin{proof}[Proof of Theorem \ref{thm:local-regularity}]
Recall the expressions  of $B(u^*\lambda,\zeta)$, $P(u,\alpha)$ and $G(du,H)$ above.
By writing the Reeb component $u^*\lambda$ in the isothermal coordinate $(x,y)$ as
$$
f = \lambda_H\left(\frac{\del u}{\del x}\right) , \quad g = \lambda_H\left(\frac{\del u}{\del y}\right)
$$
we have $\alpha = g + \sqrt{-1}f$ for the $\alpha$ defined above in \eqref{eq:alpha}, and
$$
B(u^*\lambda,\zeta) = - \frac12  g\, \CL_{R_\lambda}J \,\zeta
+ \frac12 f (\CL_{R_\lambda}J) J\zeta
-2J T^\pi(X_H^\pi(u), \gamma_x J\zeta - \gamma_y \zeta)
$$
and hence $\zeta$ satisfies
\be\label{eq:equation-for-zeta}
\nabla_x^\pi \zeta + J \nabla_y^\pi \zeta - \frac12 \lambda g \CL_{R_\lambda}\zeta
+ \frac12 f (\CL_{R_\lambda}J) J\zeta
-2J T^\pi(X_H^\pi(u), \gamma_x J\zeta - \gamma_y \zeta) = -* JP(u,\alpha).
\ee
We first observe that
The two equations \eqref{eq:main-eq-isothermal} and \eqref{eq:equation-for-alpha}
together form a nonlinear elliptic system for $(\zeta,\alpha)$ which are coupled: $\alpha = g + \sqrt{-1} f$
is fed into \eqref{eq:main-eq-isothermal} through its coefficients $f$ and $g$,
and then $\zeta$ provides the input
for the equation \eqref{eq:equation-for-alpha} and then back and forth. Using this structure of
coupling, we obtain the higher derivative estimates by alternating boot strap arguments between $\zeta$ and $\alpha$
the detail of which is now in order.

It is obvious to see that \eqref{eq:main-eq-isothermal} is an \emph{inhomogeneous}
linear elliptic equation for $\zeta$ with inhomogeneous term $-*JP(u)$
in the right hand side: Here we have
\beastar
\frac{1}{2} \lambda(J\zeta)(\CL_{R_\lambda}J)
& = & \frac{1}{2} f \,\CL_{R_\lambda}J\\
\frac{1}{2} \lambda(\zeta)(\CL_{R_\lambda}J)
& = & \frac{1}{2} g\, \CL_{R_\lambda}J J.
\eeastar
Then Theorem  \ref{thm:coercive-L2} is translated into the following
$W^{1,2}$ bound for $\zeta$.

\begin{lem}\label{lem:W12-zeta} We have
$$
\|\zeta\|_{W^{1,2}}^2 \leq C_1 \|\zeta\|_{L^4}^4 + C_2 \|\zeta\|_{L^2}^2 + C_3.
$$
\end{lem}
We then consider the equation \eqref{eq:equation-for-zeta} for $\zeta$ and
 obtain the estimate
$$
\|\overline \nabla^\pi \zeta\|_{W^{2,2}}^2 \leq C_1'(\|f\|_{W^{1,4}}^2 + \|g\|_{W^{1,4}}\|\zeta\|_{W^{1,4}}
 + C_2'\|\CL_{R_\lambda}\|_{C^0}\|\gamma\|_{C^0} \|\zeta\|_{W^{1,2}}.
$$
Then by the standard estimate for the Riemann-Hilbert problem
with Dirichlet boundary condition for the Cauchy-Riemann operator
$\nabla_x^\pi + J \nabla_y^\pi J=: \overline \nabla^\pi$, we derive
$$
\|\zeta\|_{C^{0,\delta}}^2 \leq C_1'(\|f\|_{W^{1,4}}^2 + \|g\|_{W^{1,4}})\|\zeta\|_{W^{1,4}}
 + C_2'\|\CL_{R_\lambda}\|_{C^0}\|\gamma\|_{C^0}) \|\zeta\|_{W^{1,2}}
$$
with say $\delta = 2/5 < 1/2$.

Next we consider \eqref{eq:equation-for-alpha} which is another
the Riemann-Hilbert problem with the real (or imaginary) boundary condition for the
Cauchy-Riemann operator $\delbar$.
Then again by the standard estimate for the Riemann-Hilbert problem
with the real (or imaginary) boundary condition, we derive
\be\label{eq:|alpha|22}
\|\alpha\|_{C^{0,\delta}} \leq C_3 \|\zeta\|^2_{C^{0,\delta}} + C_4 \|G(du,H)\|_{W^{1,2}}.
\ee
But we have
\beastar
G(du,H) & = & (u^*(R_\lambda[H]\, \lambda)\wedge \gamma + u^*H\, d\gamma)(\del_x,\del_y)\\
& = & (R_\lambda[H])(u) f \gamma_y - (R_\lambda[H])(u)g \gamma_x + H(u)(d\gamma)(\del_x,\del_y).
\eeastar
By differentiating this identity, it is easy to check
$$
\|G(du,H)\|_{W^{2,1}} \leq C_5 \|du\|_{L^4}(\|f\|_{L^4} + \|g\|_{L^4}) + C_6(\|f\|_{W^{1,2}} + \|g\|_{W^{1,2}}) + C_7\|du\|_{L^2} + C_8
$$
for some constants $C_5$ - $C_8$. We have already shown that the right hand side of this inequality is bounded by
$\|\zeta\|_{C^{(0,\delta}}, \|f\|_{L^4}, \|g\|_{L^4}$. Combining Lemma \ref{lem:W12-zeta}, \eqref{eq:|alpha|22}, we have
a bound for $\|\alpha\|_{W^{2,2}}$ and hence a bound for $\|\alpha\|_{C^{0,\delta}}$
in terms of $\|\zeta\|_{C^{0,\delta}}, \|f\|_{L^4}, \|g\|_{L^4}$.

At this stage, \eqref{eq:main-eq-isothermal} implies that there exist some constants $C, \, C' > 0$
such that
$$
\|\overline \nabla^\pi \zeta\|_{C^{0,\delta}} \leq C \|\zeta \|_{C^{0,\delta}} + C'
$$
for all solutions thereof. By the standard Schauder estimate (see \cite{gilbarg-trudinger}
for example), applied to \eqref{eq:main-eq-isothermal}, any solution $\zeta$ thereof indeed satisfies
$$
\|\zeta\|_{C^{1,\delta}} \leq  C (\|\overline \nabla^\pi \zeta\|_{C^{0,\delta}} + \|\zeta \|_{C^{0,\delta}}).
$$
Combining the two, we have derived
\be\label{eq:zeta-1delta}
\|\zeta\|_{C^{1,\delta}}\leq C \|\zeta \|_{C^{0,\delta}} + C'
\ee
for all solutions of \eqref{eq:equation-for-zeta}.

By substituting $\zeta$ back into \eqref{eq:equation-for-alpha}, we get a similar $C^{1,\delta}$
bound for $\alpha$ from \eqref{eq:zeta-1delta}. Repeating this alternating process between \eqref{eq:equation-for-zeta}
and \eqref{eq:equation-for-alpha}, we have established the $C^k$-estimate for all $k \geq 1$.
This finishes the proof.
\end{proof}

\begin{rem} Recall that we have assumed $u$ is a \emph{classical} solution to \eqref{eq:contacton-Legendrian-bdy}
which in particular means a smooth solution.
The above argument equally applies to any \emph{weak solution} to \eqref{eq:contacton-Legendrian-bdy} $u$
\emph{as long as $u$ is in $W^{1,4}$} (see Lemma \ref{lem:W12-zeta}) so that $\zeta$, $f$ and $g$ are contained in $L^4$. In other words,
to perform the above alternating bootstrap argument, $u$ should have regularity at least of $W^{1,4}$.
We refer to \cite[Section 8.5]{oh:book1} for a full discussion on the boundary regularity theorem
for the weak totally real boundary value problem for  $J$-holomorphic curves in symplectic geometry.
The same kind of proof also applies to \eqref{eq:contacton-Legendrian-bdy} along the same line following the
above alternating boot-strapping process whose details we leave to the readers.
\end{rem}

\section{Vanishing of asymptotic charge and subsequence convergence}
\label{sec:subsequence-convergence}

In this section, we study the asymptotic behavior of contact instantons
on the Riemann surface $(\dot\Sigma, j)$ associated with a metric $h$ with \emph{strip-like ends}.
To be precise, we assume there exists a compact set $K_\Sigma\subset \dot\Sigma$,
such that $\dot\Sigma-\Int(K_\Sigma)$ is a disjoint union of punctured semi-disks
 each of which is isometric to the half strip $[0, \infty)\times [0,1]$ or $(-\infty, 0]\times [0,1]$, where
the choice of positive or negative strips depends on the choice of analytic coordinates
at the punctures.
We denote by $\{p^+_i\}_{i=1, \cdots, l^+}$ the positive punctures, and by $\{p^-_j\}_{j=1, \cdots, l^-}$ the negative punctures.
Here $l=l^++l^-$. Denote by $\phi^{\pm}_i$ such strip-like coordinates.
We first state our assumptions for the study of the behavior of boundary punctures.
(The case of interior punctures is treated in \cite[Section 6]{oh-wang:CR-map1}.)

\begin{defn}Let $\dot\Sigma$ be a boundary-punctured Riemann surface of genus zero with punctures
$\{p^+_i\}_{i=1, \cdots, l^+}\cup \{p^-_j\}_{j=1, \cdots, l^-}$ equipped
with a metric $h$ with \emph{strip-like ends} outside a compact subset $K_\Sigma$.
Let
$u: \dot \Sigma \to M$ be any smooth map with Legendrian boundary condition
\eqref{eq:perturbed-contacton-bdy}.
We define the total $\pi$-harmonic energy $E^\pi(u)$
by
\be\label{eq:endenergy}
E_H^\pi(u) =  \frac{1}{2} \int_{\dot \Sigma} e^{g_{H, u}} |d^\pi u - X_H \otimes \gamma|^2
\ee
where the norm is taken in terms of the given metric $h$ on $\dot \Sigma$ and the triad metric on $M$.
\end{defn}

We put the following hypotheses in our asymptotic study of the finite
energy contact instanton maps $u$ as in \cite{oh-wang:CR-map1}, \emph{except not requiring the charge vanishing
condition $Q = 0$, which itself we will prove here under the hypothesis using the Legendrian boundary condition}:

Throughout this section,
we put the following hypotheses in our asymptotic study of the finite
energy contact instanton maps $u$ similarly as in \cite{oh-wang:CR-map1}, \cite{oh:contacton-Legendrian-bdy}.

\begin{hypo}\label{hypo:basic}
Let $h$ be the metric on $\dot \Sigma$ given above.
Assume $u:\dot\Sigma\to M$ satisfies the contact instanton equation
\eqref{eq:contacton-Legendrian-bdy}
and
\begin{enumerate}
\item $E_H^\pi(u)<\infty$ (finite $\pi$-energy);
\item $\|d u\|_{C^0(\dot\Sigma)}, \quad \|d_H\|_{C^0(\dot \Sigma)} <\infty$.
\item $\Image u \subset K \subset M$ for some compact set $K$.
\end{enumerate}
\end{hypo}
We then consider the asymptotic invariants at each puncture defined as
\bea
T_H & := & \frac{1}{2}\int_{[0,\infty) \times [0,1]} e^{g_{H,u}} |d_H ^\pi u|^2
+ \int_{\{0\}\times [0,1]}e^{g_{H,u}} (u^*\lambda + H\, dt)|_{\{0\}\times [0,1]})\label{eq:TQ-T}\\
Q_H & : = & \int_{\{0\}\times [0,1]}e^{g_{H,u}} ((u|_{\{0\}\times [0,1] })^*(\lambda + H\, dt)\circ j)\label{eq:TQ-Q}
\eea
in the given strip-like coordinates.
(The case of negative punctures is similar and omitted.)

Recall that for any (unperturbed) contact instanton, we have the pointwise identity
$$
d_Hu^*\lambda = \frac12 |d^\pi w|^2\, dA
$$
which enters in the subsequence convergence result for the contact instantons in \cite{oh-wang:CR-map1}
in a crucial way.

The following proposition is the correct generalization of this identity
to the Hamiltonian contact instantons.

\begin{prop}\label{prop:du*lambda}
Let $\gamma = dt$ for the strip-like coordinate on a neighborhood of
a given puncture of $\dot \Sigma$. Then on the strip-like neighborhood, we have
\be\label{eq:du*lambda=}
d(\overline u^*\lambda) = \frac12 e^{g_{H, u}}|(du - X_H \otimes \gamma)^\pi|^2 \, dA
\ee
for any map $u$ satisfying $(du - X_H \otimes \otimes dt)^{\pi(0,1)} = 0$.
\end{prop}
\begin{proof} We recall the identity
$$
 e^{g_{H, u}} u^*(\lambda + H\, dt) = \overline u^*\lambda.
$$
Since $\delbar^\pi u = 0$, we have
$$
 d(\overline u^*\lambda) = \frac12 |d^\pi \overline u|^2\, dA
= \frac 12 e^{  g_{H, u}} |d^\pi u - X_H(u)\otimes dt|_J^2 \, dA
$$
which finishes the proof.
\end{proof}

\begin{rem}\label{rem:TQ}
In particular \eqref{eq:du*lambda=} holds for any perturbed contact instanton $u$.  By Stokes' formula, we can express
$$
T_H = \frac{1}{2}\int_{[s,\infty) \times [0,1]} e^{  g_{H, u}}|d_H^\pi u|^2
+ \int_{\{0\}\times [0,1]}e^{g_{H, u}}(u^*\lambda + H\, dt)|_{\{0\}\times [0,1]}), \quad
\text{for any } s\geq 0.
$$
Moreover, since $u$ satisfies $d(u^*\lambda_H\circ j)=0$ and the Legendrian boundary condition, it follows that the integral
$$
Q_H = \int_{\{s \}\times [0,1]}e^{ g_{H, u}}(u^*\lambda + H\, dt)\circ j|_{\{0\}\times [0,1]}, \quad
\text{for any } s \geq 0
$$
does not depend on $s$ whose common value is nothing but $Q_H$.
\end{rem}
We call $T_H$ the \emph{asymptotic (perturbed contact )action}
and $Q$ the \emph{asymptotic (perturbed contact) charge} of the perturbed contact instanton
$u$ at the given puncture.

The proof of the subsequence convergence result largely follows the scheme of
\cite[Theorem 6.4]{oh-wang:CR-map1} with replacement of $d_Hu$ by $d \overline u$ and others,
especially relying on the key identity \eqref{eq:du*lambda=}. Therefore we just indicate
how we can reduce the proof of this theorem to that of \cite[Theorem 6.5]{oh:contacton-Legendrian-bdy}
referring to the proof thereof for complete details.

\begin{thm}[Subsequence Convergence]\label{thm:subsequence}
Let $u:[0, \infty)\times [0,1]\to M$ satisfy the contact instanton equations
\eqref{eq:perturbed-contacton-bdy}
and Hypothesis \eqref{hypo:basic}.

Then for any sequence $s_k\to \infty$, there exists a subsequence, still denoted by $s_k$, and a
massless instanton $u_\infty(\tau,t)$ (i.e., $E^\pi(u_\infty) = 0$)
on the cylinder $\R \times [0,1]$  that satisfies the following:
\begin{enumerate}
\item $\delbar_H^\pi u_\infty = 0$ and
$$
\lim_{k\to \infty}u(s_k + \tau, t) = u_\infty(\tau,t)
$$
in the $C^l(K \times [0,1], M)$ sense for any $l$, where $K\subset [0,\infty)$ is an arbitrary compact set.
\item $u_\infty$ has vanishing asymptotic charge $Q = 0$ and satisfies
$$
u_\infty(\tau,t)= \psi_H^t(\psi_H^1)_*\phi_{R_\lambda^{Tt}}(p)
$$
for some $p \in R_0$.
\end{enumerate}
\end{thm}
\begin{proof} Let $u$ be a perturbed contact instanton with Legendrian boundary condition
$(R_0,\cdots, R_k)$. Consider a boundary puncture of $\dot \Sigma$ and let
$(\tau,t) \in [0,\infty) \times [0,1]$ be a strip-like coordinate at the puncture.
We then make the coordinate change
$$
\overline u(\tau,t): = (\phi_H^t)^{-1}(u(\tau,t)), \quad \phi_H^t = \psi_H^t (\psi_H^1)^{-1}.
$$
Then $\overline u$ satisfies
$$
\begin{cases}
\delbar^\pi \overline u = 0 , \quad d(\overline u^*\lambda \circ j) = 0\\
\overline u(\tau,0) \in \psi_H^1(R_0), \quad \overline u(\tau,1) \in R_1.
\end{cases}
$$
By applying \cite[Theorem 6.5]{oh:contacton-Legendrian-bdy}, we derive that
for any sequence $s_k\to \infty$, there exists a subsequence, still denoted by $s_k$, and a
massless instanton $\overline u_\infty(\tau,t)$ (i.e., $E^\pi(\overline u_\infty) = 0$)
on the cylinder $\R \times [0,1]$  that satisfies the following:
\begin{enumerate}[(a)]
\item $\delbar^\pi \overline u_\infty = 0$ and
$$
\lim_{k\to \infty}\overline u(s_k + \tau, t) = \overline u_\infty(\tau,t)
$$
in the $C^l(K \times [0,1], M)$ sense for any $l$, where $K\subset [0,\infty)$ is an arbitrary compact set.
\item $\overline u_\infty$ has vanishing asymptotic charge $Q_H = 0$ and satisfies
$$
\overline u_\infty(\tau,t)= \overline \gamma(t)
$$
 for some \emph{iso-speed Reeb chord} $(\gamma,T)$ satisfying
 $$
\dot{\overline \gamma}(t) = T R_\lambda (\overline \gamma(t))
$$
 and joining $\psi_H^1(R_0)$ and $R_1$ with period $T = \int \gamma^*\lambda$
 at each puncture.
\end{enumerate}

Applying back the gauge transformation to $\overline u$, the map
$$
u(\tau,t) = \psi_H^t(\psi_H^1)^{-1} (\overline u(\tau,t))
$$
is a solution to \eqref{eq:contacton-Legendrian-bdy}.
In particular,  the curve
$$
\gamma(t) : =   \psi_H^t(\psi_H^1)^{-1} (\overline \gamma(t) )
$$
satisfies  the boundary condition
$$
\gamma(0) \in R_0, \quad \gamma(1) \in R_1,
$$
and $\overline \gamma$ is a Reeb chord and so has the form
$$
\overline \gamma(t): = \phi_{R_\lambda}^{Tt}(q)
$$
for some $q \in \psi_H^1(R_0), \, \phi_{R_\lambda}^{T}(q)\in R_1$. Writing $q = \psi_H^1(p)$,
we have $p \in R_0$, we obtain the expression
$$
\gamma(t) =  \psi_H^t(\psi_H^1)^{-1} (\phi_{R_\lambda}^{Tt} (\psi_H^1(p)) = \psi_H^t (\psi_H^1)_*\phi_{R_\lambda}^{Tt}(p)
$$
which finishes the proof.
\end{proof}

\begin{rem} If we take another gauge transformation
$$
\Phi': \Omega(R_0,R_1) \to \Omega(R_0, \psi_H^1(R_1))
$$
defined by
$$
\Phi'(\ell)(t): = \psi_H^t(\ell(t)),
$$
the above asymptotic limit curve $\gamma'$ will have the form
$$
\gamma'(t) = \psi_H^t \phi_{R_\lambda}^{T't}(p'), \, p' \in R_0, \quad \phi_{R_\lambda}^{T'}(p') \in \psi_H^1(R_1).
$$
\end{rem}

\appendix

\section{Review of the contact triad connection}
\label{sec:connection}

Assume $(M, \lambda, J)$ is a contact triad for the contact manifold $(M, \xi)$, and equip with it the contact triad metric
$g=g_\xi+\lambda\otimes\lambda$.
In \cite{oh-wang:connection}, the authors introduced the \emph{contact triad connection} associated to
every contact triad $(M, \lambda, J)$ with the contact triad metric and proved its existence and uniqueness.

\begin{thm}[Contact Triad Connection \cite{oh-wang:connection}]\label{thm:connection}
For every contact triad $(M,\lambda,J)$, there exists a unique affine connection $\nabla$, called the contact triad connection,
 satisfying the following properties:
\begin{enumerate}
\item The connection $\nabla$ is  metric with respect to the contact triad metric, i.e., $\nabla g=0$;
\item The torsion tensor $T$ of $\nabla$ satisfies $T(R_\lambda, \cdot)=0$;
\item The covariant derivatives satisfy $\nabla_{R_\lambda} R_\lambda = 0$, and $\nabla_Y R_\lambda\in \xi$ for any $Y\in \xi$;
\item The projection $\nabla^\pi := \pi \nabla|_\xi$ defines a Hermitian connection of the vector bundle
$\xi \to M$ with Hermitian structure $(d\lambda|_\xi, J)$;
\item The $\xi$-projection of the torsion $T$, denoted by $T^\pi: = \pi T$ satisfies the following property:
\be\label{eq:TJYYxi}
T^\pi(JY,Y) = 0
\ee
for all $Y$ tangent to $\xi$;
\item For $Y\in \xi$, we have the following
\be\label{eq:dnablaYR}
\del^\nabla_Y R_\lambda:= \frac12(\nabla_Y R_\lambda- J\nabla_{JY} R_\lambda)=0.
\ee
\end{enumerate}
\end{thm}
From this theorem, we see that the contact triad connection $\nabla$ canonically induces
a Hermitian connection $\nabla^\pi$ for the Hermitian vector bundle $(\xi, J, g_\xi)$,
and we call it the \emph{contact Hermitian connection}.

Moreover, the following fundamental properties of the contact triad connection was
proved in \cite{oh-wang:connection}, which will be useful to perform tensorial calculations later.

\begin{cor}\label{cor:connection}
Let $\nabla$ be the contact triad connection. Then
\begin{enumerate}
\item For any vector field $Y$ on $M$,
\be\label{eq:nablaYX}
\nabla_Y R_\lambda = \frac{1}{2}(\CL_{R_\lambda}J)JY;
\ee
\item $\lambda(T|_\xi)=d\lambda$.
\end{enumerate}
\end{cor}

We refer readers to \cite{oh-wang:connection} for more discussion on the contact triad connection and its relation with other related canonical type connections.

\section{Covariant calculus of vector valued forms}
\label{sec:vectorvalued-forms}

In this appendix, we recall the standard covariant calculus of vector-valued
differential forms and the Weitzenb\"ock formulae presented
as in \cite{wells-book}, \cite{freed-uhlenbeck} applied to our current circumstance.
More specifically, the proofs of all lemmata stated in this appendix are included in
\cite[Appendix A \& B]{oh-wang:CR-map1} to which we refer readers for a quick
reference.

Assume $(P, h)$ is a Riemannian manifold of dimension $n$ with metric $h$, and $D$
is the Levi-Civita connection. In our case, $(P,h) = (\dot \Sigma, h)$.
Let $E\to P$ be any vector bundle with inner product $\langle\cdot, \cdot\rangle$,
and assume $\nabla$ is a connection on $E$ which is compatible with $\langle\cdot, \cdot\rangle$.

For any $E$-valued form $s$, calculating the (Hodge) Laplacian of the energy density
of $s$,  we get
\beastar
-\frac{1}{2}\Delta|s|^2=|\nabla s|^2+\langle Tr\nabla^2 s, s\rangle,
\eeastar
where for $|\nabla s|$ we mean the induced norm in the vector bundle $T^*P\otimes E$, i.e.,
$|\nabla s|^2=\sum_i|\nabla_{E_i}s|^2$ with $\{E_i\}$ an orthonormal frame of $TP$.
$Tr\nabla^2$ denotes the connection Laplacian, which is defined as
$Tr\nabla^2=\sum_i\nabla^2_{E_i, E_i}s$,
where $\nabla^2_{X, Y}:=\nabla_X\nabla_Y-\nabla_{\nabla_XY}$.

Denote by $\Omega^k(E)$ the space of $E$-valued $k$-forms on $P$. The connection $\nabla$
induces an exterior derivative by
\beastar
d^\nabla&:& \Omega^k(E)\to \Omega^{k+1}(E)\\
d^\nabla(\alpha\otimes \zeta)&=&d\alpha\otimes \zeta+(-1)^k\alpha\wedge \nabla\zeta.
\eeastar

It is not hard to check that for any $1$-forms, equivalently one can write
$$
d^\nabla\beta (v_1, v_2)=(\nabla_{v_1}\beta)(v_2)-(\nabla_{v_2}\beta)(v_1),
$$
where $v_1, v_2\in TP$.

We extend the Hodge star operator to $E$-valued forms by
\beastar
*&:&\Omega^k(E)\to \Omega^{n-k}(E)\\
*\beta&=&*(\alpha\otimes\zeta)=(*\alpha)\otimes\zeta
\eeastar
for $\beta=\alpha\otimes\zeta\in \Omega^k(E)$.

Define the Hodge Laplacian of the connection $\nabla$ by
$$
\Delta^{\nabla}:=d^{\nabla}\delta^{\nabla}+\delta^{\nabla}d^{\nabla},
$$
where $\delta^{\nabla}$ is defined by
$$
\delta^{\nabla}:=(-1)^{nk+n+1}*d^{\nabla}*.
$$
The following lemma is important for the derivation of the Weitzenb\"ock formula.
\begin{lem}\label{lem:d-delta}Assume $\{e_i\}$ is an orthonormal frame of $P$, and $\{\alpha^i\}$ is the dual frame.
Then we have
\beastar
d^{\nabla}&=&\sum_i\alpha^i\wedge \nabla_{e_i}\\
\delta^{\nabla}&=&-\sum_ie_i\rfloor \nabla_{e_i}.
\eeastar
\end{lem}

Next  we recall the conventional notation of
the wedge product of $E$-valued forms over the Euclidean vector bundle $E$
which is given by
\beastar
\wedge&:&\Omega^{k_1}(E)\times \Omega^{k_2}(E)\to \Omega^{k_1+k_2}(E)\\
\beta_1\wedge\beta_2&=&\langle \zeta_1, \zeta_2\rangle\,\alpha_1\wedge\alpha_2,
\eeastar
where $\beta_1=\alpha_1\otimes\zeta_1\in \Omega^{k_1}(E)$ and $\beta_2=\alpha_2\otimes\zeta_2\in \Omega^{k_2}(E)$
are $E$-valued forms. Then we also have the following list of lemmata
\begin{lem}\label{lem:inner-star}
For $\beta_1, \beta_2\in \Omega^k(E)$,
$$
\langle \beta_1, \beta_2\rangle=*(\beta_1\wedge *\beta_2).
$$
\end{lem}

\begin{lem}
$$
d(\beta_1\wedge\beta_2)=d^\nabla\beta_1\wedge \beta_2+(-1)^{k_1}\beta_1\wedge d^\nabla\beta_2,
$$
where $\beta_1\in \Omega^{k_1}(E)$ and $\beta_2\in \Omega^{k_2}(E)$
are $E$-valued forms.
\end{lem}

\begin{lem}\label{lem:metric-property}
Assume $\beta_0\in \Omega^k(E)$ and $\beta_1\in \Omega^{k+1}(E)$, then we have
$$
\langle d^\nabla \beta_0, \beta_1\rangle-(-1)^{n(k+1)}\langle \beta_0, \delta^\nabla\beta_1\rangle
=*d(\beta_0\wedge *\beta_1).
$$
\end{lem}

\section{Conversion to the intersection theoretic version}
\label{sec:conversion}

In this appendix, we give the proof of Proposition \ref{prop:Ham=Reeb} whose
statement we reproduce here.

\begin{prop}\label{prop:Ham=Reeb-appendix} Let $J_0 \in \CJ(\lambda)$ and $J_t$ defined as
in \eqref{eq:Jt}. We equip $(\Sigma,j)$ a K\"ahler metric $h$. Let $g_{H,u}$ be the
function defined in \eqref{eq:gHu}.
Suppose $u$ satisfies
\be\label{eq:perturbed-contacton-bdy-appendix}
\begin{cases}
(du - X_H \otimes dt)^{\pi(0,1)} = 0, \quad d(e^{g_{H, u}}(u^*\lambda + H\, dt)\circ j) = 0\\
u(\tau,0) \in R_0, \quad u(\tau,1) \in R_1
\end{cases}
\ee
with respect to $J_t$. Then $\overline u$ satisfies
\be\label{eq:perturbed-intersection-appendix}
\begin{cases}
\delbar^\pi \overline u = 0, \quad d((\overline u^*\lambda \circ j) = 0 \\
\overline u(\tau,0) \in R_0, \, \overline u(\tau,1) \in \psi_H^1(R_1)
\end{cases}
\ee
for $J_0$.
\end{prop}
\begin{proof}  We first note the equality
$$
2\delbar^\pi \overline u(\del_\tau) = \left(\frac{\del \overline u}{\del \tau}\right)^\pi +
J_0\left(\frac{\del \overline u}{\del t}\right)^\pi
$$
for any smooth map $u$. We compute each term of the derivatives
$\frac{\del \overline u}{\del \tau}$ and $\frac{\del \overline u}{\del t}$
separately.

By definition, we have
$$
u(\tau,t) = \psi_H^t(\overline u(\tau,t))).
$$
Then
$$
d u(\del_t)= \frac{\del u}{\del t}
= d\psi_H^t\left(\frac{\del \overline u}{\del t}\right) + X_H(u(\tau,t))
$$
and hence we have
\be\label{eq:dbarudt}
\frac{\del \overline u}{\del t} = d(\psi_H^t)^{-1}
\left(\frac{\del u}{\del t} - X_H(u(\tau,t))\right).
\ee
And more easily, we compute
\be\label{eq:dbarudtau}
\frac{\del \overline u}{\del \tau} = d(\psi_H^t)^{-1}\left(\frac{\del u}{\del \tau}\right).
\ee
By adding up the two, we obtain
$$
2\delbar \overline u(\del_\tau) = 2 d(\psi_H^t)^{-1}\left(
\left(\frac{\del u}{\del \tau} - X_H(u(\tau,t))\right) +\left(\frac{\del u}{\del t}\right)\right)
$$
We recall that $J(\xi) \subset \xi$  for any $\lambda$-adapted CR-almost complex structure $J$,
and
$$
d(\psi_H^t)^{-1}(\xi) = \xi
$$
since $(\psi_H^t)^{-1} $ is a contactomorphism. Therefore we obtain
$$
\delbar^\pi \overline u(\del_\tau) = 0 \Longleftrightarrow (d\pi - X_H \otimes dt)^{\pi(0,1)}(\del_\tau)(\del_\tau) = 0
$$
since we have
$$
2 (d\pi - X_H \otimes dt)^{\pi(0,1)}(\del_\tau) =
\left(
\left(\frac{\del u}{\del \tau} - X_H(u(\tau,t))\right) +\left(\frac{\del u}{\del t}\right)\right)^\pi.
$$

Next we evaluate the form $d((\overline u^*\lambda \circ j))$ against the pair $(\del_\tau,\del_t)$ and obtain
$$
d(\overline u^*\lambda \circ j)(\del_\tau,\del_t) =
- \frac{\del}{\del \tau}\left(\lambda\left(\frac{\del \overline u}{\del \tau}\right)\right)
- \frac{\del}{\del t}\left(\lambda\left(\frac{\del \overline u}{\del t}\right)\right).
$$
Using \eqref{eq:dbarudtau} and \eqref{eq:dbarudt}, we obtain
\beastar
&{}&  - \frac{\del}{\del \tau}\left(\lambda\left(\frac{\del \overline u}{\del \tau}\right)\right)
- \frac{\del}{\del t}\left( \lambda\left(\frac{\del \overline u}{\del t}\right)\right) \\
& = & - \frac{\del}{\del \tau}\left( \lambda\left(d(\psi_H^t)^{-1}
\left(\frac{\del u}{\del \tau}\right)\right)\right)
- \frac{\del}{\del \tau}\left( \lambda\left(d(\psi_H^t)^{-1}
\left(\frac{\del u}{\del t} - X_H(u(\tau,t))\right) \right)\right)\\
& = & - \frac{\del}{\del \tau}\left(\lambda\left(d(\psi_H^t)^{-1}
\frac{\del u}{\del \tau}\right)\right) - \frac{\del}{\del t}\left( \lambda\left(d(\psi_H^t)^{-1}
\frac{\del u}{\del t} + H(u(\tau,t))\right)\right)\\
& = &  - \frac{\del}{\del \tau}\left(\lambda\left(\frac{\del u}{\del \tau}\right)\right) - \frac{\del}{\del t}\left(\lambda\left(\frac{\del u}{\del t} + H(u(\tau,t))\right)\right).
\eeastar
The last term can be rewritten as
$$
2d\left((u^*\lambda + H\, dt) \circ j\right)(\del_\tau,\del_t).
$$
Therefore we have proved
$$
d(\overline u^*\lambda \circ j) = 0 \Longleftrightarrow d(e^{g_{H, u}}(u^*\lambda + H\, dt) \circ j) = 0.
$$
Finally it can be easily check the equivalence of given boundary conditions which we omit.
This finishes the proof.
\end{proof}

\section{Proofs of Lemma \ref{lem:A-B}}
\label{sec:A-B}

In this appendix, we prove  Lemma \ref{lem:A-B} following the scheme
used in the proof of \cite[Lemma 4.5]{oh-wang:connection}. In fact the proof thereof uses the holomorphic
property of $\del^\pi w$ in \cite[Lemma 4.5]{oh-wang:connection} which is replaced by $\del_H^\pi u$ in our
 current situation.

Recall the notation
$$
d^\pi_H u = d^\pi u - X_H^\pi(u) \otimes \gamma, \, \quad u^*\lambda_H = u^*\lambda + u^*H \, \gamma
$$
and $\del^\pi_H u$ (resp. $\delbar^\pi_H u$) is the complex linear (resp. anti-complex linear) part
of $d^\pi_H u$.

For the equality \eqref{eq:A}, using the identities $\delta^{\nabla^\pi} = - *d^{\nabla^\pi}*$
for two-forms and $*\alpha = -\alpha \circ j$ for general one-form $\alpha$, we rewrite
\beastar
- \delta^{\nabla^\pi}[(u^*\lambda \wedge (\CL_{X_\lambda}J)J \del_H^\pi u]=
 *d^{\nabla^\pi}*[(\CL_{X_\lambda}J J)\del_H^\pi u\wedge u^*\lambda],
\eeastar
and then apply the definition of the Hodge $*$
to the expression
$$
*[(\CL_{X_\lambda}J J)\del_H^\pi u\wedge (*u^*\lambda)],
$$
and get
\beastar
&&
- \delta^{\nabla^\pi}[(u^*\lambda \wedge (\CL_{X_\lambda}J)J \del_H^\pi u]\\
&=&*d^{\nabla^\pi}\langle (\CL_{X_\lambda}J J)\del_H^\pi u, u^*\lambda\rangle\\
&=&-*\langle (\nabla^\pi(\CL_{X_\lambda}J J ))\del_H^\pi u, u^*\lambda\rangle
-*\langle (\CL_{X_\lambda}J J)\nabla^\pi\del_H^\pi u, u^*\lambda\rangle\\
&{}& -*\langle (\CL_{X_\lambda}J J)\del_H^\pi u, \nabla u^*\lambda\rangle.
\eeastar
This finishes the proof of \eqref{eq:A}. The proof of \eqref{eq:B} is similar and so omitted.

For the equality \eqref{eq:C}, we compute
\bea\label{eq:delta-nabla-Tpi}
&{}&
\delta^{\nabla^\pi}[T^\pi(X_H^\pi,\gamma \wedge d_H^\pi u)] \nonumber\\
& = &-*d^{\nabla^\pi} * T^\pi(X_H^\pi,\gamma \wedge d_H^\pi u) =
 -*d^{\nabla^\pi} T^\pi(X_H^\pi,*(\gamma \wedge d_H^\pi u)) \nonumber\\
& = & *d^{\nabla^\pi} [T^\pi(X_H^\pi, \langle d_H^\pi u, \gamma \circ j\rangle)] \nonumber\\
& = &  *(\nabla^\pi T^\pi)(X_H^\pi, \langle d_H^\pi u,\gamma \circ j \rangle)
+ T^\pi(*\nabla^\pi X_H^\pi, \langle d_H^\pi u,\gamma \circ j \rangle) \nonumber\\
&{}& + * T^\pi(X_H^\pi, \langle d^{\nabla^\pi}d_H^\pi u,\gamma \circ j\rangle )
 - * T^\pi(X_H^\pi, \langle d_H^\pi u, d(\gamma \circ j) \rangle.
\eea
This finishes the proof.

\bibliographystyle{amsalpha}

\bibliography{biblio}

\end{document}